\documentclass[12pt,leqno]{amsart}
\usepackage{amssymb,amsmath,amsthm,bm}
\usepackage{graphicx}
\usepackage[dvipsnames]{xcolor}
\usepackage{color}
\usepackage[textsize=tiny]{todonotes}
\usepackage[normalem]{ulem}
\usepackage[bookmarksopen,bookmarksdepth=3,colorlinks,citecolor=red,pagebackref,hypertexnames=true]{hyperref}
\usepackage[msc-links,nobysame,non-sorted-cites, initials]{amsrefs}
\usepackage[inline]{enumitem}
\usepackage{mathtools}
\usepackage{yfonts}
\mathtoolsset{showonlyrefs}

\definecolor{indigo}{rgb}{0.4,0,0.9}
\usepackage{lipsum}
\usepackage{mathrsfs}


\DeclareFontFamily{U}{rsfs}{\skewchar\font127 }
\DeclareFontShape{U}{rsfs}{m}{n}{%
	<-6.5> rsfs5
	<6.5-8> rsfs7
	<8-> rsfs10
}{}
\newcommand{\interior}[1]{%
  {\kern0pt#1}^{\mathrm{o}}%
}

\usepackage[T1]{fontenc}  
\usepackage[lining]{libertine}
\usepackage{cabin}
\usepackage[libertine]{newtxmath}

\makeatletter
\newcommand*\bigcdot{\mathpalette\bigcdot@{.5}}
\newcommand*\bigcdot@[2]{\mathbin{\vcenter{\hbox{\scalebox{#2}{$\m@th#1\bullet$}}}}}
\makeatother

\def\N {{\mathbb{N}}}
\def\R {{\mathbb{R}}}

\def\Z {{\mathbb{Z}}}
\def\DD {{\mathcal{D}}}
\def\dist {{\mathrm{dist}}}

\def\M  {\mathrm {M}}
\def\d {\mathrm{d}}
\def\e {\mathrm{e}}
\def\eps {\varepsilon}
\def\ZZ{\mathcal{Z}}

\newcommand{\Mod}[1]{\ (\mathrm{mod}\ #1)}
\def\1 {{\mbox{\boldmath 1}}}

\DeclareMathOperator{\supp}{supp}

\def\ind{\cic{1}}
\newcommand{\cic}{\bm}

\def\XXint#1#2#3{{\setbox0=\hbox{$#1{#2#3}{\int}$}
		\vcenter{\hbox{$#2#3$}}\kern-.5\wd0}}

\def \no#1#2#3 {{\bf #1} (#3), #2.}
\def \eds#1#2#3 {#1, #2, #3.}
\newcounter{counter}

\numberwithin{equation}{section}
\numberwithin{counter2}{section}
\newtheorem{proposition}[subsection]{Proposition}
\newtheorem{theorem}[counter]{Theorem}
\newcounter{withinsub}
\numberwithin{withinsub}{subsection}
\newtheorem{corollary}{Corollary}
\newtheorem{lemma}[withinsub]{Lemma}

\theoremstyle{definition}
\newtheorem{definition}[subsection]{Definition}
 
\newtheorem*{remark*}{Remark}
\newtheorem*{warn*}{A word of warning}

\newtheorem{remark}[withinsub]{Remark} 
 \theoremstyle{plain}

\numberwithin{corollary}{counter}

\makeatletter
\DeclareRobustCommand\widecheck[1]{{\mathpalette\@widecheck{#1}}}
\def\@widecheck#1#2{%
    \setbox\z@\hbox{\m@th$#1#2$}%
    \setbox\tw@\hbox{\m@th$#1%
       \widehat{%
          \vrule\@width\z@\@height\ht\z@
          \vrule\@height\z@\@width\wd\z@}$}%
    \dp\tw@-\ht\z@
    \@tempdima\ht\z@ \advance\@tempdima2\ht\tw@ \divide\@tempdima\thr@@
    \setbox\tw@\hbox{%
       \raise\@tempdima\hbox{\scalebox{1}[-1]{\lower\@tempdima\box
\tw@}}}%
    {\ooalign{\box\tw@ \cr \box\z@}}}
\makeatother

\numberwithin{figure}{section}

\author[O.\ Bakas]{Odysseas Bakas}
\address[O.\ Bakas]{Department of Mathematics, University of Patras \\ \newline \indent 26504 Patras, Greece}
\email{\href{mailto:obakas@upatras.gr}{\textnormal{obakas@upatras.gr}}}
\author[V. Ciccone]{Valentina Ciccone}
\address[V.\ Ciccone]{Hausdorff Center for Mathematics, Universit\"at Bonn \\ \newline \indent Endenicher Allee 60, 53115 Bonn, Germany }
\email{\href{mailto:ciccone@math.uni-bonn.de}{\textnormal{ciccone@math.uni-bonn.de}}}
\author[F.\ Di Plinio]{Francesco Di Plinio} 
\address[F.\ Di Plinio]{Dipartimento di Matematica e Applicazioni, Universit\`a di Napoli \\ \newline \indent Via Cintia, Monte S.\ Angelo 80126 Napoli, Italy}
\email{\href{mailto:francesco.diplinio@unina.it}{\textnormal{francesco.diplinio@unina.it}}}
\author[M. Fraccaroli]{Marco Fraccaroli} 
\address[M.\ Fraccaroli]{BCAM - Basque Center for Applied Mathematics \\ \newline \indent Alameda de Mazarredo 14, E48009 Bilbao, Basque Country, Spain.}
\email{\href{mailto:mfraccaroli@bcamath.org}{\textnormal{mfraccaroli@bcamath.org }}}
\author[I. Parissis]{Ioannis Parissis}
\address[I.\ Parissis]{Departamento de Matem\'aticas, Universidad del Pa\'is Vasco, Aptdo. 644, 48080 Bilbao, Spain and Ikerbasque, Basque Foundation for Science, Bilbao, Spain}
\email{\href{mailto:ioannis.parissis@ehu.eus}{\textnormal{ioannis.parissis@ehu.eus}}}
\author[M. Vitturi]{Marco Vitturi}\address[M. Vitturi]{School of Mathematical Sciences, University College Cork \\ \newline \indent Western Gateway Building, Western Road, Cork, Ireland}
\email{\href{mailto:marco.vitturi@ucc.ie}{\textnormal{marco.vitturi@ucc.ie}}}  \date{\today}
\thanks{{O.\ Bakas is partially supported by grant PID2021-122156NB-I00 funded by MICIU/AEI/10.13039/501100011033 and FEDER, UE and by the funding programme ``MEDICUS'' of the University of Patras. }}

\thanks{V.\ Ciccone is supported by the Deutsche Forschungsgemeinschaft (DFG, German Research Foundation) under Germany's Excellence Strategy -- EXC2047/1   390685813 as well as SFB 1060} 

\thanks{F.\ Di Plinio is partially supported by the FRA 2022 Program of University of Napoli Federico II, project ReSinAPAS -
Regularity and Singularity in Analysis, PDEs, and Applied Sciences. }

\thanks{M.\ Fraccaroli is supported by the Basque Government through the BERC 2022-2025 program and by the Ministry of Science and Innovation: BCAM Severo Ochoa accreditation CEX2021-001142-S / MICIN / AEI / 10.13039/501100011033. }

\thanks{I. Parissis is partially supported by grant PID2021-122156NB-I00 funded by MICIU/AEI/10.13039/501100011033 and FEDER, UE, grant IT1615-22 of the Basque Government and IKERBASQUE}

\begin{document}

%%%%%%%%%%%%%%%%%%%%%%%%%%%%%% TITLE TITLE TITLE
\title[Singular multipliers on multiscale Zygmund sets]{Singular multipliers on multiscale Zygmund sets}
%%%%%%%%%%%%%%%%%%%%%%%%%%%%%% TITLE TITLE TITLE

%%%%%%%%%%%%%%%%%%%%%%%%%%%%%% ABSTRACT ABSTRACT ABSTRACT
\begin{abstract}  
Given an Orlicz space  
 $ L^2 \subseteq X \subseteq L^1$ on $[0,1]$, with submultiplicative Young function ${\mathrm{Y}_X}$, we fully characterize the closed null sets $\Xi$ of the real line with the property that H\"ormander-Mihlin or  Marcinkiewicz multiplier operators $\mathrm{T}_m$ with singularities on $\Xi$  obey weak-type endpoint modular bounds on $X$ of the type
	\[
	\left|\left\{x\in \mathbb R : |\mathrm{T}_m f(x)| >\lambda\right\}\right| \leq C \int_{\mathbb R} \mathrm{Y}_X \left(\frac{|f|}{\lambda}\right), \qquad \forall \lambda>0.
	\]
These sets $\Xi$ are exactly those enjoying a scale invariant version of   Zygmund's $(L\sqrt{\log L},{L^2})$ improving inequality with $X$ in place of the former space, which is termed \emph{multiscale Zygmund property}. Our methods actually yield sparse and quantitative weighted estimates for the Fourier multipliers $\mathrm{T}_m$ and for the corresponding square functions.

In particular, our framework covers the case of singular sets $\Xi$ of finite lacunary order and thus leads to modular   and  quantitative weighted versions of the classical endpoint theorems of Tao and Wright for Marcinkiewicz multipliers. Moreover, we obtain a pointwise sparse bound for the Marcinkiewicz square function answering a recent conjecture of Lerner. 
On the other hand, examples of  non-lacunary  sets enjoying the multiscale Zygmund property for each $X=L^p$, $1<p\leq 2$ are also covered.

The main new ingredient in the proofs is a multi-frequency, multi-scale projection lemma based on Gabor expansion, and possessing   independent interest. 
\end{abstract}	
%%%%%%%%%%%%%%%%%%%%%%%%%%%%%% ABSTRACT ABSTRACT ABSTRACT
\maketitle

%%%%%%%%%%%%%%%%%%%%%%%%%%%%%% SECTION SECTION SECTION
\section{Introduction and main results} \label{SIntro} Let $\Xi\subset \mathbb R$ be a closed null set and {$\Omega_\Xi $} stand for the %denumerable 
countable collection of connected components of $O_\Xi\coloneqq \R \setminus \Xi$.  Consider hereafter the class of Fourier multipliers
\begin{equation}\label{e:multip}
 \mathrm{T}_m f(x) \coloneqq \int_{\mathbb R} m(\xi) \widehat{f}(\xi) \e^{-2\pi i x \xi}\, \d \xi, \qquad x\in \mathbb R,
\end{equation}
with \emph{singular set} $\Xi$, namely whose symbol $m$  satisfies the $M$-th order, for some fixed, suitably large  integer $M$, H\"ormander-Mihlin type condition
\begin{equation}\label{e:hmintro}
	\sup_{0\leq j \leq M}  \sup_{\xi \in O_\Xi} \mathrm{dist}(\xi, \Xi)^{j} \left| m^{(j)} (\xi)\right| \eqqcolon \|m\|_{\mathrm{HM}(\Xi)} <\infty.
\end{equation}
The class of multipliers satisfying \eqref{e:hmintro} will be referred to as  \emph{H\"ormander-Mihlin multipliers with respect to $\Xi$}.  
  
This paper focuses on  the endpoint and localized behavior of multipliers  as in \eqref{e:hmintro} and of their related square functions, whose singular set  $\Xi$ exhibits a suitable multiscale version of the Zygmund property \eqref{e:zygclass} below. Hereby, we refer to the  classical inequality of Zygmund, illustrating how  the  approximate independence of the lacunary characters $  \e^{2\pi i 2^k \cdot}$ leads to exponential square integrability,   written in the adjoint form as
\begin{equation} \label{e:zygclass}
	 \left|\left\langle\sum_{k\in \mathbb N  } a_k \exp(2\pi i2^k \cdot), f\right\rangle \right|  \leq C
\left( \sum_{k\in \mathbb N} |a_k|^2 \right)^{ 	 \frac12 } 
\left\|f\right\|_{L\sqrt{\log L}(0,1)}.
\end{equation} 
A more specific description of the problem at hand is  given through the next three definitions.

%%%%%%%%%%%%%%%%%%%%%%%%%%%%%% DEFINITION DEFINITION DEFINITION
\begin{definition}[Orlicz spaces, $B_p$ property, modular estimates] \label{d:locorl}
Throughout the article, $X$ stands for  the Orlicz  space of measurable functions on $[0,1]$ endowed with the Luxemburg norm
\begin{equation}\label{e:lux}
	\| f \|_{X} \coloneqq \inf\left\{t>0:\, \int _{[0,1]}\mathrm{Y}_X\left(\frac{|f(x)|}{t}\right) \, \d x \leq 1\right\}
\end{equation}
induced by a  Young function $\mathrm{Y}_X$. This means that   ${\mathrm{Y}_X}:[0,\infty)\to [0,\infty)$ is continuous, convex and strictly increasing, $\mathrm{Y}_X(0)=0$ and either  $\lim_{t\to\infty} Y(t)/t=\infty$ or ${\mathrm{Y}_X}(t)=t$ for all $t\in [0,\infty)$. The latter case  allows us to consider $X=L^1$.  In particular, $X$ is a Banach function space on $[0,1]$ and  the inclusion $  X\subseteq L^1(0,1)$ holds.  A typical example of Young function that the reader should keep in mind is of the form $\mathrm{Y}_{p,s}(t)\coloneqq t^p \log^s(e+t)$ with $p\in[1,\infty)$ and $s\in[0,\infty)$. 

The following structural assumption appears in a few of our corollaries. Say that $X$ \emph{has the $B_p$ property} for some $1<p<\infty$ if 
\begin{itemize}
\item[i.] $\mathrm{Y}_X$ is submultiplicative, namely $\mathrm{Y}_X(st) \leq  C \mathrm{Y}_X(s) \mathrm{Y}_X(t)$ uniformly over $s,t>0$;
\item[ii.] there holds
$B_{p}(X)\coloneqq \displaystyle \left(\int_0^1 s^{p-1}\mathrm{Y}_X(s^{-1}) \d s \right)^{  \frac1p }  <\infty.$ 
\end{itemize}
\end{definition}
%%%%%%%%%%%%%%%%%%%%%%%%%%%%%% DEFINITION DEFINITION DEFINITION

Condition ii.\ is easily motivated as being necessary and sufficient for the $L^p$-boundedness of the Orlicz maximal operator $\mathrm{M}_X$ defined in \eqref{e:Orl} below. This is due to C.\ P\'erez \cite{Perez95}, see also \cite{WilsonBook}. 

Modular estimates, formally defined hereafter, are the Orlicz space substitute of weak-type bounds. Let $\mathrm{T}$ be a quasi-sublinear operator sending  the class $L^\infty_0(\R)$ of bounded, compactly supported functions into measurable functions on $\R$. If $w$ is a weight on $\R$,  say that $T$ \emph{has the $X$-modular estimate with weight $w$} if
there exists $C>0$ with the property that 
\begin{equation} \label{e:modest}
w\left( \big\{ x\in  \mathbb R: \,|\mathrm{T} f(x)|>\lambda\big\}\right) \leq C   \int_{\R} \mathrm{Y}_X \left(\frac{|f(x)|}{\lambda} \right) \, w(x) \d x, 
\end{equation}
for all $f\in L^\infty_0(\R)$ and $\lambda>0$, and denote by $\left[\mathrm{ T}\right]_{X,w}$ the least such constant $C$; we omit $w$ from the subscript when $w=1$. As anticipated,   modular estimates are relevant e.g.\ because they imply \cite{BCPV} local weak type inequalities such as  
\[
\left\| \mathrm{ T}: X\mapsto L^{1,\infty}(0,1)\right\| \lesssim_{\mathrm{Y}_X}  \left[\mathrm{ T}\right]_{X},
\]
under the assumption that $\mathrm{Y}_X$ is submultiplicative.

%%%%%%%%%%%%%%%%%%%%%%%%%%%%%% DEFINITION DEFINITION DEFINITION
\begin{definition}[$\ZZ(X)$  property]\label{d:LambdaX} %Let $X$ be a local Orlicz space on $(0,1)$, see again  Def.~\ref{d:locorl}. 
Let $X$ be a local Orlicz space on $(0,1)$ as in Definition~\ref{d:locorl}. Say that $\mathbb K \subset\mathbb Z$ has the $\ZZ(X)$ property if there exists $C\geq 1$ such that, cf. \eqref{e:zygclass},
\begin{equation}\left|\left\langle\sum_{k\in \mathbb K  } a_k \exp(2\pi ik \cdot), f\right\rangle \right|  \leq C
\left( \sum_{k\in \mathbb K} |a_k|^2 \right)^{  \frac12 } \|f\|_{X}
 \end{equation}
uniformly over all finitely supported complex sequences $\{a_k:\, k \in \mathbb Z\}$ and $f\in X$. In that case we denote by $\ZZ(X,\mathbb K)$ the least such $C>0$ and refer to it as  the \emph{$\ZZ(X)$ constant of  $\mathbb K$}. Otherwise, simply set $\ZZ(X,\mathbb K) =\infty.$   It is rather obvious that
\begin{equation} \label{e:rathobv}
	\ZZ(X,\mathbb K) =\sup_{\mathbb K'\subset \mathbb K} \ZZ(X,\mathbb K'),
\end{equation} and we record this fact for future use.
\end{definition}
%%%%%%%%%%%%%%%%%%%%%%%%%%%%%% DEFINITION DEFINITION DEFINITION

%%%%%%%%%%%%%%%%%%%%%%%%%%%%%% REMARK REMARK REMARK
\begin{remark}
Some observations concerning Definition~\ref{d:LambdaX} %above 
are in order. To begin with,   $\ZZ(L^2,\mathbb Z)\leq 1$ trivially, which   typically leads to considering spaces $X$ with $L^2\subseteq X \subseteq L^1$. Secondly,  note that the classical Zygmund inequality \eqref{e:zygclass} is equivalent to
\begin{equation}\label{eq:classic}
\ZZ\big(	{L\sqrt{\log L}}, \{2^{k}:\,k\in \mathbb N\}\big)  < \infty,
\end{equation}
which is why we refer to the defining inequality for $\ZZ(X,\mathbb K)$ as  \emph{Zygmund property}. 
Note that \eqref{eq:classic} may be equivalently restated as
\[
\|f\|_{\exp(L^2)(0,1)} {\sim\sup_{p\geq 2}\left(p^{-\frac12}\left\|f\right\|_{p} \right)}\lesssim \|f\|_{L^2(0,1)}
\]
for all trigonometric polynomials $f$ with frequencies in $\{2^k:\, k\in\N\}$. 
Analogously, the $\mathcal Z(X)$ property  may be equivalently  rewritten in the  
adjoint form
\begin{equation}
\label{e:adjzig}
\|f\|_{X'} \leq \ZZ(X,\mathbb K) \|f\|_{L^2(0,1)}, \qquad   f(x)=\sum_{k\in \mathbb K  } a_k \e^{2\pi i k x},
\end{equation}
where $X'$ is the dual Banach space of $X$. As trigonometric polynomials \emph{a priori} belong to any   $X$ as above, the possible lack of reflexivity of $X$ is inconsequential.  In particular, the $\ZZ(L^p)$ property for $1<p<2$ is equivalent to the classical  $\Lambda(q)$ property of   $\mathbb K\subset\Z$, $q=p'$. The latter is a central property  in the study of thin sets in analysis, a theme with extensive literature; see e.g.\  \cites{bourgain,rudin} and also the discussion in \S\ref{sec:LPchar}.
\end{remark}
%%%%%%%%%%%%%%%%%%%%%%%%%%%%%% REMARK REMARK REMARK

%%%%%%%%%%%%%%%%%%%%%%%%%%%%%% DEFINITION DEFINITION DEFINITION
\begin{definition}[{Multiscale} $\ZZ(X)$ property]\label{sec:LambdaX} Let $X$ be a local Orlicz space on $(0,1)$ as in Definition~\ref{d:locorl}. Say that a singular set  $\Xi \subset \R$ has the \emph{{multiscale} $\ZZ(X)$ property}, or \emph{multiscale Zygmund property} when $X$ is generic or clear from context, if  
\[
\ZZ^\star(X,\Xi)\coloneqq 
\sup_{n\in   \mathbb Z} \ZZ\left(X,\left\lfloor 2^n \Xi \right\rfloor\right)<\infty
\] 
where $\left\lfloor \lambda \Xi \right\rfloor\coloneqq \{\left\lfloor \lambda \xi \right\rfloor:\,\xi \in \Xi\}$.
In words, the integer parts  of each dyadic rescaling of the set $\Xi$ have the $\ZZ(X)$ property uniformly in the rescaling.    Property \eqref{e:rathobv} is  inherited, so that
\begin{equation} \label{e:rathobv2}
	\ZZ^\star(X,\Xi) =\sup_{\Xi'\subset \Xi} \ZZ^\star(X,\Xi'). 
\end{equation}

 As we shall see momentarily in  \S\ref{sec:examples}, examples of infinite sets with nontrivial multiscale Zygmund properties are those enjoying limited additive structure, such as lacunary sets of finite order. At the other end of the spectrum, the set of integers and more generally	 sets containing arbitrarily long arithmetic progressions do not satisfy any non-trivial multiscale Zygmund property.
\end{definition}
%%%%%%%%%%%%%%%%%%%%%%%%%%%%%% DEFINITION DEFINITION DEFINITION

The above definitions are tied together by one of the main results of this article, which also serves as a guiding principle throughout the introduction. In words, the content of the following theorem is a quantitative version of the equivalence that $\mathrm T_m$ has the $X$-modular estimate uniformly in $m\in \mathrm{HM}(\Xi)$ if and only if $\Xi$ has the multiscale $\ZZ(X)$ property.

%%%%%%%%%%%%%%%%%%%%%%%%%%%%%% THEOREM THEOREM THEOREM
\begin{theorem} \label{t:mainchar} Suppose $X$ has the $B_p$ property for some  $1<p<\infty$  and $B_p(X)\lesssim 1$.  Then, with reference to \eqref{e:modest}, 
\begin{equation}
	\label{e:mainchar} \frac{1}{C \mathrm{Y}_X\left(\frac{1} {\ZZ^\star\left(X,\Xi \right)}\right)}\leq   \sup\left\{ \left[\mathrm{T}_m\right]_{X}: \, \left\|m\right\|_{\mathrm{HM}(\Xi)}=1\right\} \leq C \mathrm{Y}_X\left( \ZZ^\star\left(X,\Xi \right)\right).
\end{equation}
\end{theorem}
%%%%%%%%%%%%%%%%%%%%%%%%%%%%%% THEOREM THEOREM THEOREM
  The structural $B_p$ property assumption is rather tame, as it is automatically true, with $p=3$ say, whenever $\mathrm{Y}_X $ is submultiplicative and  $L^2$ continuously embeds into $X$.  The characterization of endpoint bounds for Fourier multiplier operators is a central problem within Fourier analysis and a question to this regard appears explicitly e.g.\ in \cite[pp. 521]{TW}, albeit  
tailored to the case of $L^p \log^s(L)$ scales.  Theorem \ref{t:mainchar}  provides a complete answer to such a question, under the mere assumption of submultiplicativity.

%%%%%%%%%%%%%%%%%%%%%%%%%%%%%% SECTION SECTION SECTION
\subsection{Main results} Theorem \ref{t:mainchar} actually descends from an essentially stronger inequality quantifying the sparse form behavior of $ \mathrm{HM}(\Xi)$-multipliers. More broadly, our work    yields \emph{pointwise} and \emph{bilinear form sparse bounds} for multipliers and square functions whose singular set $\Xi$ enjoys the  multiscale $\ZZ(X)$ properties. Through the rest of the introduction, the widespread local norm notation $\langle f\rangle_{X,Q}$ points to \eqref{e:Orl},  and   sparse collections  are defined in \S\ref{ss:sop}.

%%%%%%%%%%%%%%%%%%%%%%%%%%%%%% DEFINITION DEFINITION DEFINITION
\begin{definition}[Sparse norms]\label{d:locspa} Let $X_1, X_2$ be a pair of local Orlicz spaces on $(0,1)$ as in Definition~\ref{d:locorl}. Let $\mathrm{T}$ be a linear operator sending $ f \in L^\infty_0(\mathbb R)$ to $\mathrm{T}f \in \mathrm{L^1_{\mathrm{loc}}(\R^d)}$. Say that $\mathrm{T}$ has the $(X_1,X_2)$-\emph{sparse bound} if there exists $C>0$ such that the inequality
\[
\left| \left \langle  \mathrm{T} f_1,f_2\right \rangle \right| \leq C \sup_{\mathcal S  } \sum_{Q\in \mathcal S} |Q| \langle f_1 \rangle_{X_1,Q} \langle f_2 \rangle_{X_2,Q}
\]
holds for all pairs $f_1,f_2\in L^\infty_0(\R)$, where the supremum is being taken over all $\eta$-sparse collections $\mathcal S$ of intervals of the real line for some   fixed parameter $\eta\in (0,1)$. The least such constant $C$ is indicated with $\|\mathrm{T}\|_{X_1,X_2}$ and  termed the $(X_1,X_2)$-\emph{sparse norm} of $\mathrm{T}$.   The dependence on $\eta$ in the definitions above is suppressed,  as the   minimum value of $\eta$ remains fixed throughout this paper.
\end{definition} 
%%%%%%%%%%%%%%%%%%%%%%%%%%%%%% DEFINITION DEFINITION DEFINITION

Originally motivated by  the precise, and sometimes  {sharp} quantification of  the weighted Lebesgue behavior of Calder\'on-Zygmund singular integrals, the modern usage of sparse operators dates back to the earlier works of Lerner \cites{LN2011,LN2013}. Sparse form bounds akin to those  {in} Definition \ref{d:locspa} have since  appeared in the pursuit of weighted and local estimates within and beyond Calder\'on-Zygmund theory, in a subsequent flurry of activity, see e.g.\ \cites{BCL,BFP,CCDPO17,CDPPV22,CDPO,LSph}. The sparse bounds of this paper will be used  to deduce the weighted and  optimal endpoint behavior of several operators covered by our general formalism. In the unweighted case, and for the special case that the singular set $\Xi$ is a lacunary set of finite order, these endpoint results first appeared  \cite{TW} for first order lacunary sets, and in \cite{BCPV} for lacunary sets of general order.

The first couple of results deals with two types of  square functions closely related to the multipliers of \eqref{e:multip}-\eqref{e:hmintro}. The first is the family of   square functions
\begin{equation} \label{e:Hxiprim}
\mathrm{H}_{\Xi,m} f\coloneqq\left(\sum_{\omega \in \Omega_\Xi} | \mathrm{T}_{m\mathbf{1}_\omega }f|^2\right)^{	 \frac12 }
\end{equation}
with  $m\in \mathrm{HM}(\Xi)$, cf.\ \eqref{e:hmintro}. When $m=\mathbf{1}_\R $, in analogy with the well studied case $\Xi=\{2^k : k \in \mathbb Z\}$, \eqref{e:Hxiprim} is  termed the \emph{$\Xi$-Marcinkiewicz square function} associated to the rough frequency projections 
$\mathrm{H}_\omega  \coloneqq  \mathrm{T}_{\mathbf{1}_\omega}  $, $\omega \in \Omega_\Xi$, and we reserve the notation  $\mathrm{H}_{\Xi}$ for this special case. 

%%%%%%%%%%%%%%%%%%%%%%%%%%%%%% THEOREM THEOREM THEOREM
\begin{theorem}\label{thm:roughsf} For each $f\in L^\infty_0(\R)$ there exists a sparse collection $\mathcal S$ such that 	
\begin{equation} \label{e:Hxi}
\mathrm{H}_{\Xi,m} f  \lesssim \ZZ^\star (X,\Xi) \|m\|_{\mathrm{HM}(\Xi)}\sum_{Q\in \mathcal S} \langle f\rangle_{X,Q} \cic{1}_Q
\end{equation} pointwise almost everywhere on $\R$.
The implied constant is absolute. \emph{A fortiori,}
\[
\left\| \mathrm{H}_{\Xi,m}\right\|_{X,L^1} \lesssim \ZZ^\star (X,\Xi) \|m\|_{\mathrm{HM}(\Xi)} .
\]
\end{theorem}
%%%%%%%%%%%%%%%%%%%%%%%%%%%%%% THEOREM THEOREM THEOREM

Theorem~\ref{thm:roughsf} entails  an equivalence between the $\ZZ^\star(X,\Xi)$ constant and the best constant in global modular inequality for the square functions $\mathrm{H}_{m,\Xi}$. 
%%%%%%%%%%%%%%%%%%%%%%%%%%%%%% COROLLARY COROLLARY COROLLARY
\begin{corollary}\label{c:sample}
Suppose $X$ has the $B_p$ property for some  $1<p<\infty$ and $B_p(X)\lesssim 1$.  Then, with reference to \eqref{e:modest}, 
\begin{equation}\label{e:maincharsf} 
\frac{1}{C \mathrm{Y}_X\left(\frac{1} {\ZZ^\star\left(X,\Xi \right)}\right)} \leq    \sup_{\Xi'\subset \Xi} \sup_{\|m\|_{\mathrm{HM}(\Xi')} \leq 1}\left[\mathrm{H}_{\Xi',m}\right]_{X} \leq C\mathrm{Y}_X\left(  \ZZ^\star\left(X,\Xi \right)\right). 
\end{equation}
Furthermore, there exists a positive increasing function $\mathcal Q$  such that 
\begin{equation}\label{e:wcsample} 
\sup_{\|m\|_{\mathrm{HM}(\Xi)}\leq 1} \left[\mathrm{H}_{\Xi,m}\right]_{X,w} \lesssim \mathcal Q  ([w]_{A_1})  \mathrm{Y}_X\left(  \ZZ^\star\left(X,\Xi \right)\right)
\end{equation}
uniformly over all weights $w$.
\end{corollary}
%%%%%%%%%%%%%%%%%%%%%%%%%%%%%% COROLLARY COROLLARY COROLLARY

Estimate \eqref{e:wcsample} is deduced immediately from the theorem through an application of Proposition \ref{p:weighmod}, which is stated and proved in Appendix \ref{s:app}. The right side bound in \eqref{e:maincharsf}, by virtue of \eqref{e:rathobv2}, is just a particular case of \eqref{e:wcsample}, while the leftmost almost inequality is proved in Section \ref{proof_equiv_Zyg_prop}. The leftward estimate of \eqref{e:maincharsf} also shows that the $(X,L^1)$ sparse form of Theorem~\ref{thm:roughsf} is best possible, in the sense that $X$ may not be replaced with an Orlicz space $X_1 \supsetneq X$ for which $\mathcal Z^\star(X_1,\Xi)=\infty.$ 

If $\omega$ is a bounded interval, a smooth analogue of $\mathrm{H}_\omega$ may be defined {by} using smooth frequency projections from the class
\[
\Phi_\omega  \coloneqq\left\{\varphi\in\mathcal S(\R):\;\mathrm{supp}\, \varphi \subset \omega,\; \sup_{0\leq \alpha\leq D}\ell_\omega ^\alpha\left\|\varphi^{(\alpha)} \right\|_{\infty}\leq 1 \right\}
\]
consisting of $D$-smooth functions supported on $\omega$ and  $L^\infty$-normalized. Throughout the paper, the smoothness parameter $D$ will in general be chosen to be as large as needed and will be omitted from the notation.  One then defines the intrinsic smooth frequency projections on $\omega$ as
\[
\mathrm{G}_\omega f(x) = \sup_{\varphi \in \Phi_\omega   } \left| \mathrm{T}_{\varphi } f (x) \right|, \qquad x\in \mathbb R,
\]
and introduces  the $\Xi$-\emph{Littlewood-Paley square function} by %analogue of
\begin{equation} \label{e:Gxidef}
\mathrm{G}_{\Xi} f\coloneqq
\left(\sum_{\omega \in   \mathbf{w}_{\mathcal D}(O_\Xi) } | \mathrm{G}_\omega f|^2\right)^{	 \frac12 } ,
\end{equation}
where $\mathbf{w}_{\mathcal D}(O)$ stands for the dyadic Whitney decomposition of the open set $O\subset \mathbb R$ as detailed in \S\ref{ss:disc}. For example, if $\Xi=\{0\}$ then $\mathbf{w}_{\mathcal D}(O)$ is the standard collection of Littlewood-Paley intervals and $\mathrm{G} _\Xi$ is the intrinsic version of the smooth Littlewood-Paley square function, considered e.g.\ by Wilson \cite{Wint}.

%%%%%%%%%%%%%%%%%%%%%%%%%%%%%% THEOREM THEOREM THEOREM
\begin{theorem} \label{thm:smoothsf}    For each $f\in L^\infty_0(\R)$ there exists a sparse collection $\mathcal S$ such that 
\begin{equation} \label{e:Gxi}
\mathrm{G}_{\Xi} f \lesssim 
\ZZ^\star (X,\Xi)\left(\sum_{Q\in \mathcal S} \langle f\rangle_{X,Q}^2 \cic{1}_Q \right)^{ 	 \frac12 } 
\end{equation}
pointwise almost everywhere on $\R$. The implied constant is absolute.
\end{theorem}
%%%%%%%%%%%%%%%%%%%%%%%%%%%%%% THEOREM THEOREM THEOREM

%%%%%%%%%%%%%%%%%%%%%%%%%%%%%% COROLLARY COROLLARY COROLLARY
\begin{corollary}\label{c:sample2} Suppose $X$ has the $B_p$ property for some  $1<p<\infty$ and $B_p(X)\lesssim 1$. Then, with reference to \eqref{e:modest}, 
\begin{equation}\label{e:maincharssf}
\frac{1}{C \mathrm{Y}_X\left(\frac{1} {\ZZ^\star\left(X,\Xi \right)}\right)} \leq \sup_{\Xi'\subset \Xi}\left[\mathrm{G}_{\Xi'}\right]_{X} \leq C\mathrm{Y}_X\left(  \ZZ^\star\left(X,\Xi \right)\right). 
\end{equation}
\end{corollary}
%%%%%%%%%%%%%%%%%%%%%%%%%%%%%% COROLLARY COROLLARY COROLLARY

The left side bound in \eqref{e:maincharssf} is also proved in Section \ref{proof_equiv_Zyg_prop} and provides, along with Theorem \ref{t:mainchar} and Corollary \ref{c:sample}, another characterization of the multiscale Zygmund property. The estimate in the conclusion of Theorem~\ref{thm:smoothsf} is stronger than the corresponding one of Theorem~\ref{thm:roughsf}, whence $\mathrm{G}_{{\Xi}}$ also satisfies the weighted modular inequalities of Corollary~\ref{c:sample}, and in particular the right side estimate in \eqref{e:maincharssf}. However, stronger quantitative weighted estimates may be deduced from the quadratic sparse domination of the conclusion of  Theorem~\ref{thm:smoothsf}; see Corollary \ref{cor:ssfw} {in} Section \ref{s:weights}.  As it was the case in Theorem~\ref{thm:roughsf}, the sparse domination of Theorem~\ref{thm:smoothsf} is best possible. Firstly, the space $X$ cannot be replaced by any Orlicz space $X_1\supsetneq X$ for which $\mathcal Z^\star(X_1,\Xi)=\infty$. Furthermore, the $\ell^2$-sum over the sparse collection in the right hand side cannot be replaced in general by any $\ell^q$-sum with $q>2$. To see this it is enough to note that sparse forms 
\[
f\mapsto \mathcal S_qf\coloneqq \left(\sum_{Q\in\mathcal S} \langle f\rangle_{X,Q} ^q \ind_Q\right)^{\frac1q}
\]
satisfy $\|\mathcal S_q f\|_p=O(p^{\frac 1q})$ for $p$ large. If for example $\Xi=\{0\}$, $\mathrm{G}_\Xi$ is the usual smooth Littlewood-Paley square function, and $\|\mathrm{G}_{\Xi}\|_{L^p}\gtrsim p^{1/2}$ as $p\to \infty$, hence this forces $q\leq 2$; see e.g.\ \cites{lernersf,LPR}.

For H\"ormander-Mihlin multipliers adapted to a singular set $\Xi$ as in \eqref{e:hmintro}, a sparse form domination holds instead.

%%%%%%%%%%%%%%%%%%%%%%%%%%%%%% THEOREM THEOREM THEOREM
\begin{theorem}	\label{thm:HMsparse} 
Let $\Xi $ be a singular set with the multiscale $\ZZ(X_j)$ property for $j=1,2$ and suppose $m$ is a multiplier satisfying \eqref{e:hmintro}. Then  
\[
\left\| \mathrm{T}_m  \right\|_{X_1,X_2} \lesssim \|m\|_{\mathrm{HM}(\Xi)} \prod_{j=1,2} \ZZ^\star(X_j, \Xi)   
\]
with absolute implicit constant.
\end{theorem}
%%%%%%%%%%%%%%%%%%%%%%%%%%%%%% THEOREM THEOREM THEOREM

Taking $X_2=L^q(0,1)$ with $1\leq q\leq 2$ and applying Proposition \ref{p:weighmod} yields the next corollary. In the particular case  $q=2$, the corresponding $\ZZ^\star$-constant trivializes to 1. Note that the rightmost bound in \eqref{e:mainchar} from Theorem \ref{t:mainchar} is a particular case, while the leftward bound is proved in Section \ref{proof_equiv_Zyg_prop}.

%%%%%%%%%%%%%%%%%%%%%%%%%%%%%% COROLLARY COROLLARY COROLLARY
\begin{corollary}  \label{cor:sample3} 
Let $1\leq q \leq 2 $ and suppose that $X$ has the $B_p$ property for some $1<p<q'$ and $B_p(X)\lesssim 1$. Then, there exists a positive increasing function $\mathcal Q$ such that  
\[
\begin{split}
\sup_{\|m\|_{\mathrm{HM}(\Xi)}\leq 1} \left[\mathrm{T}_m\right]_{X,w}   & \lesssim  \mathcal Q\left([w]_{A_1}, [w]_{\mathrm{RH}_{\frac{q}{q-p(q-1)}}}\right) \mathrm{Y}_X\left(\ZZ^\star (X,\Xi)\right) \mathrm{Y}_X\left(\ZZ^\star (L^q(0,1),\Xi)\right)  
\end{split}
\]
uniformly over all weights $w$.
\end{corollary}
%%%%%%%%%%%%%%%%{}%%%%%%%%%%%%%% COROLLARY COROLLARY COROLLARY
 
%%%%%%%%%%%%%%%%%%%%%%%%%%%%%% SECTION SECTION SECTION
\subsection{Marcinkiewicz-type multipliers and maximal multiscale $\ZZ(X)$ property}\label{sec:marc}  The next  definition strengthens  that  of the multiscale $\ZZ(X)$ property.

%%%%%%%%%%%%%%%%%%%%%%%%%%%%%% DEFINITION DEFINITION DEFINITION
\begin{definition}[{Maximal multiscale $\ZZ(X)$ property}] {Let $X$ be   a local Orlicz space on $(0,1)$} as in Definition~\ref{d:locorl}. Say that a pairwise disjoint collection of intervals $\Omega$ has the \emph{maximal multiscale $\ZZ(X)$ property} if 
\[
\ZZ^{\star\star}(X,\Omega)\coloneqq \sup \ZZ^\star\left(X,\left\{p_\omega :\, \omega \in \Omega\right\} \right)<\infty,
\]
the supremum being taken over all choices of points $\{p_\omega:\, \omega\in\Omega\}$ consisting of exactly one point $p_\omega\in\overline{\omega}$ per interval $\omega\in %\Omega_\Xi
{\Omega}$. If $\Xi $ is a singular set, we abuse notation to say that $\Xi$ has the maximal multiscale $\ZZ(X)$ property if the collection of complementary intervals $\Omega_\Xi$ does, and write $\ZZ^{\star\star}(X,\Xi)$ in place of  $\ZZ^{\star\star}(X,\Omega_\Xi)$.
\end{definition}
%%%%%%%%%%%%%%%%%%%%%%%%%%%%%% DEFINITION DEFINITION DEFINITION
{Typical examples of sets possessing non-trivial maximal multiscale Zygmund properties are again finite unions of lacunary sets of finite order, see \S\ref{sec:examples} for more details}.

Given a singular set $\Xi$, a bounded function $m:\R\to \mathbb C$ is said to be a  {\emph{Marcinkiewicz multiplier with singular set $\Xi$}} if $m$ has bounded variation uniformly on all $\omega\in\Omega_{{\Xi}}$, where $\Omega_{{\Xi}}$ indicates the complementary intervals of $\Xi$ as before. Denote by $ \|m\|_{\mathrm{Mar}(\Xi)}$ the corresponding multiplier norm.
These Marcinkiewicz multipliers satisfy a similar result to Theorem~\ref{thm:HMsparse}, but requiring of $\Xi$ the formally  stronger maximal multiscale  Zygmund property.  
%%%%%%%%%%%%%%%%%%%%%%%%%%%%%% THEOREM THEOREM THEOREM
\begin{theorem}	\label{thm:Marc}  Let $\Xi $ be a singular set with the {maximal multiscale} $\ZZ(X_j)$ property for $j=1,2$, and suppose $m$ is a  Marcinkiewicz multiplier with singular set $\Xi$. Then  
\[
\left\|  \mathrm{T}_m  \right\|_{X_1,X_2} \lesssim   \|m\|_{\mathrm{Mar}(\Xi)} \prod_{j=1,2} \ZZ^{\star\star}(X_j, \Xi)  
\]
with absolute implicit constant.
\end{theorem}  
%%%%%%%%%%%%%%%%%%%%%%%%%%%%%% THEOREM THEOREM THEOREM

Weighted modular inequalities for these multipliers identical to those   of Corollary~\ref{cor:sample3} hold, up to replacing $\ZZ^{\star}$ by $\ZZ^{\star\star}$ and $\mathrm{HM}(\Xi)$ by $\mathrm{Mar}(\Xi)$; we omit the formal statements.   However, in analogy with Theorem \ref{t:mainchar}, it is worthwhile to record the following characterization of the maximal multiscale Zygmund constant  $\ZZ^{\star\star}\left(X,\Xi \right)$. 

%%%%%%%%%%%%%%%%%%%%%%%%%%%%%% COROLLARY COROLLARY COROLLARY
\begin{corollary} \label{c:maincharmar} Suppose $X$ has the $B_p$ property for some  $1<p<\infty$ and $B_p(X)\lesssim 1$. Then, with reference to \eqref{e:modest}, 
\begin{equation}
	\label{e:maincharE}\frac{1}{C \mathrm{Y}_X\left(\frac{1} {\ZZ^{\star\star}\left(X,\Xi \right)}\right)} \leq    \sup\left\{ \left[\mathrm{T}_m\right]_{X}: \left\|m\right\|_{\mathrm{Mar}(\Xi)}=1\right\} \leq C \mathrm{Y}_X\left(  \ZZ^{\star\star}\left(X,\Xi \right)\right).
\end{equation}
\end{corollary}
%%%%%%%%%%%%%%%%%%%%%%%%%%%%%% COROLLARY COROLLARY COROLLARY

%%%%%%%%%%%%%%%%%%%%%%%%%%%%%% THEOREM THEOREM THEOREM

In accordance with the point of view of \cites{BCPV,CRDFS,TW}, Marcinkiewicz multipliers can be understood by embedding them into a wider class of  symbols satisfying suitable variation estimates, uniformly on each interval of $\Omega_\Xi$. These are the classes of $R_{p,q} ^\Xi$-multipliers which will be discussed in \S\ref{sec:Rpq} below. We note here that Marcinkiewicz multipliers with singular set $\Xi$ are $R_{1,1} ^\Xi$-multipliers and when $\Xi=\Lambda_1\coloneqq \{2^k:\, k\in\Z\}$ these are exactly the classical Marcinkiewicz multipliers; see \cite{CCDP2018}*{Proposition 2.9}. The class of multipliers $R_{2,2} ^{\Lambda_1}$ is the $R_2$ class of \cites{BCPV,TW}. In \S\ref{sec:Rpq} we will also present a suitable version of Theorem~\ref{thm:Marc} for $R_{p,1} ^\Xi$ multipliers  with $1 \leq p \leq 2${;} see Theorem \ref{thm:F}.

%%%%%%%%%%%%%%%%%%%%%%%%%%%%%% SECTION SECTION SECTION
\subsection{Lacunary examples and weighted bounds}\label{sec:examples}  For $\xi\in\R$, the singular set $\Xi=\{\xi\}$ of standard H\"ormander-Mihlin multipliers is a singleton, enjoying the strongest possible maximal multiscale property 
$\ZZ^{\star \star}(L^1,\{\xi\}) \leq 2.
$ In this case, Theorem \ref{thm:roughsf} and Theorem \ref{thm:smoothsf}   recover respectively the pointwise sparse bound for H\"ormander-Mihlin multipliers, and the well known square sparse bound for the Littlewood-Paley square function \cite{Br2020}. Theorem \ref{thm:HMsparse} is instead a   sparse form domination   for  H\"ormander-Mihlin multipliers, well within the results of \cites{LN2013,LacEl}. 
If $\Xi$ is a finite set, it is easy to check that
\begin{equation}
	\label{e:finiteset}
	\ZZ^{\star\star}(L^p, \Xi) \lesssim \left[\# \Xi\right]^{\frac1p-\frac12}, \qquad 1\leq p \leq 2. 
\end{equation}
{Specializing} our results to this case leads to sparse and weighted versions of the multi-frequency estimates discussed in \cite{Bmf}. A more general family of singular sets with the Zygmund property is that of lacunary sets. The next definition has countless analogues in the literature, the closest being that of \cite{SjSj}; see also \cite{PR}.

%%%%%%%%%%%%%%%%%%%%%%%%%%%%%% DEFINITION DEFINITION DEFINITION
\begin{definition}\label{def:lacun}
Let $\gamma\in(1,\infty)$. A sequence $\{\theta_k\}_{k\in\Z}$ is called $\gamma$-\emph{lacunary} if there exists $\theta\in\R$ such that $\theta_k\neq \theta$ for all $k\in\Z$ and $ \gamma(\theta_{k+1}-\theta)\leq (\theta_k-\theta) $ for all $k\in \Z$. A $\gamma$-\emph{lacunary set of order $0$} is a single point in $\R$. If $\tau\geq 1$ is a positive integer then a set $\Xi\subset\R$ will be called \emph{$\gamma$-lacunary of order $\tau$} if there exists a $\gamma$-lacunary sequence $\{\theta_k\}_{k\in\Z}$ such that, for every $	k\in\Z$, the set $\Xi\cap(\theta_{k+1},\theta_k]$ is lacunary of order $\tau-1$. In the sequel, we do not refer to $\gamma$ explicitly and simply say \emph{lacunary} of order $\tau$.
\end{definition}
%%%%%%%%%%%%%%%%%%%%%%%%%%%%%% DEFINITION DEFINITION DEFINITION

Lacunary sets $\Xi$ of order $\tau\geq 0$ enjoy the maximal multiscale Zygmund properties
\begin{align}
\label{e:zygplac} 
\ZZ^{\star\star}(L^p, \Xi) &\lesssim (p-1)^{-\frac{\tau}{2}}, \qquad 1<p \leq2,  
\\
\label{e:zygloglac} 
\ZZ^{\star\star}\left(L\left[\log L\right]^{\frac\tau 2}, \Xi\right) &\lesssim 1 ,
\end{align}
with implied constants depending on $\gamma,\tau$ only. This is trivial for $\tau=0$,  a routine check relying upon the already mentioned inequality of Zygmund for $\tau=1$, see \cite{Zyg_pap} and \cite{Zygmund_book}*{Theorem 7.6, Chapter XII}, or upon the higher order analogue due to Bonami \cite{Bonami} for  $\tau>1$; see also \cite{BCPV}*{Remark 4.3}.  We note in passing that the corresponding Young function for the local $L[\log L]^\frac{\tau}{2}$ space can be taken to be $Y_{1,\frac{\tau}{2}}(t)\coloneqq t[\log(e+t)]^{\frac{\tau}{2}}$ which satisfies the formalism of Orlicz spaces in Definition~\ref{d:locorl}; in particular $Y_{1,\frac{\tau}{2}}$ satisfies the $B_p$-condition in Definition~\ref{d:locorl} for all $p>1$.

The article \cite{LerMRL2019}*{\S5.2} conjectures sparse norm  bounds respectively for the  rough Littlewood-Paley square function $\mathrm{H}_\Xi$ of Theorem \ref{thm:roughsf} and for the corresponding  Marcinkiewicz  multipliers $\mathrm{T}_m$ of Theorem \ref{thm:Marc} when $\Xi $ is a lacunary set of order $1$. Theorem \ref{thm:roughsf} yields in this case
\begin{equation}
	\label{e:thepartrough_1}
\|\mathrm{H}_\Xi\|_{L\sqrt{\log L}, L^1} \lesssim 1, \qquad 
\|\mathrm{H}_\Xi\|_{L^p, L^1} \lesssim (p-1)^{-\frac12}, \quad 1<p\leq 2
\end{equation}
and the second estimate matches the conjectured bound in  \cite{LerMRL2019}. On the other hand,  applying Theorem \ref{thm:Marc} tells us that 
\begin{equation}
	\label{e:thepartrough_2}
\|\mathrm{T}_m\|_{L\sqrt{\log L}, L\sqrt{\log L}} \lesssim 1, \qquad 
\|\mathrm{T}_m\|_{L^p, L^p} \lesssim (p-1)^{-1}, \quad 1<p\leq 2.
\end{equation}
The second estimate is sharp as $p\to 1^+$, cf. \cite{CCDP2018}*{Prop.\ 7.1}, both showing that \cite{LerMRL2019}*{eq.\ (5.1)} is too optimistic and deducing the correct substitute. In addition, neither space in the first estimate can be improved to a space $X\supsetneq L\sqrt{\log L}$, as otherwise $T_m$ would satisfy an $X$-modular estimate of the form appearing in Corollary \ref{cor:sample3}, and in particular map $X \to L^{1,\infty}$ locally, which is known to fail whenever $X\supsetneq L\sqrt{\log L}$; see \cite{TW}.

Lerner's conjectures \cite{LerMRL2019} aimed at furthering the study of quantitative weighted norm inequalities for the operators $\mathrm{H}_\Xi, \mathrm{T}_m$ in the first order lacunary case. Our sparse estimates lead to  a few improvements of known quantifications, and extend the scope to a much wider array of singular sets. For these sets the lacunarity assumption is replaced by requiring the multiscale or maximal multiscale Zygmund properties with $X=L^p$ with a specific blow-up rate as $p\to 1^+$, covering in particular the finite order lacunary case. See Section \ref{s:weights} for statements and proofs.

%%%%%%%%%%%%%%%%%%%%%%%%%%%%%% SECTION SECTION SECTION
\subsection{Characterizing the LP-property and   non-lacunary examples}\label{sec:LPchar} Let $\Xi\subset \R$ be a closed null set and recall that $\Omega_\Xi=\{\omega\}_{\omega\in\Omega_\Xi}$ denotes the collection of complementary intervals of $\Xi$, namely $O_\Xi=\bigcup_{\omega\in\Omega_\Xi} \omega$. Let $1<p<\infty$. The set $\Xi$ has the \emph{Littlewood-Paley $p$-property}, in short $\mathrm{LP}(p)$, if there exist constants $C_p,c_p>1$, depending only on $p$, such that the following two-sided square function estimate 
\begin{equation}\label{eq:LP}
 c_p ^{-1} \|f\|_{p}\leq \left\|\mathrm{H}_{\Xi} f\right\|_p\leq C_p \|f\|_p
\end{equation}
holds, with $\mathrm{H}_\Xi$ defined as in \eqref{e:Hxiprim}. By duality,  it is clear that $\Xi$ has the $\mathrm{LP}(p)$ property if and only if it has the   $\mathrm{LP}(p')$ property.  
Moreover, say that $\Xi$ is an  \emph{$\mathrm{LP}$-set or that it has the $\mathrm{LP}$ property} if $\Xi$ has the $\mathrm{LP}(p)$-property for all $p\in(1,\infty)$. 

Similar definitions can be given in terms of H\"ormander-Mihlin multipliers or Marcinkiewicz multipliers with singular set $\Xi$. Say that a closed  null set $\Xi\subset\R$ has the property $\mathrm{HM}(p)$ if every H\"ormander-Mihlin multiplier with singular set $\Xi$ as in \eqref{e:multip} is bounded on $L^p(\R)$. Similarly, %$\Xi$ is said to have
say $\Xi$ has the property $\mathrm{Mar}(p)$ if every Marcinkiewicz multiplier with singular set $\Xi$ as in \S\ref{sec:marc} is bounded on $L^p(\R)$.

These definitions are rather classical, see \cite{SjSj} for the case of the real line and \cite{HaKl} %for
{of} the torus. We stress that the definitions above are insensitive to the value of the constants involved in $L^p(\R)$-boundedness assumptions. Furthermore, the main results in \cite{SjSj} show that the properties $\mathrm{LP}(p)$, $\mathrm{Mar}(p)$ and $\mathrm{HM}(p)$ are all equivalent for any fixed $p\in(1,\infty)$, whence  we focus on the $\mathrm{LP}(p)$ property in the next characterization.

%%%%%%%%%%%%%%%%%%%%%%%%%%%%%% THEOREM THEOREM THEOREM
\begin{theorem}\label{thm:LP} Let $\Xi\subset \mathbb R$ be a singular set and let $q\in [1,2)$. Then the following are equivalent.
\begin{itemize}
\item[1.]  $\ZZ^{\star\star}(L^{p},\Xi)<\infty$ for all $p\in(q,2)$.
\item[2.]  $\ZZ^\star(L^{p},\Xi)<\infty$ for all $p\in(q,2)$.
\item[3.] The set $\Xi$ has the \emph{$\mathrm{LP}(p)$ property} for all $p \in (q,q')$	.
\end{itemize} 
In particular, the set $\Xi$ has the $\mathrm{LP}$ property if and only if $ \ZZ^\star(L^p, \Xi)<+\infty$ for all $p\in(1,2)$.
\end{theorem}
%%%%%%%%%%%%%%%%%%%%%%%%%%%%%% THEOREM THEOREM THEOREM

Theorem \ref{thm:LP}, although interesting   on its own, helps us delineate the connections between the multiscale Zygmund  property and the  $\mathrm{LP}(p)$ property. Additionally, it leads to examples of non-lacunary sets enjoying nontrivial $\ZZ^\star(L^p)$-properties, thereby greatly extending the applicability of our multiplier theorems. This is expounded in the next series of remarks

%%%%%%%%%%%%%%%%%%%%%%%%%%%%%% REMARK REMARK REMARK
\begin{remark}[Relation with $\Lambda(q)$]As previously observed,  finiteness of $\ZZ(L^{p},\mathbb K) $ for $p\in(1,2]$ coincides with the $\Lambda(q)$-property of $\mathbb K\subset \mathbb Z$ for $q=p'\in [2,\infty)$. It follows from  the work of Pisier \cite{Pisier_78} that $\mathbb K\subset \Z$ satisfies the $\ZZ(L\sqrt{\log L})$ property if and only if it is a Sidon set; see also the monograph by Marcus and Pisier \cite{M-P_book}, where an analogous characterization is obtained in the setting of compact groups. Note that the easier of the two equivalences in the previous characterization was first proved by Rudin, \cite{rudin}*{Theorem 3.1}. Arithmetic characterizations of Sidon sets in the dual of a compact abelian group, and in particular of subsets of integers satisfying the Zygmund property, were also found  by  Pisier, see  e.g. \cite{Pisier_83} and  references therein. In \cite{Bourgain_Sidon} Bourgain, using a different approach,  obtained an additional characterization of Sidon sets in the dual of a compact abelian group and recovered the aforementioned results of Pisier. See also the  treatise by Graham and Hare \cite{G-H_book} and   references therein.

It is well known that a $\mathrm{LP}(p)$ set is necessarily a  $\Lambda(\max(p,p'))$ set; see \cite{HaKl}*{\S3}. At the same time, there exist $\Lambda(q)$-sets,  $q>2$, which are not $\mathrm{LP}(q)$-sets. Most importantly for us,  there exist sets of integers  that are $\Lambda(q)$ for all $q<\infty$ but are not $\mathrm {LP}$-sets,  and in particular there exist sets that satisfy the $\ZZ(L\sqrt{\log L})$ property but are not $\mathrm{LP}$-sets, see \cite{HaKl}*{\S4}. Combining these examples with Theorem  \ref{thm:LP}, we infer that the $\Lambda(q)$, $2<q<\infty$, property of $\mathbb K$ is   strictly weaker than  the finiteness of  $\ZZ^\star(L^{q'},\mathbb K)$. Equivalently, the $\ZZ(L^{q'})$ property is strictly weaker than the $\ZZ^\star(L^{q'})$ property.
\end{remark}
%%%%%%%%%%%%%%%%%%%%%%%%%%%%%% REMARK REMARK REMARK

%%%%%%%%%%%%%%%%%%%%%%%%%%%%%% REMARK REMARK REMARK
\begin{remark}[Non-lacunary examples of multiscale Zygmund sets] An  example of an $\mathrm{LP}$-set $\Xi$ which  may not be written as a finite union of finite-order lacunary sets has been constructed in \cites{HaKl,HaKlnew}. More precisely, the authors in \cite{HaKl} construct a certain family of sets $E_\infty\subset \mathbb Z$ which are \emph{not} finite unions of lacunary sets of finite order. In \cite{HaKlnew} the authors verify that for a suitable choice of parameters in the construction of $E_\infty$, the latter set gives rise to a partition of the integers that satisfies the $\mathrm{LP}$ property on the torus. In \cite{HaKlperm}, the authors  mention that the proof of the Littlewood-Paley property for the torus transfers to the real line. By Theorem~\ref{thm:LP}, we gather that $\ZZ^\star(L^p, E_\infty)<+\infty$ for all $p\in(1,2)$.
\end{remark}
%%%%%%%%%%%%%%%%%%%%%%%%%%%%%% REMARK REMARK REMARK
 
%%%%%%%%%%%%%%%%%%%%%%%%%%%%%% SECTION SECTION SECTION
\begin{remark}[Idempotent multipliers]Let   $\mathcal M_p(\R)$ denote the algebra of $L^p$-bounded Fourier multipliers, $1<p<\infty$.  Every measurable set $E\subset \R$ generates the idempotent multiplier $\ind_E\in \mathcal M_2(\R)$. If $\ind_E\in\mathcal M_p(\R)$ for some $p\neq 2$ then the set $E$ generates a complemented subspace of $L^p$. By a theorem of Rudin, \cite{rudin_inv}, extended by Rosenthal to non-compact groups, \cite{Ros}, all translation invariant complemented subspaces of $L^p(\R)$, $1<p<\infty$, are of the form 
\[
\left\{T_E f:\, f\in L^p(\R)\right\}\subset L^p(\R),\qquad   \mathrm{T}_E f \coloneqq \left(\ind_E \hat f\right)^\vee,
\]
for some unique measurable $E$ such that $\ind_E \in \mathcal M_p(\R)$. A theorem of Lebedev and Olevskii, \cite{LeOle}, shows that a necessary condition for $\ind_E\in \mathcal M_p(\R)$ for $p\neq 2$ is that $E$ is open up to a set of measure zero. Using the results of the previous paragraph, we can give a sufficient condition for $E$ to generate an idempotent multiplier in %$\mathcal M_p(\R)$ 
$\mathcal M_{{q}}(\R)$ via the multiscale Zygmund property.

%%%%%%%%%%%%%%%%%%%%%%%%%%%%%% COROLLARY COROLLARY COROLLARY
\begin{corollary}Let $E\subset \R$ be measurable. If $E$ %is equivalent to 
 coincides up to a set of measure zero with an open set and $\ZZ^\star(L^p,\Xi_E)<+\infty$, where $\Xi_E$ denotes the set of endpoints of the component intervals of $E$, then 
	 $\ind_E\in\mathcal M_q(\R)$ for all $q \in (p,p')$.
\end{corollary}
\end{remark}
%%%%%%%%%%%%%%%%%%%%%%%%%%%%%% COROLLARY COROLLARY COROLLARY

\subsection*{Structure of the article} Section \ref{s:not}  establishes the basic notation being used throughout the paper. It also contains the reduction of the multiplier classes $\mathrm{HM}(\Xi)$, $\mathrm{Mar}(\Xi)$, and corresponding square functions  to suitably defined phase plane model sums adapted to the singular set $\Xi$. The discretization of the class $\mathrm{Mar}(\Xi)$ is actually realized by viewing it as a special case  of the more general family $R_{p,1}^\Xi, 1\leq p\leq 2$. A sparse domination bound extending Theorem \ref{thm:Marc}  to this family may be found in Section \ref{s:not} as well, see Theorem \ref{thm:F}.  

Section \ref{s:gab} is the technical heart of the article. It contains Lemma \ref{l:proj}, where the multiscale Zygmund property of $\Xi$ is exploited to construct  an appropriate Gabor projection of a given function $f$ localized to an interval. The function $f$ is projected
on the subspace generated by wave packets with frequency localized coming from a Whitney decomposition of $O_\Xi$. Most importantly, the localized $L^2$ norm of this projection are kept under control by the local $X$-norms of $f$. In the same section, Lemma \ref{l:proj} is used to produce single scale and tail estimates for $\Xi$-adapted model sums.

Section \ref{s:sce} contains the main multiscale part of the arguments, mostly summarized by Proposition \ref{p:maintech}, whose proof relies again on the projection Lemma \ref{l:proj} together with the single scale estimates of Section \ref{s:gab}. Sections \ref{s:thmroughsf}, \ref{s:thmsmoothsf}, and \ref{s:nextprob} are devoted to the proofs of the main Theorems \ref{thm:roughsf}, \ref{thm:smoothsf},    \ref{thm:HMsparse} and \ref{thm:Marc}  respectively, and all rely on the above mentioned proposition or variants thereof, as well as on arguments typical of sparse domination. 

Section \ref{s:weights} contains a plethora of quantitative weighted norm inequalities for $\Xi$-singular multipliers under the assumption of controlled blowup of the constants, e.g.\ $\ZZ^\star(L^p, \Xi)$, as $p \to 1^+.$ Sections \ref{proof_equiv_Zyg_prop} and \ref{s:LPp} are respectively dedicated to the reverse controls in Theorems \ref{t:mainchar}, Corollaries \ref{c:sample}, \ref{c:sample2} and \ref{cor:sample3}, and to the proof of Theorem \ref{thm:LP}.

%%%%%%%%%%%%%%%%%%%%%%%%%%%%%% SECTION SECTION SECTION
\subsection*{Acknowledgments} The authors are grateful to Giacomo Ascione for an expert reading and valuable suggestions on the exposition.

%%%%%%%%%%%%%%%%%%%%%%%%%%%%%% SECTION SECTION SECTION
\section{Notation and background material} \label{s:not}
 
%%%%%%%%%%%%%%%%%%%%%%%%%%%%%% SECTION SECTION SECTION
\subsection{Recurring notation} Throughout the paper, the same convention used in the introduction is kept: $\Xi$ is a closed null set,  {$ O_\Xi\coloneqq \R \setminus \Xi$,} and {$ \Omega_\Xi$} stands for the collection of  connected components of $O_\Xi$. The Fourier transform follows the normalization
\[
\widehat f(\xi) = \int_{\R} f(x) \e^{-2\pi i  x\xi}\, \d x, \qquad \xi \in \mathbb R.
\]
The large positive constant $C$, small positive constant $c$, and those constants implied by the almost inequality signs are meant to be absolute unless otherwise specified, and may vary at each occurrence without explicit mention.

%%%%%%%%%%%%%%%%%%%%%%%%%%%%%% SECTION SECTION SECTION
\subsection{Dyadic grids, weighted spaces, maximal functions}Given an interval $I\subset \R$,  denote by $c_I,\ell_I$ the center and Euclidean length of $I$, respectively. Hereafter, $\mathcal D$ stands for a generic  dyadic system on $\mathbb R$. For instance
\[
\mathcal D= \bigcup_{n\in \mathbb Z}  \mathcal D_n=\bigcup_{n\in \mathbb Z}    \Big\{\left[k2^{-n}, (k+1) 2^{-n}\right):\, k \in \mathbb Z\Big\}.
\]
It is useful to isolate
\[
\mathcal D(I) \coloneqq \{J\in \mathcal D:\, J\subseteq I\}, \qquad  \mathcal D_n(I) \coloneqq \left\{J\in \mathcal D(I): \, \ell_{J} =2^{-n} \ell_I \right\}, \qquad n\in \mathbb N.
\]
If  $I\in \mathcal D$ and $n\in \mathbb N$ then there is exactly  one element $J$ of $\mathcal  D$ with $I \in \mathcal D_n(J)$. Denote this element by $ I^{\uparrow n}$  and refer to it as   the \emph{$n$-th dyadic parent of $I$.}  For $k\in \mathbb Z$ and $n \in \mathbb N$  denote by
\[
I^{\uparrow n,+k}\coloneqq I^{\uparrow n}+ k\ell_{I^{\uparrow n}}, \qquad I^{+k} \coloneqq I+k\ell_I,
\] 
the corresponding higher order parents and their translates.  Linear changes of coordinates are indicated by
$$ \mathrm{Tr}_a f:= f(\cdot - a), \qquad \mathrm{Dil}_b^p f:= b^{-\tfrac{1}{p}}f(b^{-1}\cdot), \qquad a\in \mathbb R, \quad b>0 .$$
The smooth replacement for indicators will be 
\[
\chi(x) \coloneqq \frac{1}{1+x^2}, \qquad \chi_I \coloneqq \mathrm{Tr}_{c_I} \mathrm{Dil}^\infty_{\ell_I} \chi = \chi\circ  {\mathrm{Sy}_I}^{-1},
\]
where $I\in \mathcal D$. Note that  $ \mathrm{Sy}_I$ is  the unique linear map such that $I$ is the image of $[0,1)$. 
Positively/negatively weighted  local  $L^p$-norms are denoted  by
\[
\langle f \rangle_{p,I,\pm}\coloneqq  \left(\frac{1}{|I|}\int_{\R} \left|f\chi_{I}^{\pm\mathsf{dec}}\right|^p \right)^{  \frac1p }  = \left\|\left[f\circ{\mathrm{Sy}_I} \right] \chi^{\pm\mathsf{dec}}\right\|_p
\]
where {$\mathsf{dec}$ is a large but bounded power whose value might change at each occurrence. The dependence on $\mathsf{dec}$ may thus be kept implicit in the notation. A word on functions $f$ with $\langle f \rangle_{p,I,-}$ under control: these are functions that are strongly localized to $I$ in the sense that they decay polynomially fast away from it. On the other hand $\langle f\rangle_{p,I,+}$ is a tailed average of $f$ on $I$, and an exact average if $f$ happens to be supported on $I$.

Let $X$ be a local Orlicz space as  in Definition~\ref{d:locorl}. Orlicz averages localized to a not necessarily dyadic interval $I$, and the corresponding maximal operator, are indicated with 
\begin{align}\label{e:Orl}
&\langle f \rangle_{X,I}\coloneqq
\left\|(f\cic{1}_I)\circ{\mathrm{Sy}_I} \right\|_X,  \qquad
\\ 
& \mathrm{M}_X f(x) \coloneqq \sup_{I\ni x  } \langle f \rangle_{X,I}, \qquad x\in \mathbb R,
\\
&\langle f \rangle_{X,I,\pm}\coloneqq
\left\|\left[f\circ{\mathrm{Sy}_I} \right] \chi^{\pm\mathsf{dec}} \right\|_X.  
\end{align}
When $I=[0,1)$, we use just $\langle f \rangle_{X}$ in place of $\langle f \rangle_{X,I}$ for the local Orlicz norms.

%%%%%%%%%%%%%%%%%%%%%%%%%%%%%% SECTION SECTION SECTION
\subsection{Sparse operators} \label{ss:sop} If $0<\eta<1$ and  $E\subset F\subset \mathbb \R$ are measurable, say that $E$ is \emph{$\eta$-major}, or simply \emph{major} in $F$, if $|E|\geq \eta|F|$. A collection $\mathcal S$ of intervals in $\mathbb R$ is said to be \emph{$\eta$-sparse}  if there exists a pairwise disjoint collection of $\eta$-major subsets
$\{E_S\subset S:S\in \mathcal S\}$. A well-known principle, see e.g.\ \cite{LN}, is that for each $0<\eta<1$, a collection $ \mathcal S\subset \mathcal D$ for some fixed dyadic grid $\mathcal D$ is $\eta$-sparse if and only if the   packing condition 
\begin{equation} 
\sum_{\substack{S\in \mathcal S\\ S %\subset 
		{\subseteq} I} } |S| \leq \eta^{-1} |I|
\end{equation}
holds  uniformly over all intervals $I\subset\mathbb R$. To each sparse collection $\mathcal S$,  local Orlicz space $X$ and exponent $0<p<\infty$,  associate the sublinear sparse operators
\begin{equation}
\label{e:thesparsedefs}
\mathcal S_{X,p} f\coloneqq  \left(\sum_{S\in \mathcal S} \langle f\rangle_{X,S}^p {\cic{1}_S} \right)^{ \frac1p }, \qquad 
\mathcal S_{X,p,\pm}  f\coloneqq  \left(\sum_{S\in \mathcal S} \langle f\rangle_{X,S,\pm}^p {\cic{1}_S} \right)^{ 	 \frac1p }.
\end{equation}
{The former appears for example in the statements of Theorems~\ref{thm:roughsf}~and~\ref{thm:smoothsf}, with $p=1$ and $p=2$ respectively.} When $X=L^{q}(\R)$ for some $0<q<\infty$, the notations in \eqref{e:thesparsedefs} are replaced by the simpler   $\mathcal S_{q,p}$ and $\mathcal S_{q,p,\pm}$ respectively. The  bilinear sparse forms
\begin{equation}
\label{e:thesparsedefs_2}
{\Lambda^{\mathcal S}_{X_1,X_2}}(f_1,f_2) \coloneqq \sum_{S\in \mathcal S}|S| \langle f_1\rangle_{X_1,S}\langle f_2\rangle_{X_2,S}, \qquad  {\Lambda^{\mathcal S}_{X_1,X_2,\pm}}(f_1,f_2) \coloneqq \sum_{S\in \mathcal S}|S| \langle f_1\rangle_{X_1,S,\pm}\langle f_2\rangle_{X_2,S,\pm},
\end{equation}
are also used, with the first one  appearing for example in the  statements of Theorems~\ref{thm:HMsparse}~and~\ref{thm:Marc} through  Definition \ref{d:locspa}.  As before, when $X_j=L^{q_j}(\R)$ for some $0<q_j<\infty$ the simpler notations $\Lambda^{\mathcal S}_{q_1,q_2}$ and $ \Lambda^{\mathcal S}_{q_1,q_2,\pm}$, respectively, are preferred.

%%%%%%%%%%%%%%%%%%%%%%%%%%%%%% SECTION SECTION SECTION
\subsection{Tiles, adapted classes}  A \emph{tile} $P=I_P\times \omega_P$ is a product of dyadic intervals with area $1$, that is \begin{equation} \label{e:alltiles}
I_P\in \mathcal D, \qquad \omega_P \in \mathcal D', \qquad \ell_{I_P} \ell_{\omega_P}=1,
\end{equation}
where $\mathcal D, \mathcal D'$ are two possibly different dyadic grids.  Overloading notation, write $\ell_P\coloneqq\ell_{I_P}$ when $P=I_P\times \omega_P$ and refer to it as the \emph{scale} of $P$. 
The notation $\mathbb P_{\mathcal D,\mathcal D'}$ refers to the collection of all tiles arising as in \eqref{e:alltiles}, and the subscript is dropped once  $\mathcal D,\mathcal D'$ are fixed and clear from context.
We use  $\mathbb Q $  to denote a generic subset of $\mathbb P= \mathbb P_{\mathcal D,\mathcal D'}$.
 In general, if $\mathcal J \subset \mathcal D', $    write
 \begin{equation} \label{e:qJ}
 \mathbb Q^{\mathcal J} =\left\{P=I_P\times \omega_P\in \mathbb Q:\, \omega_P\in \mathcal J \right\}.
 \end{equation}
Instead, when $I$ is any interval in $\mathbb R$, 
\begin{equation}
	\label{e:tilesdef}
\mathbb Q_{=}(I)  \coloneqq\left\{P=I_P\times \omega_P\in \mathbb Q:\, I_P=I\right\}, \qquad \mathbb Q(I)    \coloneqq\left\{P=I_P\times \omega_P\in \mathbb Q:\, I_P\subseteq I\right\}.
\end{equation}
 To each  tile $P$  associate the $L^1$-normalized wavelet class $\Psi_P(D)$ consisting of those $\phi\in \mathcal C^\infty(\R) $ with 
\begin{equation}
\label{e:wavsupp}  \supp\widehat \phi \subset \omega_P
, \qquad  \sup_{0 \leq j\leq D} |I_P|^{1+j} \left\| \chi_{I_P}^{-D}\big(\exp( {2\pi} i c_{\omega_P} \cdot)\phi\big)^{(j)}\right\|_\infty \leq 1,
\end{equation} 
where $D$ is a large positive integer.

%%%%%%%%%%%%%%%%%%%%%%%%%%%%%% REMARK REMARK REMARK
\begin{remark}\label{rmrk:loc}The localization trick in the form
\[
\upphi_P \in \Psi_P(D),\quad I_P \subset   I \implies  \varphi_P \coloneqq \chi_{I_{{P}}}^{-{\mathsf{dec}}} \upphi_P \in C \Psi_P(D-\mathsf{dec})
\]
for a suitable constant $C$ depending on ${\mathsf{dec}}$, will be used often, but only finitely many times within each argument. We may thus assume that every instance of $D$ in  $\Psi_P(D)$ is much larger than 
 any instance of ${\mathsf{dec}}$ appearing in the proofs, and thus drop $D$ from the notation altogether and  write $\Psi_P$ for each instance. We do not aim to optimize the number of derivatives of the wavelets used in this paper and the choices of parameters $D = 100$ and $0 \leq \mathsf{dec}\leq 50$ will suffice.
\end{remark}
%%%%%%%%%%%%%%%%%%%%%%%%%%%%%% REMARK REMARK REMARK

%%%%%%%%%%%%%%%%%%%%%%%%%%%%%% SECTION SECTION SECTION
\subsection{Almost orthogonal collections}A collection of tiles $\mathbb Q\subset \mathbb P_{\mathcal D,\mathcal D'}$ will be called \emph{almost orthogonal} if
\begin{equation}\label{e:aotiles}
P,P'\in \mathbb Q,\,\omega_P \cap \omega_{P'} \neq \varnothing \implies  \omega_{P} =\omega_{P'}.
\end{equation} 
For instance,  $\mathbb Q^{\mathcal J}$ is an almost orthogonal collection whenever the elements of $\mathcal J \subset \mathcal D'$ are pairwise disjoint.
The choice of name is motivated via the following reasoning, ultimately leading to \eqref{e:aotilescons} below. To begin with, let us introduce the operators
\begin{equation}\label{e:defmulti}
T_{\mathbb Q}f \coloneqq \sum_{P\in  \mathbb Q} |I_P| \langle f,\upphi_P \rangle \psi_P,
\end{equation}
where $\mathbb Q$ is a subset of $\mathbb P_{\mathcal D,\mathcal D'}$ and $\upphi_P,\psi_P\in \Psi_{P}$ are choices of adapted wave packets for each $P \in \mathbb P_{\mathcal D,\mathcal D'}$.  

%%%%%%%%%%%%%%%%%%%%%%%%%%%%%% LEMMA LEMMA LEMMA
\begin{lemma}\label{lem:ao} Suppose $\mathbb Q$ satisfies \eqref{e:aotiles}. Then  $\displaystyle \left\langle T_{\mathbb Q (I)} f\right\rangle_{2,I,-} \lesssim \left\langle f\right\rangle_{2,I,+} $ uniformly over $I\in \mathcal D$.
\end{lemma}
%%%%%%%%%%%%%%%%%%%%%%%%%%%%%% LEMMA LEMMA LEMMA

%%%%%%%%%%%%%%%%%%%%%%%%%%%%%% PROOF PROOF PROOF
\begin{proof} Let us denote
\[
\widetilde{T}_{\mathbb Q(I)}f \coloneqq \sum_{P\in\mathbb Q(I)}|I_P|\langle f,\widetilde{\upphi}_P\rangle\widetilde{\psi}_P,\qquad \widetilde{\psi}_P\coloneqq \chi_I ^{-{\mathsf{dec}}}\psi_P,\qquad \widetilde{\upphi}_P\coloneqq \chi_I ^{-{\mathsf{dec}}}\upphi_P.
\]
By Remark~\ref{rmrk:loc} and a standard $TT^*$ argument reliant on the almost orthogonality property \eqref{e:aotiles} as for example in \cite{APHPR}*{\S4.3}, we have
\[
\langle T_{\mathbb Q(I)}f \rangle_{2,I,-} ^2= \frac{1}{|I|}\int_{\R} \left|\widetilde{T}_{\mathbb Q(I)}(f\chi_I ^{{\mathsf{dec}}})\right|^2\lesssim\frac{1}{|I|}\int\left|f\chi_I ^{{\mathsf{dec}}}\right|^2 = \langle f \rangle_{2,I,+} ^2
\]
which is the desired estimate.
\end{proof}
%%%%%%%%%%%%%%%%%%%%%%%%%%%%%% PROOF PROOF PROOF

If  $\mathbb Q$ satisfies \eqref{e:aotiles}, an   immediate consequence of the lemma is that
\begin{equation}\label{e:aotilescons}
\frac{1}{|I|}\sum_{P \in \mathbb Q(I) } |I_P| |\langle f,\upphi_P \rangle |^2 = \frac{\left \langle T_{\mathbb Q (I)}  f,f \right\rangle}{|I|} \leq  \langle  T_{\mathbb Q (I)}  f\rangle_{2,I,-}   \langle f\rangle_{2,I,+}  \lesssim  \langle f\rangle_{2,I,+}^2
\end{equation} 
uniformly over $I\in \mathcal D$ and choices of wave packets $\upphi_P\in \Psi_P$.

%%%%%%%%%%%%%%%%%%%%%%%%%%%%%% SECTION SECTION SECTION
\subsection{Discretization of multipliers with singular set $\Xi$} \label{ss:disc} In this paragraph, we recast the well-known time-frequency discretization of Fourier multipliers of the type \eqref{e:multip}, originating e.g.\ in \cites{LT1,LT}; see also \cite{DPFR} for a list of more recent references and a proof of \eqref{e:cvxhullms} below.

Let $\mathcal D$ be a standard dyadic grid on $\R$. The \emph{$\mathcal D$-Whitney decomposition} of an open set  $\varnothing\subsetneq O\subsetneq \mathbb R$ is the collection $\mathbf{w}_{\mathcal D}(O)\subset \mathcal D$ of intervals $I$ 
satisfying 
\[
c_1 \ell_I \leq \dist(I,\mathbb R\setminus O) \leq c_2\ell_I
\] which are maximal with respect to inclusion. Here
 {$1< c_1 < c_2$} are fixed  numerical constants. For the arguments of this paper, it suffices to take $c_1=3,c_2=5$. 
To each  such nontrivial open set $O$ and $\mathbb Q\subset \mathbb P^{O}_{\mathcal D, \mathcal D'}$ we may associate the collection  
\begin{equation}\label{e:tilesO} 
	\mathbb Q^{O} \coloneqq  \mathbb Q^{\mathbf{w}_{\mathcal D'}(O)}.
\end{equation}{Observe that there is no  notational overload between \eqref{e:qJ} and \eqref{e:tilesO}; in the former, the superscript is a collection of intervals, while for the latter the superscript is an open set.}
The elements of $\mathbf{w}_{\mathcal D'}(O)$ are pairwise disjoint whence each $\mathbb Q^{O}$ above is almost orthogonal. The purpose of this definition is that for each  multiplier $m$ satisfying \eqref{e:multip} with singular set $\Xi$, the equality
\begin{equation}\label{e:cvxhullms}
 	\mathrm{T}_{m} =\sum_{j=1}^{3} c_j T_{	\mathbb P^{O_\Xi}_{\mathcal D, \mathcal D_j}}
\end{equation} 
holds for a canonical choice of grids $\mathcal D, \mathcal D_{j}, 1\leq j \leq 3$,  suitably chosen constants $c_j$ and adapted wave packets in \eqref{e:defmulti}. From this point onward, we will work with a fixed pair $\mathcal D, \mathcal D_j$ of grids and omit this pair from the notation, writing for instance $\mathbb P^{O}$ in place of $\mathbb P^{O}_{\mathcal D, \mathcal D_j}$. Therefore, well-known reductions summarized e.g.\ in \cite{DPF+RDF}*{\S4.9},  see  also \cite{DPFR}*{eq. (2.10)}, show that   Theorem~\ref{thm:HMsparse} follows from  the same estimate for the generic model form, namely 
\begin{equation}\label{e:cvxhullms2} 
\sup_{\substack{\mathbb Q \subset \mathbb P^{O_\Xi} \\ \# \mathbb Q <\infty } } \left| \langle T_{	\mathbb Q} f_1,f_2 \rangle \right| \lesssim \left( \prod_{j=1,2} \ZZ^\star(X_j, \Xi) \right)  \sup_{\mathcal S \,\mathrm{ sparse}} {\Lambda^{\mathcal S}_{X_1,X_2}}(f_1,f_2)  .
\end{equation} 
Similarly, the  square function estimate involving  $\mathrm{H}_{\Xi,m}$ of Theorem~\ref{thm:roughsf} is a consequence of the pointwise bound  
\begin{equation} \label{eq:modelr}
 \mathrm{H}_{\Xi,m} f  \lesssim \|m\|_{\mathrm{HM}(\Xi)}  \left\| T_{\mathbb P^{\omega}} f\right\|_{\ell^2(\omega\in \Omega_\Xi)} \lesssim \|m\|_{\mathrm{HM}(\Xi)}  \ZZ^{{\star}} (X,\Xi) \mathcal S_{X,1}f
\end{equation}   
for each $f\in L^\infty_0(\R)$ and  a suitable choice of sparse collection $\mathcal S$. In the definition of  $\mathrm{H}_{\Xi,m}$ we have tacitly used that ${\mathbb P^{O_\Xi}}$ is the union of the collections $\mathbb P^\omega$ over the connected components $\omega\in \Omega_\Xi$ of $O_\Xi$.  {In particular, through the almost orthogonal decomposition of $f$ for each fixed scale detailed in \eqref{e:gab1}, the H\"{o}rmander-Mihlin type condition in \eqref{e:hmintro} is inherited by the wavelets appropriately chosen in the classes $\Psi_P(D)$ with $P \in {\mathbb P^{O_\Xi}}$.
Finally, Theorem~\ref{thm:smoothsf} for the smooth square function $\mathrm{G}_\Xi$ is a consequence of the pointwise bound
 \begin{equation}\label{eq:models}
\mathrm{G}_{\Xi} f  \coloneqq \left( \sum_{P\in \mathbb P^ {O_\Xi}}  \left( \sup_{\upphi \in \Psi_P} |\langle f,\upphi \rangle|\right)^2  {\ind}_{I_P} \right)^{  \frac12 } \lesssim \ZZ^{{\star}} (X,\Xi)   \mathcal S_{X,2}f 
\end{equation}   
for each $f\in L^\infty_0(\R)$ and  a suitable choice of sparse collection $\mathcal S$. 

%%%%%%%%%%%%%%%%%%%%%%%%%%%%%% SECTION SECTION SECTION
\subsection{Discretization of multipliers obeying $\Xi$-variation assumptions}\label{sec:Rpq} This subsection addresses the discretization of the multiplier class $\mathrm{Mar}(\Xi)$ as a special case   of the more general $R_{p,1}^\Xi, 1\leq p\leq \Xi$, classes defined momentarily. In order to describe these, the following equivalent version of the Lorentz sequence norms 
 \[
 \left\| \left\{a_j :\, j\in \mathbb Z \right\} \right\|_{\ell^{p,q}(j)} \coloneqq \left\|   2^{n}\left( \#\left\{j\in \mathbb Z:\, |a_j|\in [2^{n},2^{n+1})\right\}\right)^{\frac{1}{p}} \right\|
 _{\ell^q(n\in \mathbb Z)}
 \] 
with $0<p<\infty, \,0<q\leq \infty$, is used below.

%%%%%%%%%%%%%%%%%%%%%%%%%%%%%% DEFINITION DEFINITION DEFINITION
\begin{definition}[$R_{p,q} ^\Xi$ multipliers] Let $\Xi\subset\R$ be a closed null set. Say that $m$ is a \emph{$R_{p,q}^\Xi$-atom} if for each $\omega \in \Omega_\Xi$ there exists $J_\omega\in\N$ and coefficients $\{a_{j,\omega}:\, 1\leq j \leq J_\omega\}$ and a collection $\{\alpha(j,\omega): \, 1\leq j\leq J_\omega\}$ consisting of pairwise disjoint subintervals of $\omega$, such that
\begin{equation}
\label{e:rpatom}
m\ind_\omega=  \sum_{1\leq j \leq J_\omega} a_{j,\omega} \ind_{\alpha(j,\omega)}, \qquad  \left\| a_{j,\omega} \right\|_{\ell^{p,q}(j)} \leq 1.
\end{equation}
The class of  \emph{$R_{p,q}^\Xi$ multipliers} is then the atomic space generated by  $R_{p,q}^\Xi$-atoms as defined above, equipped with the corresponding atomic norm. As customary, $R_{p}^\Xi$  stands for $R_{p,q}^\Xi$ when $p=q$.
\end{definition}
%%%%%%%%%%%%%%%%%%%%%%%%%%%%%% DEFINITION DEFINITION DEFINITION

 These classes of multipliers have previously appeared in the literature in different forms, at least when $\Xi $ has lacunary structure.  The class of $R_{1,1}^{\Lambda_1}$ multipliers associated to the typical first order lacunary set $\Lambda_1\coloneqq\{2^k:\, k\in\Z\}$  coincides with the classical Marcinkiewicz class. More generally, the class $R_{1,1} ^\Xi$ coincides with the $\mathrm{Mar}(\Xi)$ class of \S\ref{sec:marc} with singular set $\Xi$; see the elementary argument of \cite{CCDP2018}*{Proposition 2.9}. Similarly, the class of multipliers  $R_{p,p} ^{\Lambda_1}$ coincides with the  $R_p$-multipliers appearing in \cites{BCPV,TW}; see also \cite{CRDFS}. The class $R_2=R_{2,2} ^{\Lambda_\tau}$ with singular set $\Lambda_\tau$ being a lacunary set of general order $\tau\geq 1$ has been introduced and studied in \cite{BCPV} where it is shown that these multipliers satisfy the best-possible endpoint modular estimate as in Corollary~\ref{c:sample} with $X=L\log^{\frac{\tau}{2}}(L)$ and $w\equiv 1$.
 
One more definition is needed for the discretization of $R_{p,1}^\Xi$-multipliers. Its importance is revealed by Lemma \ref{lem:Jatom} below.
%%%%%%%%%%%%%%%%%%%%%%%%%%%%%% DEFINITION DEFINITION DEFINITION
\begin{definition}\label{def:Jatom} Let $J$ be a positive integer. We say that $m$ is a  \emph{$R_{p,1,J} ^\Xi$-atom}  and write $m\in\mathcal R_{p,1,J} ^\Xi$ if for every $\omega\in \Omega_\Xi$ there exist $J_\omega\leq J$ and coefficients $\{a_{j,\omega}:\, 1\leq j \leq J_\omega\}$ and a collection $\{\alpha(j,\omega): \, 1\leq j\leq J_\omega\}$ consisting of pairwise disjoint subintervals of $\omega$, such that
\[
m \ind_\omega= \frac{1}{J ^{\frac1p}}    \sum_{1\leq j \leq J_\omega} \eps_{j,\omega}  \ind_{\alpha(j,\omega)}, \qquad  \sup_{\substack{  1\leq j \leq J_\omega}} |\eps_{j,\omega}|\leq  1.
\]	
\end{definition}
%%%%%%%%%%%%%%%%%%%%%%%%%%%%%% DEFINITION DEFINITION DEFINITION

%%%%%%%%%%%%%%%%%%%%%%%%%%%%%% LEMMA LEMMA LEMMA
\begin{lemma}\label{lem:Jatom} The class $R_{p,1} ^\Xi$ coincides with the atomic space generated by   $\bigcup_{J \in 2^\N} \mathcal R_{p,1,J}$.
\end{lemma}
%%%%%%%%%%%%%%%%%%%%%%%%%%%%%% LEMMA LEMMA LEMMA

%%%%%%%%%%%%%%%%%%%%%%%%%%%%%% PROOF PROOF PROOF
\begin{proof} The fact that each $R_{p,1,J} ^\Xi$-atom is a uniformly bounded multiple of an $R_{p,1}^\Xi$-atom follows by the routine verification of  the bound
\[
\sup_{\omega\in\Omega_{\Xi}}\left\| \left\{J_\omega ^{-	 \frac1p } \eps_{j,\omega}:\, j\in\Z\right\}\right\|_{\ell^{p,1}(j)}\lesssim 1
\]
with $\left\{\eps_{j,\omega}:\, 1\leq j\leq J_\omega,\,\omega\in\Omega_\Xi\right\}$ as in the definition of a $R_{p,1,J} ^\Xi$-atom.  It remains to check that if $m$ is an $R_{p,1}^\Xi$-atom then  $m$ may be obtained as a convex combination in $J\in 2^\N$ of $R_{p,1,J} ^\Xi$- atoms. To that end let us fix $\omega\in\Omega_\Xi$ and consider an $R_{p,%q
{1}} ^\Xi$-atom
\[
m\ind_\omega=\sum_{1\leq j\leq J_\omega}a_{j,\omega}\ind_{\alpha(j,\omega)}.
\]
For $J\in 2^\N$ consider the non-increasing rearrangement of the sequence $\{a_{j,\omega}:\, 1\leq j\leq J_\omega\}$, defined as
\[
a_{J,\omega} ^* \coloneqq\inf\big\{t>0: \,\#\{j: |a_{j,\omega}|>t\}\leq J\big\},
\]
and set
\[
  \mathrm{ind}(J,\omega)\coloneqq\left\{j: \,|a_{j,\omega}|\in \left(a_{2J,\omega} ^*, a_{J,\omega} ^* 	\right]\right\}.
 \]
By definition $\#\mathrm{ind}(J,\omega)\leq 2J$, whence
\[
m{ \ind_\omega}=\sum_{J\in   2^{\mathbb N}} J^{\frac1p} a_{J/2,\omega} ^*   m_{J}\cic{1}_\omega   , \qquad m_{J}\cic{1}_\omega \coloneqq \frac{1}{J^{\frac1p}}\sum_{j\in \mathrm{ind}(J/2,\omega)} \frac{a_{j,\omega}}{a_{J/2,\omega} ^*} \cic{1}_{\alpha(j,\omega)}.
\]
It is clear that $m_J\in  R_{p,1,J} ^\Xi$ and it is not difficult to see that
\[
\sup_{\omega\in \Omega} \sum_{J\in   2^{\mathbb N}} {J}^{{\frac1p}} a_{J,\omega} ^*   \eqsim \sup_{\omega\in \Omega} \left\| a_{j,\omega} \right\|_{\ell^{p,1}(j)} =1.
\]
Thus we proved that every $R_{p,q} ^\Xi$-atom is in the convex hull  over $J\in 2^\N$ of  $R_{p,1,J} ^\Xi$-atoms. In particular, these two  collections of atoms generate the same space and the proof of the lemma is complete.	
\end{proof}
%%%%%%%%%%%%%%%%%%%%%%%%%%%%%% PROOF PROOF PROOF

The concepts above are linked via  the following two lemmas. It is convenient to separate the two cases $p\in \{1,2\}$ but a unified statement is also possible.

%%%%%%%%%%%%%%%%%%%%%%%%%%%%%% LEMMA LEMMA LEMMA
\begin{lemma}\label{lem:marctomodel} Let $m$ be a $R_{1,1,J} ^\Xi$-atom as in Definition~\ref{def:Jatom}. Then the multiplier operator associated to $m$ belongs to the convex hull of the model sums $T_{\mathbb Q } f$ defined as in \eqref{e:defmulti} with $\mathbb Q\subset \mathbb P^{O_{\Xi_m}}$, with $\mathbb P ^{O_{\Xi_m}}$ given as in \eqref{e:tilesO} and $\Xi_m\subset\R$ being a closed null set satisfying the Zygmund property \[\ZZ^\star(X,\Xi_m)\lesssim \ZZ^{\star\star}(X,\Xi).\]
\end{lemma}	
%%%%%%%%%%%%%%%%%%%%%%%%%%%%%% LEMMA LEMMA LEMMA

%%%%%%%%%%%%%%%%%%%%%%%%%%%%%% PROOF PROOF PROOF
\begin{proof} Since $m$ is a $R_{1,1,J} ^\Xi$-atom, the multiplier operator associated with $m$ can be written in the form
\[
\mathrm{T}_m =\frac{1}{J}\sum_{1\leq j \leq J} \sum_{\omega\in\Omega_\Xi}\eps_{j,\omega} \mathrm{H}_{\alpha(j,\omega)},\qquad\sup_{1\leq j \leq J_\omega}{|\eps_{j,\omega}|}\leq 1,
\]
where we recall that $\mathrm{H}_\omega$ denotes the rough frequency projection onto some interval $\omega\subset \R$. As  already seen in the discussion leading up to \eqref{eq:modelr}, the operator $\mathrm{H}_\omega$ is in the convex hull of model forms $T_{\mathbb Q}$ with $\mathbb Q\subset \mathbb P^\omega$. Defining ${\Xi_{m,j}}$ to be the set of endpoints of the collection of intervals $\left\{\alpha(j,\omega):\, \omega\in\Omega_\Xi\right\}$ it follows that $\mathrm{T}_m$ is in the convex hull in $j$ of the model sums $T_{\mathbb Q} f$ with $\mathbb Q\subset \mathbb P^{O_{\Xi_{m,j}}}$. Note that as a direct consequence of the definitions of property $\ZZ^{\star \star}(X,\Xi)$ and of the sets $\Xi_{m,j}$ it follows that
\[
\sup_j \ZZ^{\star}(X,\Xi_{m,j}) \lesssim \ZZ^{\star \star}(X,\Xi)
\]
and the proof is complete.
\end{proof}
%%%%%%%%%%%%%%%%%%%%%%%%%%%%%% PROOF PROOF PROOF

%%%%%%%%%%%%%%%%%%%%%%%%%%%%%% LEMMA LEMMA LEMMA
\begin{lemma}\label{lem:R21tomodel}Let $m$ be a $R_{2,1,J}^\Xi$ atom as given in Definition~\ref{def:Jatom}. Then the multiplier operator associated to $m$ belongs to the convex hull of the model sums $J^{-{\frac12}}T_{\mathbb Q}$  where $\mathbb Q\subset \mathbb P^{O_{\Xi_m}}$,  the set of tiles $\mathbb P ^{O_{\Xi_m}}$ is given as in \eqref{e:tilesO}, and $\Xi_m\subset\R$ is a closed null set satisfying the Zygmund property \[\ZZ^\star(X,\Xi_m)\lesssim \sqrt{J}\ZZ^{\star\star}(X,\Xi).\]
\end{lemma}	
%%%%%%%%%%%%%%%%%%%%%%%%%%%%%% LEMMA LEMMA LEMMA

%%%%%%%%%%%%%%%%%%%%%%%%%%%%%% PROOF PROOF PROOF
\begin{proof}As before, use that $m\in R_{2,1,J}^\Xi$ in order to write
\[
\mathrm{T}_m=\frac{1}{{J}^{\frac12}}\sum_{1\leq j\leq J} \sum_{\omega\in\Omega_\Xi}\eps_{j,\omega}\mathrm{H}_{\alpha(j,\omega)}.
\]
This time we set $ \Xi_m\coloneqq \cup_{1\leq j\leq J} \Xi_{m,j}$ where each $\Xi_{m,j}$ is defined as in the previous proof. It follows that $T_m$ can be written as a convex combination of model sums $J^{-%{\frac12}
{\frac 1{2}}} T_{\mathbb Q}$ with $\mathbb Q\subset \mathbb P^{O_{\Xi_m}}$.  As before,   note that $\sup_j\ZZ^\star(X,\Xi_{m,j})\lesssim \ZZ^{\star\star}(X,\Xi)$. In order to complete the proof of the lemma it thus suffices to observe that, in general, if $\Theta=\cup_{1\leq j\leq J}\Theta_j$ then there holds
\begin{equation} \label{e:Jzyg}
\ZZ^{\star}(X,\Theta)\leq \sqrt{J} \sup_{1\leq j \leq J} \ZZ^{\star }(X,\Theta_j)
\end{equation}  
as can be easily checked by applying the definition of the $\ZZ^\star$-property and the Cauchy-Schwarz inequality.
\end{proof}
%%%%%%%%%%%%%%%%%%%%%%%%%%%%%% PROOF PROOF PROOF

Lemma~\ref{lem:R21tomodel} will be used to deduce a sparse domination theorem for $R_{p,1}^\Xi$-multipliers with $1\leq p\leq 2$,\ involving a  complex interpolation space, see \cite{HarHas}*{Theorem 5.5.1}, in the  Orlicz scale}
\begin{equation} 
\label{e:Xintp}  [X,L^2(0,1)]_{\theta}= %X^\theta [L^2(0,1)]^{1-\theta},
{X^{1-\theta} [L^2(0,1)]^{\theta}}, \qquad 0\leq \theta \leq 1.
\end{equation}

%%%%%%%%%%%%%%%%%%%%%%%%%%%%%% THEOREM THEOREM THEOREM
\begin{theorem}  \label{thm:F} Suppose  $\Xi $  has the {maximal multiscale} $\ZZ(X)$ property. Let $m$ be a  $R_{p,1} ^\Xi$, $1\leq p\leq 2$,  multiplier with singular set $\Xi$. Then 
\[
\left\|\mathrm{T}_m  \right\|_{X,X_p} \lesssim \|m\|_{R_{p,1}^\Xi} \left[ \ZZ^{\star\star}(X, \Xi)   \right]^{\frac2p}
\]
where $X_p $ is given by \eqref{e:Xintp} with $\theta=\frac{2(p-1)}{p}$ and the implicit constant is absolute.
\end{theorem}
%%%%%%%%%%%%%%%%%%%%%%%%%%%%%% THEOREM THEOREM THEOREM

The proof of Theorem \ref{thm:F} is postponed to Section \ref{s:nextprob}.

%%%%%%%%%%%%%%%%%%%%%%%%%%%%%% SECTION SECTION SECTION
\section{Gabor decomposition and smooth Zygmund projection} \label{s:gab} This section develops a smooth projection theorem adapted to a singular set $\Xi$ enjoying the multiscale $\ZZ(X)$ property. This is done in \S\ref{ss:projl}. In \S\ref{ss:ss}, the projection lemma is employed to deduce a few estimates for  model operators in~\eqref{e:defmulti}.

%%%%%%%%%%%%%%%%%%%%%%%%%%%%%% SECTION SECTION SECTION
\subsection{Gabor decomposition and sum of localized functions} The main projection argument employed in this paper uses two previously known ingredients. The first is the following almost orthogonal decomposition \cites{DLTT,LT1}.  Let  $\eta$ be a Schwartz function on $\R $ with the properties
\begin{equation}\label{e:1} 
\supp \widehat{\eta} \subset \left[0,1\right], \qquad \sum_{k\in \mathbb Z} \left| \widehat{\eta} \left({\textstyle \xi -\frac{k}{2} }\right)\right|^2 \equiv 1.
\end{equation}
For each $I \in \mathcal D,k\in \mathbb Z$ let
\[
\eta_{I,k}\coloneqq \mathrm{Tr}_{c_I} \mathrm{Dil}^{1}_{\ell_I} \textrm{Mod}_{\frac{k}{2}} \eta, \qquad \zeta_{I,k}\coloneqq \chi^{-6{\mathsf{dec}}}_I  \eta_{I,k} ,
\]
where $\mathrm{Mod}_a f(x):= \e^{2\pi i a x } f(x)$, $a \in \R$. For each $m \in \mathbb Z$, Poisson summation yields
\begin{equation}\label{e:gab1}
f=\sum_{k\in \mathbb Z} \sum_{I\in \mathcal D_m} |I| \langle f, \eta_{I,k}  \rangle\eta_{I,k}
\end{equation}
with convergence in $L^2(\R)$ and almost everywhere.  

%%%%%%%%%%%%%%%%%%%%%%%%%%%%%% REMARK REMARK REMARK
\begin{remark} Given $A\subset \mathbb R$, set  
\begin{equation} \label{e:neighscale}
	\mathcal N(A){\coloneqq }\left\{ k\in \mathbb Z :\, \mathrm{dist}(A,k)<9\right\}.
\end{equation}
Suppose $\Xi$ has the multiscale $\ZZ(X)$ property. Triangle inequality and averaging then yield   \begin{equation} \label{e:smoothxc}
\begin{aligned}
\left\| \sum_{k\in \mathcal N(\lambda \Xi)} a_k \zeta_{[0,1),k} \right\|_{X'} & \leq   \sum_{j=-9}^{9} \left\| \sum_{k\in j+ \lfloor \lambda \Xi \rfloor } a_k \zeta_{[0,1),k} \right\|_{X'} \leq \sum_{j=-9}^{9} \ZZ(X, j +\lfloor \lambda \Xi \rfloor) \sqrt{\sum_{k\in j + \lfloor \lambda \Xi \rfloor } |a_k|^2}  \\
& \lesssim \ZZ^\star(X,\Xi)\sqrt{\sum_{k\in  \mathcal N(\lambda\Xi)  } |a_k|^2}
\end{aligned}
\end{equation}
for  all $\lambda \in  2^{\mathbb Z}$ and  all $\{a_k:\, k\in \mathcal N(\lambda \Xi)\}$.
We used that, by modulation invariance,
\begin{equation} \label{e:rathobv_bis}
  \ZZ(X,\mathbb K) = \ZZ(X,k + \mathbb K), \qquad \forall \mathbb K \subseteq \Z,\quad  k \in \Z
\end{equation}
to pass to the second line.
Property   \eqref{e:smoothxc} will be used in the proof of Lemma~\ref{l:proj}.  
\end{remark}
%%%%%%%%%%%%%%%%%%%%%%%%%%%%%% REMARK REMARK REMARK

The second is a technical lemma used to efficiently $L^2$-estimate  sums of spatially localized functions. Its proof is literally a rewriting of \cite{MTT2}*{Lemma 5.1} and thus we omit it.

%%%%%%%%%%%%%%%%%%%%%%%%%%%%%% LEMMA LEMMA LEMMA
\begin{lemma} \label{l:L2sum} Let $\mathcal L\subset \mathcal D$ be a collection of pairwise disjoint  intervals. For each   $L\in \mathcal L$ let $F_{L}\in L^2(\R)$.  	Then
\[
 \left\|  \sum_{L\in \mathcal L} F_{L}\right\|_{2} \lesssim  \left( \sum_{L \in \mathcal L} |L| \right)^{\frac12} \sup_{L\in \mathcal L}\left\langle F_{L}\right\rangle_{2,L,-}.
\]
\end{lemma}
%%%%%%%%%%%%%%%%%%%%%%%%%%%%%% LEMMA LEMMA LEMMA

%%%%%%%%%%%%%%%%%%%%%%%%%%%%%% SECTION SECTION SECTION
\subsection{The projection lemma} \label{ss:projl}   In this paragraph,  $\Xi$ is a closed null set with the % scale invariant 
{multiscale} $\ZZ(X)$ property. Notice that the next lemma involves the tile collection $\mathbb P^{O_\Xi}$ as defined in \eqref{e:tilesO} and the frequencies $\mathcal N(\lambda \Xi)$ from   \eqref{e:neighscale}.

%%%%%%%%%%%%%%%%%%%%%%%%%%%%%% LEMMA LEMMA LEMMA
\begin{lemma}[Projection]\label{l:proj}  Let $L\in \mathcal D$ and $f$ be a function with $\supp f \subset L$. Then
\[
g \coloneqq \sum_{m\in \mathbb Z} \sum_{k\in  \mathcal N(\ell_L \Xi)} |L| \langle f , \eta_{L^{+m},k}  \rangle\eta_{L^{+m},k}  
\] 
satisfies
\begin{align} \label{e:goodproj}
&\left\langle f,\upphi_{P}\right\rangle = \langle g,\upphi_{P}\rangle , \qquad \forall P\in \mathbb P^{O_\Xi}:\, \ell_P \geq \ell_L,\, \forall \upphi_P\in \Psi_P, 
\\ \label{e:l2proj} & 
\left\langle g \right\rangle_{2,L,-}\lesssim \ZZ^{\star}(X,\Xi) \left\|f\circ \mathrm{Sy}_L\right\|_{X}. 
\end{align}
\end{lemma}
%%%%%%%%%%%%%%%%%%%%%%%%%%%%%% LEMMA LEMMA LEMMA

%%%%%%%%%%%%%%%%%%%%%%%%%%%%%% PROOF PROOF PROOF
\begin{proof} Statement \eqref{e:goodproj} is immediate from the representation \eqref{e:1} applied to $f$,
\[
f= \sum_{m\in \mathbb Z} \sum_{k\in \mathbb Z} |L| \langle f , \eta_{L^{+m},k}  \rangle\eta_{L^{+m},k}  
\]
and from the comparison of the frequency supports of $f-g$ and $\upphi_P,P  \in \mathbb P^{O_\Xi}$, $\ell_P\geq \ell_L$. To wit, if $P  \in \mathbb P^{O_\Xi}$, the definitions force the existence of $\xi\in \Xi$ such that 
\[
\omega_P \subset \left[\xi-6\ell_{P}^{-1},  \xi+6\ell_{P}^{-1}\right)  \subset \left[\xi-6\ell_{L}^{-1},  \xi+6\ell_{L}^{-1}\right) \subset \ell_{L}^{-1}\bigg[ \left\lfloor \ell_{L}\xi \right \rfloor -7 ,  \left\lfloor \ell_{L}\xi \right \rfloor +7 \bigg)
\]
and the latter set is disjoint from the support of $f-g$.

Turn to the proof of \eqref{e:l2proj}. By the invariance over linear changes of coordinates of the assumptions and  statement, it suffices to prove the case $L=[0,1), $ so that $L^{+m}=[m,m+1)$. For this reason, we may write $\eta_{m,k}$ in place of $\eta_{[m,m+1),k}  $, $\chi_m=\chi_{[m,m+1)}$ and similarly,  $\mathbb K=\mathcal N(1\Xi)$ and
\[
g= \sum_{m\in \mathbb Z} g_m, \qquad g_m\coloneqq \sum_{k\in \mathbb K}  \langle f , \eta_{m,k}  \rangle\eta_{m,k} . 
\]
By the triangle inequality, it suffices to prove
\begin{equation}
\label{e:projl3}
\sup_{\|h\|_2=1}
\left|\langle g_m\chi^{-{\mathsf{dec}}},h \rangle\right| \lesssim  \ZZ^\star(X,\Xi) (1+m^2)^{-4{\mathsf{dec}}}\left\|f \right\|_{X}.
\end{equation}
In fact, rewriting
\[
 g_m\chi^{-{\mathsf{dec}}} = \gamma_m \cdot\left[ \chi^{-{\mathsf{dec}}} \chi_m^{6{\mathsf{dec}}}\right], \qquad  \gamma_m\coloneqq    \sum_{k\in \mathbb K}  \langle f\chi_m^{6{\mathsf{dec}}} , \zeta_{m,k}  \rangle\zeta_{m,k} 
\]
and using  that $\|\chi^{-{\mathsf{dec}}} \chi_m^{{\mathsf{dec}}}\|_\infty \lesssim (1+|m|^2)^{{\mathsf{dec}}}$,
it is enough to prove that 
\begin{equation} \label{e:projl1}
\sup_{\|h\|_2=1}
\left|\langle \gamma_m ,h \rangle\right| \lesssim  \ZZ^\star(X,\Xi)(1+m^2)^{-6{\mathsf{dec}}}\left\|f \right\|_{X} {.}
\end{equation}
We apply the representation formula \eqref{e:gab1} to $h$ of $L^2$-norm 1, obtaining
\begin{equation} \label{e:projl2}
h= \sum_{z \in \mathbb  Z} \sum_{n\in \mathbb Z} \langle h,\eta_{n,z} \rangle  \eta_{n,z}, \qquad \sum_{z \in \mathbb  Z} \sum_{n\in \mathbb Z} |\langle h,\eta_{n,z} \rangle  |^2\lesssim 1.
\end{equation}
Then, by disjointness  of frequency supports,
\[
\begin{split} 
\left| \langle \gamma_m ,h \rangle \right|&= \left|\sum_{\kappa\in \{0,\pm1\} } \sum_{k\in \mathbb K} \sum_{n\in \mathbb Z} \langle f\chi_m^{6{\mathsf{dec}}} , \zeta_{m,k}  \rangle \left \langle \zeta_{m,k} , \eta_{n,k+\kappa} \right \rangle\langle \eta_{n,k+\kappa}, h \rangle \right| =    \left| \langle f\chi_m^{6{\mathsf{dec}}}, h_{m } \rangle \right|  
\\ 
&
\leq \| f\chi_m^{6{\mathsf{dec}}}\|_X    \|h_{m}\|_{X'} \lesssim   (1+m^2)^{-6{\mathsf{dec}}}  \| f \|_X \|h_{m}\|_{X'} ,
\end{split}
\]
having set
\[
h_{m}= \sum_{k\in \mathbb K} a_{m,k} \zeta_{m,k}, \qquad a_{m,k} =\sum_{\kappa\in \{0,\pm1\} }  \sum_{n\in \mathbb Z} \left \langle \zeta_{m,k} , \eta_{n,k+\kappa} \right \rangle\langle \eta_{n,k+\kappa}, h \rangle
\]
and used $\supp f\subset [0,1)$, $ \|\chi_m^{6{\mathsf{dec}}}\cic{1}_{[0,1)}\|_\infty \lesssim (1+m^2)^{-6{\mathsf{dec}}}$. To finish the proof of \eqref{e:projl1} it remains to appeal to \eqref{e:smoothxc} translated to $[m,m+1)$ and estimate 
\[
\begin{split}
\|h_{m}\|_{X'} &\leq \ZZ^\star(X,\Xi) \left(\sum_{k\in \mathbb K} |a_{m,k} |^2 \right)^{\frac12} 
\\ 
& \lesssim \ZZ^\star(X,\Xi) \left(\sup_{\substack{ \kappa \in \{0,\pm1\}\\ k\in \mathbb K}}  \sum_{u\in \mathbb Z}|\langle \zeta_{m,k} , \eta_{m+u,k+\kappa}\rangle|^{{2}} \right)^{\frac12} \left( \sum_{\substack{z \in \mathbb  Z\\ n\in \mathbb Z}} |\langle h,\eta_{n,z} \rangle  |^2\right)^{ \frac12} \lesssim \ZZ^\star(X,\Xi),
\end{split}
\]
taking into account \eqref{e:projl2}
and the easy estimate $|\langle \zeta_{m,k} , \eta_{m+u,k+\kappa}\rangle|\lesssim (1+|u|^2)^{-1}$. The proof of \eqref{e:projl1}, and consequently of the lemma, is thus complete.  
 \end{proof}
%%%%%%%%%%%%%%%%%%%%%%%%%%%%%% PROOF PROOF PROOF

%%%%%%%%%%%%%%%%%%%%%%%%%%%%%% SECTION SECTION SECTION
\subsection{Tail estimates for the multipliers.} \label{ss:ss} We  now turn to a first series of applications of Lemma~\ref{l:proj}. As anticipated, throughout this paragraph $\Xi$ is a fixed singular set with the multiscale $\ZZ(X)$ property for $X\in \{X_1,X_2\}$. For brevity, $C_{X_j}\coloneqq  \ZZ^\star (X_j, \Xi)$ for $j=1,2$. Hereafter, $\mathbb Q$ is a generic subset of $\mathbb P^{O_\Xi}$ so that the collections   $\mathbb{Q} _=(G)$ contain tiles  from $\mathbb Q$  of fixed spatial component $G\in \mathcal D$ and frequency component belonging to the  $\mathcal D'$-Whitney decomposition of the open set $O_\Xi$. The latter fact is what makes the $\ZZ(X)$ property relevant. Furthermore,    unless otherwise mentioned $\upphi_P$ stands for a generic element of $\Psi_P$, $P\in \mathbb P $. 

%%%%%%%%%%%%%%%%%%%%%%%%%%%%%% LEMMA LEMMA LEMMA
\begin{lemma} \label{l:newtail_fin} Let $f_1, f_2 \in L^\infty_0(\R^d)$ be given functions, and 
$G_0, G_1, G_2 \in \mathcal{D}$ be intervals of equal length $\ell$ such that $\max\{\dist(G_0,G_1), \dist(G_0,G_2) \} \geq \ell/2$. 
Then
\begin{align}\label{e:sszygm_fin}	
\qquad & \left(  \sum_{ P\in \mathbb{Q} (G_0) }  {|I_P|^2} \left|\langle f_1 {\ind_{G_1}}, \upphi_P  \rangle\right|^2 \left| \chi_P^{\mathsf{dec}} (x) \ind_{{G_2}} (x) \right|^2 \right)^{\frac12}
\\
& \qquad \qquad \qquad \qquad \quad \lesssim C_{X_1} \prod_{j=1,2} \left( 1 + \frac{\dist(G_0,G_j)}{\ell} \right)^{2-\mathsf{dec}} \inf_{G_0 {\cup G_1} \cup G_2} \mathrm{M}_{X_1} f_1,
\\ \label{eq:decayzyg_fin}
& \left\| T_{ \mathbb{Q} (G_0) } (f_1 \ind_{G_1}) {\ind_{G_2}} \right\|_{2} \lesssim C_{X_1} \sqrt{\ell} \prod_{j=1,2} \left( 1 + \frac{\dist(G_0,G_j)}{\ell} \right)^{2-\mathsf{dec}} \inf_{G_0 \cup G_1 {\cup G_2}} \mathrm{M}_{X_1} f_1, \\
\label{e:newtail2_fin}
& \left|\left\langle T_{\mathbb{Q} (G_0) }\left(f_1 \ind_{G_1}\right),\ind_{G_2} f_2 \right\rangle  \right|\lesssim \prod_{j=1,2} C_{X_j} \sqrt{\ell} \left( 1 + \frac{\dist(G_0,G_j)}{\ell} \right)^{2-\mathsf{dec}} \inf_{G_0 \cup G_j} \mathrm{M}_{X_j}f_j. 
\end{align}
\end{lemma}}
%%%%%%%%%%%%%%%%%%%%%%%%%%%%%% LEMMA LEMMA LEMMA

%%%%%%%%%%%%%%%%%%%%%%%%%%%%%% REMARK REMARK REMARK
\begin{remark} The presence of the multiscale collections, e.g.\ $\mathbb{Q} (G_0) $, in the statement is more convenient for applications below. However, at the root of the lemma lies a purely single scale analysis, which is exemplified by the intermediate inequality \eqref{e:sszygm_new}. Consulting \eqref{e:sszygm_new} in advance might facilitate the parsing of the rather technical and crowded statements contained in Lemma \ref{l:newtail_fin}. 
\end{remark}
%%%%%%%%%%%%%%%%%%%%%%%%%%%%%% REMARK REMARK REMARK

%%%%%%%%%%%%%%%%%%%%%%%%%%%%%% PROOF PROOF PROOF
\begin{proof}[Proof of Lemma~\ref{l:newtail_fin}]  We begin the proof of \eqref{e:sszygm_fin} showing an \emph{a priori} weaker version of the inequality in the case of dyadic intervals $R, S \in \mathcal{D}$ such that $\ell_R = \ell_S$ restricting $\mathbb{Q} (R)$ to $\mathbb{Q}_{=}(R)$, namely
\begin{equation} \label{e:sszygm_new}
\left(  \sum_{P\in \mathbb{Q} _=(R)}   \left|\langle f_1 \ind_{S}, \upphi_P  \rangle\right|^2 \right)^{  \frac12}  \lesssim C_{X_1} \left( 1 + \frac{\dist(R,S)}{\ell_R} \right)^{1-\mathsf{dec}}   \inf_{R \cup S} \mathrm{M}_{X_1} f_1.
\end{equation}
We apply Lemma~\ref{l:proj} to $f_1 \ind_{S}$ yielding the projection $g_1$ with 
\[
\langle f_1\ind_{S},\upphi_P\rangle=\langle g_1,  \upphi_P\rangle,\qquad \forall P\in\mathbb Q_{=}(R),\quad\upphi_P\in\Psi_P,
\] 
and $\langle g_1 \rangle _{2,S,-}\lesssim C_{X_1} \langle f_1 \rangle _{X_1,S}$.  Since $\mathbb P^{O_\Xi}$ satisfies \eqref{e:aotiles}, the left hand side of \eqref{e:sszygm_new} can be estimated {applying the $L^2$-bound for $T_{\mathbb Q_=(R)}$ from Lemma~\ref{lem:ao}} as
\[
\begin{split}
\left(  \sum_{P\in \mathbb Q _{=}(R)}  \left|\langle f_1 \ind_{S}, \upphi_P  \rangle\right|^2 \right)^{\frac12} \lesssim \langle g_1 \rangle_{2,R,+}.
\end{split}
\]
Using the decay properties of the function $\chi$ in the first line and the estimates on the averages of the functions $g_1$ to pass to the second line, we have
\[
\begin{split}
\langle g_1 \rangle_{2,R,+} ^2 &\leq \frac{1}{\ell_R}\int_{\R}|g_1|^2 \chi_{S} ^{-2\mathsf{dec}}(\chi_{R}\chi_{S}) ^{2\mathsf{dec}} \lesssim \left(1+\frac{\dist(R,S)}{\ell_R}\right)^{-2\mathsf{dec}}\langle g_1\rangle_{2,S,-} ^2
\\
& \lesssim C_{X_1}^2 \left(1+\frac{\dist(R,S)}{\ell_R}\right)^{-2\mathsf{dec}} \langle f_1 \rangle_{X_1,S}^{{2}} \lesssim  C_{X_1}^{{2}} \left(1+\frac{\dist(R,S)}{\ell_R}\right)^{%1
	{2}-2\mathsf{dec}} \langle f_1 \rangle_{X_1,R ^*}^2
\end{split}
\]
where $R ^* \coloneq 3\left(1+\frac{\dist(R,S)}{\ell_R}\right) R$ so that $R,S \subseteq R ^*$, hence the proof is complete. To obtain the proof of \eqref{e:sszygm_fin}, we observe that splitting $ \mathbb{Q} (G_0)$ and $G_1$, and summing up it suffices to prove that for $m \in \N$, $R\in \mathcal D_m (G_0)$, and $S \in \mathcal D_m (G_1)$ we have
\begin{equation}
\begin{split}
& \left(  \sum_{ P\in \mathbb{Q}_= (R) }  {|I_P|^2} \left|\langle f_1 \ind_{S} , \upphi_P  \rangle\right|^2 \left| \chi_P^{{\mathsf{dec}}} (x)  \ind_{G_2} (x) \right|^2 \right)^{\frac12} 
\\
& \qquad \lesssim C_{X_1} \left( 1 + \frac{\dist(R,S)}{\ell_R} \right)^{1-\mathsf{dec}} \left( 1 + \frac{\dist(R,G_2)}{\ell_R} \right)^{1-\mathsf{dec}} \inf_{G_0 \cup G_1 \cup G_2} \mathrm{M}_{X_1} f_1 . 
\end{split}
\end{equation}
However, using \eqref{e:sszygm_new} to pass to the second line, we have
\[
\begin{split} 
& \left(  \sum_{ P\in \mathbb{Q}_= (R) }  {|I_P|^2} \left|\langle f_1 \ind_{S}, \upphi_P  \rangle\right|^2 \left| \chi_P^{{\mathsf{dec}}} (x)  \ind_{G_2} (x) \right|^2 \right)^{\frac12} \leq  \left( \sum_{ P\in \mathbb{Q}_= (R) } \left|\langle f_1 \ind_{S}, \upphi_P  \rangle\right|^2 \right)^{\frac12} \sup_{{G_2}} \chi^{\mathsf{dec}}_{R}(x) 
\\
&\qquad \lesssim \left(  C_{X_1} \inf_{R ^*} \mathrm{M}_{X_1} f_1 \right) \left( 1 + \frac{\dist(R,S)}{\ell_R} \right)^{1-\mathsf{dec}}
\left( 1+ \frac{\mathrm{dist}(R,%G_1
	{G_2}) }{\ell_R}\right)^{-\mathsf{dec}}
\\
&\qquad \lesssim C_{X_1} \left( 1 + \frac{\dist(R,S)}{\ell_R} \right)^{1-\mathsf{dec}} \left(1+ \frac{\mathrm{dist}(R,%G_1
	{G_2}) }{\ell_R}\right)^{1-\mathsf{dec}} \inf_{{R ^{**}}} \mathrm{M}_{X_1} f_1 ,
\end{split}
\]
where $R ^{**} \coloneq 3\left(1+\frac{\dist(R,	{G_2})}{\ell_R}\right) R^{*}$ so that $G_0, G_1, G_2 \subseteq R ^{**}$, hence the proof is complete.

The proof of \eqref{eq:decayzyg_fin} is almost identical and relies on the inequality in the case of dyadic intervals $R,S \in \mathcal{D}$ such that $\ell_R = \ell_S $ restricting $\mathbb{Q}(R)$ to $\mathbb{Q}_{=} (R)$, namely
\begin{equation}\label{eq:decayzyg}
\left\|T_{\mathbb{Q} _=(R) }\left(f_1 \ind_{S}\right)\right\|_{2} \lesssim \sqrt{\ell_R} C_{X_1}  \left(1 + \frac{\dist(R,S)}{\ell_R}\right)^{1-\mathsf{dec}} \inf_{R \cup S} \mathrm{M}_{X_1}f_1.
\end{equation}
To prove \eqref{eq:decayzyg}, apply again  Lemma~\ref{l:proj} to $f_1\ind_{S}$ thus constructing the projection $g_1$ with 
\[
\langle f_1\ind_{S},\upphi_P\rangle=\langle g_1,  \upphi_P\rangle,\qquad \forall P\in\mathbb Q_{=}(R),\quad\upphi_P\in\Psi_P,
\] 	
with $\langle g_1 \rangle _{2,S,-}\lesssim C_{X_1} \langle f_1 \rangle _{X_1,S}$. Using   \eqref{e:aotiles} of $\mathbb P^{O_\Xi}$,  the left hand side of \eqref{eq:decayzyg} is  estimated {applying the $L^2$-bound for $T_{\mathbb Q_=(R)}$ from Lemma~\ref{lem:ao}} as
\[
\left\|T_{\mathbb Q_=(R) }\left(f_1 \ind_{S}\right)\right\|_{2} \lesssim \sqrt{\ell_R}  \langle g_1 \rangle_{2,S,+}.
\]
As before we can pass from $R$ to $S$ to estimate
\[
\begin{split}
\langle g_1 \rangle_{2,R,+} ^2 & \lesssim \left(1 + \frac{\dist(R,S)}{\ell_R}\right)^{-2\mathsf{dec}} \langle g_1\rangle_{2,S,-} ^2 \lesssim C_X ^2  \left(1 + \frac{\dist(R,S)}{\ell_R}\right)^{-2\mathsf{dec}} \langle f_1\rangle_{X_1,S} ^2 
\\
&\lesssim  C_X ^2
\left(1 + \frac{\dist(R,S)}{\ell_R}\right)^{2-2\mathsf{dec}}  \langle f_1 \rangle_{X_1, R ^*}^2
\end{split}
\]
where $R ^* \coloneq 3\left(1+\frac{\dist(R,S)}{\ell_R}\right) R$ so that $R,S \subseteq R ^*$, hence concluding the proof of \eqref{eq:decayzyg}. 

To prove \eqref{eq:decayzyg_fin}, we observe that splitting $ \mathbb{Q} (G_0)$ and $G_1$ and summing up it suffices to prove that for $m \in \N$, $R\in \mathcal D_m (G_0)$, and $S \in \mathcal D_m (G_1)$ we have
\begin{equation}
\left\| T_{ \mathbb{Q}_= (R) } (f_1 \ind_{S}) \ind_{G_2} \right\|_{2} \lesssim C_{X_1} \sqrt{\ell} \left( {1 +} \frac{\dist(R,S)}{\ell} \right)^{2-\mathsf{dec}} {\left( 1 + \frac{\dist(R,G_2)}{\ell} \right)^{2-\mathsf{dec}}} \inf_{G_0 \cup G_1 {\cup G_2}} \mathrm{M}_{X_1} f_1.
\end{equation}
However, using \eqref{eq:decayzyg} to pass to the second line, we have
\[
\begin{split} \left\| T_{ \mathbb{Q}_= (R) } (f_1 \ind_{S}) \ind_{G_2} \right\|_{2} & \leq  \left\| T_{ \mathbb{Q}_= (R) } (f_1 \ind_{S}) \right\|_{2} \sup_{G_2} \chi^{\mathsf{dec}}_{R}(x) 
\\
&\lesssim \left( \sqrt{\ell_R} C_{X_1} \left(1 + \frac{\dist(R,S)}{\ell_R}\right)^{1-\mathsf{dec}} \inf_{R ^*} \mathrm{M}_{X_1} f_1 \right)
\left( 1+ \frac{\mathrm{dist}(R,G_2) }{\ell_R}\right)^{-\mathsf{dec}}
\\
&\lesssim C_{X_1} \sqrt{\ell_R} \left( 1 + \frac{\dist(R,S)}{\ell_R} \right)^{1-\mathsf{dec}} \left(1+ \frac{\mathrm{dist}(R,G_2) }{\ell_R}\right)^{1-\mathsf{dec}} \inf_{{R ^{**}}} \mathrm{M}_{X_1} f_1 ,
\end{split}
\]
where $R ^{**} \coloneq 3\left(1+\frac{\dist(R,G_2)}{\ell_R}\right) R^{*}$ so that $G_0, G_1, G_2 \subseteq R ^{**}$, hence the proof is complete. 

Finally, the proof of \eqref{e:newtail2_fin} is  also similar and relies on the inequality in the case of dyadic intervals $R,S_1,S_2 \in \mathcal{D}$ such that $\ell_R = \ell_{S_1} = \ell_{S_2} $ restricting $\mathbb{Q}(R)$ to $\mathbb{Q}_{=} (R)$, namely
\begin{equation}\label{e:newtail2}
\left|\left\langle T_{\mathbb{Q}_=(R) }\left(f_1 \ind_{S_1}\right),\ind_{S_2} f_2 \right\rangle  \right|\lesssim  \prod_{j=1,2} C_{X_j} \sqrt{\ell_R} 
\left(\frac{\dist(R,S_j)}{\ell_{R}} \right)^{1-\mathsf{dec}} \inf_{R \cup S_j} \mathrm{M}_{X_j}f_j.
\end{equation}
To prove \eqref{e:newtail2}, apply Lemma~\ref{l:proj} to $f_j\ind_{S_j}$ and constructing the projections $g_j$, $j\in\{1,2\}$, with $\langle f_j\ind_{S_j},\upphi_P\rangle=\langle g_j,\upphi_P\rangle$for all $P\in \mathbb Q_{=}(R)$ and $\langle g_j\rangle_{2,S_j,-}\lesssim C_{X_j}\langle f_j\rangle_{X_j,S_j}$. Then we apply the Cauchy-Schwarz inequality {and the $L^2$-bound for $T_{\mathbb Q_=(R)}$ from Lemma~\ref{lem:ao}} to estimate
\[
\begin{split}
\left|\left\langle T_{\mathbb Q_=(R) }\left(f_1 \ind_{S_1}\right),\ind_{S_2} f_2 \right\rangle\right|&\leq \sum_{P\in \mathbb Q_{=}(R)} {|I_P|} \left|\langle g_1,\upphi_P\rangle\right|\left|\langle g_2,\psi_P\rangle \right|\lesssim \ell_R \langle g_1 \rangle_{2,R,+} \langle g_2\rangle_{2,R,+}
\end{split}
\]
and the proof is completed by using the same argument as for \eqref{eq:decayzyg} in order to pass from $\langle g_j\rangle_{2,R,+}$ to $\langle g_j\rangle_{2,S_j,-}$. This concludes the proof of \eqref{e:newtail2}. In order to prove \eqref{e:newtail2_fin}, again it suffices to split $ \mathbb{Q} (G_0)$, $G_1$, and $G_2$, use \eqref{e:newtail2}, and sum up the contributions for different $R \in\mathcal D_m (G_0)$ and $S_j \in \mathcal D_m (G_j)$.
\end{proof}
%%%%%%%%%%%%%%%%%%%%%%%%%%%%%% PROOF PROOF PROOF

The next lemmas use the obvious disjoint splitting
\[
\mathbb P^{O_\Xi}\supset\mathbb Q = \bigcup_{\omega \in \Omega_\Xi}\mathbb Q^{\omega}.
\]
For $P\in \mathbb Q^{\omega}$ let $P^{\downarrow,\omega}=I_{P}\times (\omega_P -a_\omega)$. Then $\upphi_{P}=\exp(2\pi i a_\omega\cdot) \upphi_{P^{\downarrow,\omega}}$ for some $\upphi_{P^{\downarrow,\omega}}\in \Psi_{P^{\downarrow,\omega}}$. Thus the distance from the origin to the  frequency support of $\upphi_{P^{\downarrow,\omega}}$ is $\sim \ell_{P^{\downarrow,\omega}}^{-1}$, whence the latter is a bump function adapted to $I_P$. Clearly  
\begin{equation} \label{e:Ttil}
T_{\mathbb Q^\omega} f = \exp(2\pi i a_\omega\cdot) \sum_{P\in \mathbb Q^\omega} |I_P| \langle f,\upphi_P \rangle\upphi_{P^{\downarrow,\omega}} \eqqcolon \exp(2\pi i a_\omega\cdot) Z_{\mathbb Q^\omega} f  
\end{equation}
whence the equality, used in   Lemma~\ref{l:3sszyg} below,
\begin{equation}\label{e:Ttil2}
\left\| T_{\mathbb Q^{\omega}} f\right\|_{\ell^2(\omega\in \Omega_\Xi)} =\left\| Z_{\mathbb Q^{\omega}} f\right\|_{\ell^2(\omega\in \Omega_\Xi)}.
\end{equation}
Recall that $\mathcal D_{m}(I)=\{J\in \mathcal D(I):\ell(J)=2^{-m} \ell(I)\}.$ In the next 
lemma we deal with the first and second terms in the decomposition
\begin{equation}\label{e:splitsf}
\mathbb Q = \left(\bigcup_{|j|\leq 1} \mathbb Q (L^{+j})\right) \sqcup  \left(\bigcup_{\substack{|j|\geq 2 }} \mathbb Q (L^{+j})\right)\sqcup  \left( %\bigcup_{\substack{ m\geq 1 \\  j\in \mathbb Z }}
{\bigcup_{ m\geq 1} \bigcup_{j\in \mathbb Z}} \mathbb Q_=(L^{\uparrow m,+j})\right) \eqqcolon 
{\mathbb Q(3L)}\sqcup  \mathbb Q_{\mathsf{far},L}  \sqcup  \mathbb Q_{\mathsf{ov},L} 
\end{equation}
that we will use for a generic $\mathbb Q\subset \mathbb P^{O_\Xi}$ and any $L\in\mathcal D$. 

%%%%%%%%%%%%%%%%%%%%%%%%%%%%%% LEMMA LEMMA LEMMA
\begin{lemma} \label{lem:estimates} We have
\begin{align}\label{eq:2asszyg}	
& \sup_{\R \setminus 5L}  \left\| T_{\mathbb Q^{\omega}(3L)} f \right\|_{\ell^2(\omega\in \Omega_\Xi)}\lesssim C_X \inf_{L} \mathrm{M}_{X} (f \chi_{L}^{\mathsf{dec}}), 
\\
\label{eq:2sszyg}	
& \sup_{L} \left\|  T_{\mathbb Q_{\mathsf{far},L}^\omega } f \right\|_{\ell^2(\omega\in \Omega_\Xi)} \lesssim C_X \inf_{L} \mathrm{M}_{X} (f \chi_{L}^{\mathsf{dec}}),
\\
\label{eq:farfccc}	
& \sup_{\R} \left\| T _{\mathbb Q_{\mathsf{far},L}^\omega } (f  \ind_{L}  )  \right\|_{\ell^2(\omega\in \Omega_\Xi)} \lesssim C_X \inf_{L} \mathrm{M}_{X} (f \ind_{L}  ) . 
\end{align}
\end{lemma}
%%%%%%%%%%%%%%%%%%%%%%%%%%%%%% LEMMA LEMMA LEMMA

%%%%%%%%%%%%%%%%%%%%%%%%%%%%%% PROOF PROOF PROOF
\begin{proof}  By the localization trick in Remark~\ref{rmrk:loc} it suffices to bound the left hand sides by $\inf_{L} \mathrm{M}_{X} f $, without the $\chi_L^{\mathsf{dec}}$ factor. The desired inequality follows by {the} triangle inequality, partitioning $3L$, $\R$ and $\R \setminus (2k_0+1)L$ into intervals of the form $\{ L^{-1}, L, L^{+1} \}$ and $L^{+j}$, $L^{+k}$ for $j,k \in \Z$ with $|k| \geq k_0 + 1$, and \eqref{e:sszygm_fin} of Lemma~\ref{l:newtail_fin}.
\end{proof}
%%%%%%%%%%%%%%%%%%%%%%%%%%%%%% PROOF PROOF PROOF

For the third term   in the decomposition \eqref{e:splitsf} we obtain an  oscillation inequality instead. Compare with \eqref{e:Ttil}-\eqref{e:Ttil2}.

%%%%%%%%%%%%%%%%%%%%%%%%%%%%%% LEMMA LEMMA LEMMA
\begin{lemma} 
\label{l:3sszyg}	$\displaystyle \sup_{x,y\in 9L} \left\|   Z_{\mathbb Q^\omega_{\mathsf{ov},L}} f (x) -
  Z_{\mathbb Q_{\mathsf{ov},L}^\omega } f (y)  \right\|_{\ell^2(\omega\in\Omega_\Xi)} \lesssim C_X \inf_{L} \mathrm{M}_{X} %(f \chi_{L}^{10})$.
  {(f \chi_{L}^{\mathsf{dec}})}$.
\end{lemma}
%%%%%%%%%%%%%%%%%%%%%%%%%%%%%% LEMMA LEMMA LEMMA

%%%%%%%%%%%%%%%%%%%%%%%%%%%%%% PROOF PROOF PROOF
\begin{proof}
 Fix $x,y \in 9L$. By the triangle inequality and the localization trick {in Remark~\ref{rmrk:loc}}, it suffices to estimate
\[
\left\|  Z_{\mathbb Q^\omega_=(L^{\uparrow m,+j})} f (x) -  Z_{\mathbb Q^\omega_=(L^{\uparrow m,+j})} f (y)  \right\|_{\ell^2(\Omega_\Xi)} \lesssim C_X 2^{-m}  (1+|j|)^{-10}\inf_{L}  \mathrm{M}_X f.
\]
To do so, apply a Lipschitz estimate and 
{\eqref{e:sszygm_new}} to obtain
\[
\begin{split}
& 
\left\| {Z}_{\mathbb Q^\omega(L^{\uparrow m,+j})} f (x) -  {Z}_{\mathbb Q^\omega(L^{\uparrow m,+j})} f (y)  \right\|_{\ell^2(\Omega_\Xi)} \\
& {\quad} = {| I_P |} \left(\sum_{\omega \in \Omega_\Xi} \sum_{P\in \mathbb Q^\omega_=(L^{\uparrow m,+j})} \left| \langle f, \upphi_P \rangle \right|^2 \left| \upphi_{P^{\downarrow,\omega}}(x) - \upphi_{P^{\downarrow,\omega}}(y) \right|^2\right)^{	 \frac12 } 
\\ 
& \quad \lesssim \frac{|x-y|}{\ell(L^{\uparrow m,+j})} (1+|j|)^{-\mathsf{dec}} \left(\sum_{\omega \in \Omega} \sum_{P\in \mathbb Q^\omega_=(L^{\uparrow m,+j})} \left| \langle f, \upphi_P \rangle \right|^2  \right)^{  \frac12 } \lesssim C_X 2^{-m}  (1+|j|)^{-\mathsf{dec}}\inf_{L^{\uparrow m,+j}} \mathrm{M}_X f 
\\ 
& \quad \lesssim C_X 2^{-m}  (1+|j|)^{-\mathsf{dec}+2}\inf_{4jL^{\uparrow m,+j}} \mathrm{M}_X f \leq C_X 2^{-m}  (1+|j|)^{-10}\inf_{L}  \mathrm{M}_X f,
\end{split}
\]
where the latter step simply holds because $4jL^{\uparrow m,+j}$ contains $L$. The proof is complete.
\end{proof}
%%%%%%%%%%%%%%%%%%%%%%%%%%%%%% PROOF PROOF PROOF

An easier estimate in the spirit of \eqref{eq:farfccc} of Lemma~\ref{lem:estimates} is available for compactly supported functions.

%%%%%%%%%%%%%%%%%%%%%%%%%%%%%% LEMMA LEMMA LEMMA
\begin{lemma} \label{lem:ovfccc}	$\displaystyle  \left\|   T_{\mathbb Q^\omega_{\mathsf{ov},L} } (f\ind_{{L}})  \right\|_{\ell^2(\omega\in \Omega_\Xi)} \lesssim C_X \inf_{L} \mathrm{M}_{X} (f\ind_{
{L}})  $.
\end{lemma}
%%%%%%%%%%%%%%%%%%%%%%%%%%%%%% LEMMA LEMMA LEMMA

%%%%%%%%%%%%%%%%%%%%%%%%%%%%%% PROOF PROOF PROOF
\begin{proof} Fixing $m\geq 0$ and $j\in\Z$, apply \eqref{e:sszygm_new} to obtain
\[
\begin{split}
\left\| T_{\mathbb P_=(L^{\uparrow m,+j})} (f\ind_L {\ind_{L^{\uparrow m}}}) \right\|_{{\ell^2(\Omega_\Xi)}} (x) &\lesssim \left(C_X\inf_{L^{\uparrow m}} \M_X (f\ind_L)\right) \left(1+\frac{\dist(L^{\uparrow m},L^{\uparrow m,+j})}{2^m\ell_L}\right)^{-\mathsf{dec}}
\\
&\lesssim  C_X 2^{-m}\langle f \rangle_{X,L}(1+|j|)^{-9}.
\end{split}
\]
A routine summation completes the proof.
\end{proof}
%%%%%%%%%%%%%%%%%%%%%%%%%%%%%% PROOF PROOF PROOF

%%%%%%%%%%%%%%%%%%%%%%%%%%%%%% SECTION SECTION SECTION
\section{Stopping collection estimates} \label{s:sce} This section contains the main multiscale part of the argument. The projection Lemma~\ref{l:proj} is used to estimate the localized $L^2$-norm of model %multipliers
operators in~\eqref{e:defmulti} restricted to tiles from $\mathbb P^{O_\Xi}$, see Proposition~\ref{p:maintech} below. In the subsequent paragraphs, 
localized, weak-type, and bilinear estimates are derived from the proposition in preparation  for the proofs of the main theorems.

%%%%%%%%%%%%%%%%%%%%%%%%%%%%%% SECTION SECTION SECTION
\subsection{Stopping collections and an $L^2$-estimate} \label{ss:stopL2}Let $I\in \mathcal D$ and $f\in  L^\infty_0(\R)$ be fixed throughout this paragraph. Also fix  $\mathcal L \subset \mathcal  D(3I)$ to be a collection of pairwise disjoint dyadic intervals whose union is $B\subset 3I$. Such {a} collection $\mathcal L$ will be referred to as a 
\emph{stopping collection}, following the terminology introduced in \cite{CCDP2018}{;} see also \cite{CDPPV22}. Let
\begin{equation} \label{e:stopcol}
\begin{split}
&\mathcal L'=\left\{L'\in \mathcal D(I):\, L'=L^{+j} \textrm{ for some } L\in \mathcal L, \, j=0,\pm1 \right\},
\\ 
& \mathcal G=\left\{G\in \mathcal D(I): \, G\not\subset 3L \textrm{ for all } L \in \mathcal L  \right\}, 
\end{split}
\end{equation}
and let $\mathcal L''$ be the collection of maximal elements in $\mathcal L'$. The collection $\mathcal G$ has  the property that
\[
\begin{split}
& G\in \mathcal G,\, L \in \mathcal L, \,\ell_G<\ell_L \implies \mathrm{dist}(L,G) \geq \ell_L ,
\\
& G\in \mathcal D(I)\setminus \mathcal G \implies G \in \mathcal D( {L''}) \textrm{ for some }  L''\in \mathcal L''.
\end{split}
\] 
Define $\mathbb P^{O_\Xi}_{\mathcal G}(I){\coloneqq} \{\mathbb P^{O_\Xi}(I): \, I_P\in \mathcal G\}$,  and partition for each $L\in \mathcal L$
\begin{equation}\label{e:gsplit}
\begin{split}
& \mathcal G = \mathcal G_0(L)\cup \bigcup_{\substack k\in \mathbb Z \setminus\{0,\pm1\}}\mathcal G_{k}(L)  {,}
\\ 
&\mathcal G_0(L)\coloneqq\{G\in \mathcal G: \, \ell_G > \ell_L\}, 
\\ 
& \mathcal G_{k}(L) \coloneqq \mathcal G \cap \mathcal D(L^{+k}), \quad k \in \mathbb Z\setminus\{0, \pm 1\}.
\end{split}
\end{equation}
Moreover,  set $\widetilde{\mathcal{G}}_k (L) \coloneqq \mathcal G_{k}(L) \setminus \{ L^{+k} \}$. Below, we turn to the estimation of the operator
\begin{equation}
\label{e:dec1.1}   T_{\mathbb P^{O_\Xi}_{\mathcal G}(I)}= T_{\mathbb P^{O_\Xi}(I)} -\sum_{L \in {\mathcal L''}} 
T_{\mathbb P^{O_\Xi}(L)}
\end{equation} 
in terms of the local quantity
\begin{equation}
\label{e:lambdf} \lambda_{f,X} \coloneqq 
\left\|f\cic{1}_{3I\setminus B}\right\|_\infty+  \sup_{L\in \mathcal L} \inf_{L} \mathrm{M}_X f,
\end{equation}
cf. \eqref{e:Orl}. When the space $X$ is fixed and understood from context in \eqref{e:lambdf}, write $\lambda_f$ instead.

%%%%%%%%%%%%%%%%%%%%%%%%%%%%%% PROPOSITION PROPOSITION PROPOSITION
\begin{proposition}\label{p:maintech}  There holds  
\begin{equation}\label{e:Tgoodloc}
 \left\|T_{\mathbb P_{\mathcal G}^{O_\Xi}(I)} f \right\|_2 \lesssim \ZZ^\star(X,\Xi)\sqrt{|I|}\lambda_{f\chi_{I}^{\mathsf{dec}}}.
\end{equation}
\end{proposition} 
%%%%%%%%%%%%%%%%%%%%%%%%%%%%%% PROPOSITION PROPOSITION PROPOSITION

%%%%%%%%%%%%%%%%%%%%%%%%%%%%%% PROOF PROOF PROOF
\begin{proof} During this proof, as no confusion may arise, we drop the superscript ${O_\Xi}$ from all occurrences and write for brevity $C_X\coloneqq \ZZ^\star(X,\Xi)$. Notice that it is enough to prove the weaker estimate
\begin{equation}\label{e:Tgood} 
\left\|T_{\mathbb P_{\mathcal G}(I)} f\right\|_2 \lesssim C_X\sqrt{|I|}\lambda_f 
\end{equation}
as in fact \eqref{e:Tgoodloc} then follows from \eqref{e:Tgood} and the localization trick {in Remark~\ref{rmrk:loc}}. The proof of the proposition is thus reduced to the proof of \eqref{e:Tgood} which involves estimating several pieces{.} Namely, after setting
\begin{equation} \label{e:upL}
\mathbb P_{\uparrow,L}\coloneqq\left\{P\in \mathbb P (I):\, I_P\in \mathcal G_0(L)\right\},
\end{equation} 
decompose
\begin{align}\label{e:dec1.2b} 
& \quad T_{\mathbb P_{\mathcal G}(I)}  \left(f\cic{1}_{\R\setminus 3I}\right)
\\ \label{e:dec1.3a}
&+ T_{\mathbb P_{\mathcal G}(I)}  \left(f\cic{1}_{3I\setminus B}\right)  
\\  \label{e:dec1.1a} 
 T_{\mathbb P_{\mathcal G}(I)} f= &
\\ \label{e:dec1.4a} & 
+ \sum_{L\in \mathcal L} T_{\mathbb P_{\uparrow,L}} \left(f\cic{1}_{L}\right)
\\  \label{e:dec1.4b} &
+ \sum_{L\in \mathcal L} \sum_{|k|\geq 2} T_{\mathbb P_{\mathcal  G_{k}(L)}}\left(f\cic{1}_L\right) .
\end{align}
Applying Lemma~\ref{l:proj} to each $f\cic{1}_L$ {and} denoting by $g_L$ the corresponding output,
we have
\begin{align}\label{e:gLs}
&\sup_{L\in \mathcal L} \langle g_L \rangle_{2,L,-}\lesssim C_X \lambda_f{,}
\\
\label{e:gL2}  &  \left\|g \coloneqq \sum_{L \in \mathcal L} g_L\right\|_2 \lesssim  C_X\lambda_f\sqrt{\sum_{L\in \mathcal L} |L|}{,}
\end{align}
where \eqref{e:gL2} follows from \eqref{e:gLs} and Lemma~\ref{l:L2sum}. It is then easy to see that
\begin{align}\eqref{e:dec1.4a}
& =  \sum_{L\in \mathcal L} T_{\mathbb P_{\uparrow,L}} \left(f\cic{1}_{L}\right)  = \sum_{L\in \mathcal L} T_{\mathbb P_{\uparrow,L}} \left(g_L\right) 
\\ \label{e:main} 
& =T_{\mathbb P_\mathcal G(I)} \left(g\right) 
\\  \label{e:dec15b}
&\quad -  \sum_{L\in \mathcal L} \sum_{|k|\geq 2 }  T_{\mathbb P_{
		{\widetilde{\mathcal G}_k(L)}}} \left(g_L\right) {.}
\end{align}
It remains to estimate the term \eqref{e:dec1.3a} and main term \eqref{e:main}, followed by the tail terms \eqref{e:dec1.2b},  \eqref{e:dec1.4b}, \eqref{e:dec15b}.

%%%%%%%%%%%%%%%%%%%%%%%%%%%%%% SECTION SECTION SECTION
\subsubsection*{Estimates for \eqref{e:dec1.3a} and \eqref{e:main}}  These terms are easily estimated by applying the $L^2$-bound for $T_{\mathbb P_{\mathcal G}(I)}$ {from Lemma~\ref{lem:ao}} in the form
\[
\left\|T_{\mathbb P_{\mathcal G}(I)} h \right\|_{L^2(\R)}\lesssim \sqrt{|I|} \langle h \rangle_{2,I,+},
\]
and then using the definition of $\lambda_f$ for the proof of \eqref{e:dec1.3a}, and \eqref{e:gLs} for the proof of \eqref{e:main}, respectively.

%%%%%%%%%%%%%%%%%%%%%%%%%%%%%% SECTION SECTION SECTION
\subsubsection*{Estimate for  \eqref{e:dec1.2b}}\label{sec:dec1.2b}
The inequality 
\[
\left\| T_{\mathbb P_{\mathcal{G}} (I)} ( f\cic{1}_{\R \setminus 3I}) \right\|_2 \lesssim  C_X \inf_{I}\M_Xf \sqrt{|I|} \lesssim  C_X {\inf_{3I}\M_Xf}\sqrt{|I|}
\]
follows by {the} triangle inequality, partitioning $\R$ and $\R \setminus 3 I$ into intervals of the form $I^{+j}$, $I^{+k}$ for $j,k \in \Z$ with $|k| \geq 2$, and using \eqref{eq:decayzyg_fin} of Lemma~\ref{l:newtail_fin}. Since
\[
\bigcup_{L\in\mathcal L} L\subset 3I,\qquad \inf_{3I}\M_X f\leq \inf_{L\in\mathcal L}\inf_L\M_X f \leq \lambda_f,
\]
the proof of \eqref{e:dec1.2b} is complete.

%%%%%%%%%%%%%%%%%%%%%%%%%%%%%% SECTION SECTION SECTION
\subsubsection*{Estimate for \eqref{e:dec1.4b}} In this paragraph, $\mathsf{dec}$  is fixed to being the one from {Lemma~\ref{l:proj}}, and   $\mathsf{dec}'= \mathsf{dec} -10$ {is good enough for Lemma~\ref{l:L2sum}}. We start by the basic estimate
\begin{equation}\label{e:dec1.4b1} \left\|\textrm{\eqref{e:dec1.4b}} \right\|_{2} \leq \sum_{|k|\geq 2} \left\|\sum_{L\in\mathcal L} F_L \right\|_2
\end{equation}
where we have set $F_L\coloneqq T_{\mathbb P_{\mathcal{G}_k(L)}} (f\ind_L)$. An application of Lemma~\ref{l:L2sum} thus tells us that 
\[
\begin{split}
\left\| \sum_{L\in\mathcal L} F_L\right\|_{2} &\lesssim \sqrt{\sum_{L\in\mathcal L}|L|} \sup_{L\in\mathcal L}\langle F_L\rangle_{2,L,-} \lesssim \sqrt{\sum_{L\in\mathcal L}|L|} \sup_{L\in\mathcal L}\sum_{j \in \mathbb{Z} } \frac{(1 + |j|)^{\mathsf{dec}'}}{\sqrt{|L|}} \left\| T_{\mathbb P_{\mathcal{G}_k(L)}}(f \ind_L) \ind_{L^{+j}}\right\|_2 .
\end{split}
\]
Fixing momentarily $L\in\mathcal L$ and $j \in \mathbb{Z}$, we apply \eqref{eq:decayzyg_fin} of Lemma~\ref{l:newtail_fin} to estimate
\[
\begin{split}
\left\| T_{\mathbb P_{\mathcal{G}_k(L)}}(f \ind_L) \ind_{L^{+j}} \right\|_2 &\lesssim C_X \sqrt{|L|} |k|^{2-\mathsf{dec}} (1 + |j|)^{2-\mathsf{dec}} \inf_L \M_Xf 
\\
&\lesssim C_X \sqrt{|L|} |k|^{2-\mathsf{dec}} (1 + |j|)^{2-\mathsf{dec}} \lambda_f
\end{split}
\]
which yields the desired estimate upon summing in $j \in \Z$, and in $k \in \mathbb{Z} \setminus \{ 0,\pm 1 \}$ in \eqref{e:dec1.4b1}.

%%%%%%%%%%%%%%%%%%%%%%%%%%%%%% SECTION SECTION SECTION
\subsubsection*{Estimate for \eqref{e:dec15b}}  In this paragraph,   $\mathsf{dec}$  is fixed to being the one from \eqref{e:gLs}{, hence from Lemma~\ref{l:proj}}, and   $\mathsf{dec}'= \mathsf{dec} -10$.  We make use of the  operator
\[
Q_{k,L} h  \coloneqq 
\chi_{L}^{-\mathsf{dec}'} T_{\mathbb P_{ \widetilde{\mathcal G}_{k}(L)}} ( \chi_{L^{+k}}^{-\mathsf{dec}} h) {\eqcolon} \chi_{L}^{-\mathsf{dec}'} \chi_{L^{+k}}^{\mathsf{dec}'}\widetilde T_{{\mathbb P_{
\widetilde{\mathcal G}_{k}(L)}}}   h {.}
\]
Note that $\|Q_{k,L}\|_{2\to 2}\lesssim (1+|k|)^{\mathsf{dec}'}$ as we have localized using $ \chi_{L} $ instead of $ \chi_{L^{+k}}$, and applied Lemma~\ref{lem:ao}. At this point

\begin{equation} \label{e:loctrick2}
\begin{split}
&\quad \left\langle   T_{\mathbb P_{\widetilde{\mathcal G}_{k}(L)}} g_L \right\rangle_{2,L,-} 
=  \frac{\left\|\chi_{L}^{-\mathsf{dec}'}  T_{\mathbb P_{\widetilde{\mathcal G}_{k}(L)}}( g_L)   \right\|_2}{\sqrt{|L|}}  
= \frac{\left\| Q_{k,\ell} \left(\chi_{L^{+k}}^{\mathsf {dec}}g_L\right)  \right\|_2}{\sqrt{|L|}} 
\\ 
& \lesssim  (1+|k|)^{\mathsf{dec}'} \left\|  \chi_{L^{+k}}^{\mathsf {dec}}\chi_{L}^{\mathsf {dec}}  \right\|_\infty  \frac{\left\|\chi_{L}^{-\mathsf{dec}}    g_L  \right\|_2}{\sqrt{|L|}}   \lesssim (1+|k|)^{-10}   \left\langle    g_L \right\rangle_{2,L,-}
\lesssim C_X (1+|k|)^{-10}\lambda_f.
\end{split}
\end{equation}
We may then apply Lemma~\ref{l:L2sum} to  $T_{\widetilde{\mathcal G}_{k}(L)} (g_L)$, thus obtaining for $|k|\geq 2$
\[
\left\|\sum_{L\in \mathcal L}   T_{\mathbb P_{\widetilde{\mathcal G}_{k}(L)}}  g_L  \right\|_2 \lesssim  C_X  |k|^{-10}  \lambda_f\sqrt{\sum_{L\in \mathcal L} |L|}
\]
whence the estimate
\[
\left\|\textrm{\eqref{e:dec15b}}  = -\sum_{k\in \mathbb Z}  \sum_{L\in \mathcal L}    T_{\mathbb P_{\widetilde{\mathcal G}_{k}(L)}} \left(g_L\right) \right\|_2 \lesssim   C_X \lambda_f\sqrt{\sum_{L\in \mathcal L} |L|}\lesssim C_X \sqrt{|I|}\lambda_f
\]
follows upon summing in $|k|\geq 2$. 

Summing the estimates for the terms \eqref{e:dec1.2b}, \eqref{e:dec1.3a}, \eqref{e:dec1.4b}, \eqref{e:main} and \eqref{e:dec15b} above concludes the proof of \eqref{e:Tgood} and with that the proof of the proposition is complete.
\end{proof}
%%%%%%%%%%%%%%%%%%%%%%%%%%%%%% PROOF PROOF PROOF

%%%%%%%%%%%%%%%%%%%%%%%%%%%%%% SECTION SECTION SECTION
\subsection{Localized estimates} Hereafter, $\mathbb Q\subset \mathbb P^{O_\Xi}$  denotes again a generic subcollection. Proposition~\ref{p:maintech} can then be reformulated as
\begin{equation} \label{e:maintechref} \left\langle T_{\mathbb Q (I)} f\right\rangle_{2,I,-} \lesssim \ZZ^\star(X,\Xi)\sup_{P\in \mathbb Q (I)} \inf_{I_P} \mathrm{M}_X \left(f\chi_I^{\mathsf{dec}}\right).
\end{equation}
Indeed, letting $\lambda\coloneqq \sup_{P\in \mathbb Q (I)} \inf_{I_P} \mathrm{M}_X (f\chi_I ^{\mathsf{dec}})$, we define the stopping collection $\mathcal L_{\mathrm{all}}$ to be the maximal elements of the collection 
\[
\left\{ L \in\mathcal D(I): \, \langle f\chi_I ^{\mathsf{dec}}\rangle_{X,L} > 2 \lambda\right\}.
\]
Note that $\inf_L \M_X\left(f\chi_I ^{\mathsf{dec}}\right)\geq 2\lambda$ for every $L\in\mathcal L_{\mathrm{all}}$ and so $\mathbb Q(I)$ coincides with $\mathbb Q_{\mathcal G}(I)$ for the choice of stopping collection $\mathcal L_{\mathrm{all}}$. Proposition~\ref{p:maintech} {and the localization trick of Remark~\ref{rmrk:loc}} then implies
\[
\langle  T_{\mathbb Q(I)}f\rangle_{2,I,-} \lesssim  \ZZ^\star(X,\Xi)\lambda_{f\chi_I ^{\mathsf{dec}}}\lesssim  \ZZ^\star(X,\Xi)\lambda,
\]
the last approximate inequality following by the definition of the stopping collection. 

Secondly, the tail estimate 
\begin{equation}\label{e:maintechtiles} \left\langle \cic{1}_{\R\setminus 3I} T_{\mathbb Q  (I)} f\right\rangle_{2,I,-} \lesssim \ZZ^\star(X,\Xi)  \inf_{I} \mathrm{M}_X \left(f\chi_I^{\mathsf{dec}}\right)
\end{equation}
can also be easily deduced from \eqref{e:maintechref} as follows. Fix $I\in \mathcal D$, let $m\geq 0$ and $J\in \mathcal D_m(I)$. Then
\begin{equation} \label{e:useless}
\begin{split}
 \left\langle \cic{1}_{\R\setminus 3I} T_{\mathbb Q_{=}  (J)} f\right\rangle_{2,I,-} &{=} \frac{ \left\|\chi_{I}^{-\mathsf{dec}} \left[ \cic{1}_{\R\setminus 3I} T_{\mathbb Q_{=}  (J)} f\right] \right\|_2}{\sqrt{|I|}} =  \frac{ \left\|\left[\chi_{I}^{-\mathsf{dec} } \chi_{J}^{\mathsf{dec} }   \cic{1}_{\R\setminus 3I}\right]  \widetilde{T}_{\mathbb Q_{=}  (J)} (f\chi_J ^{\mathsf{dec}})   \right\|_2}{\sqrt{|I|}} 
\\ 
& \leq    \left\|\chi_{I}^{-\mathsf{dec} } \chi_{J}^{\mathsf{dec} } \cic{1}_{\R\setminus 3I}\right\|_\infty\frac{\left\| \widetilde{T}_{\mathbb Q_{=}  (J)} (f\chi_J ^{\mathsf{dec}})  \right\|_2}{\sqrt{|I|}} \lesssim   2^{-m\mathsf{dec}}\frac{\left\| \widetilde{T}_{\mathbb Q_{=}  (J)} (f\chi_J ^{\mathsf{dec}})   \right\|_2}{\sqrt{|I|}}
\\ 
& \lesssim   \ZZ^\star(X,\Xi) 2^{-m\mathsf{dec}} \inf_{J} \mathrm{M}_X (f\chi_J ^{\mathsf{dec}}) \lesssim  \ZZ^\star(X,\Xi)  2^{-m(\mathsf{dec}-1)} \inf_{I} \mathrm{M}_X (f\chi_I ^{\mathsf{dec}}),
\end{split}
\end{equation} 
{where $ \widetilde{T}_{\mathbb Q_{=}  (J)} f \coloneq \chi_J ^{-\mathsf{dec}} T_{\mathbb Q_{=}  (J)} (\chi_J ^{-\mathsf{dec}} f)$.}
In the passage to the last line we applied \eqref{e:maintechref} for the collection $\mathbb Q_{=}(J)$. This estimate is summable over $J\in \mathcal D_m$ and $m\geq 0$ which completes the proof of \eqref{e:maintechtiles}. Of course, we could just as well have deduced \eqref{e:maintechtiles} directly from the tail estimates carried out in the proof of Proposition~\ref{p:maintech}.

Estimates \eqref{e:maintechref} and \eqref{e:maintechtiles} have vector-valued counterparts in the following sense. If  $\mathbb Q\subset\mathbb P^{O_\Xi} $, we have
\begin{equation}\label{e:maintechrefvv}
\begin{split}
 \left\langle   \left\|{T}_{\mathbb Q^\omega (I)} f\right\|_{\ell^2(\omega\in \Omega_\Xi)}\right\rangle_{2,I,-}& \lesssim  \ZZ^\star(X,\Xi) \sup_{P\in \mathbb Q (I)} \inf_{I_P} \mathrm{M}_X \left(f\chi_I^{\mathsf{dec}}\right)
\\
\left\langle \cic{1}_{\R\setminus 3I}\left\|{T}_{\mathbb Q^\omega (I)} f\right\|_{\ell^2(\omega\in \Omega_\Xi)}\right\rangle_{2,I,-} &\lesssim \ZZ^\star(X,\Xi)   \inf_{I} \mathrm{M}_X \left(f\chi_I^{\mathsf{dec}}\right).
\end{split}
\end{equation}
This estimates follow easily, for example by using the corresponding scalar estimates \eqref{e:maintechref} and \eqref{e:maintechtiles} and a randomization argument involving Khintchine's inequality; we omit the details.

%%%%%%%%%%%%%%%%%%%%%%%%%%%%%% SECTION SECTION SECTION
\subsection{Weak-type estimates for multiplier operators} The estimates we obtained up to this point may be easily repurposed to obtain a local weak-type estimate for the multipliers $T_{\mathbb Q}$ of \eqref{e:defmulti}, when $\mathbb Q\subset \mathbb P^{O_\Xi}$ as above.

%%%%%%%%%%%%%%%%%%%%%%%%%%%%%% PROPOSITION PROPOSITION PROPOSITION
\begin{proposition} \label{p:weaktype} There exists an absolute constant $\Theta$ such that the following holds. Let $\mathbb Q\subset \mathbb P^{O_\Xi}$. Then 
\[
 \left|\left\{ x\in I : \,\left|T_{\mathbb Q  (I^{+j})} f(x)\right|> \Theta \ZZ^\star(X,\Xi)\langle f \rangle_{X,I,+} \right\} \right|  \leq 2^{-9}  |I|
\]
for all $ I\in \mathcal D$, $|j|\leq 2$.
\end{proposition}
%%%%%%%%%%%%%%%%%%%%%%%%%%%%%% PROPOSITION PROPOSITION PROPOSITION

%%%%%%%%%%%%%%%%%%%%%%%%%%%%%% PROOF PROOF PROOF
\begin{proof} To keep the notation compact,  set again $C_X\coloneqq \ZZ^\star(X,\Xi)$ in this proof, and  restrict attention to the case $j=0$ in what follows, as the cases $0<|j|\leq 2$ follow by a completely analogous argument. Also, the statement is invariant under replacement of $f$ by $f\circ \mathrm{Sy}_I$ and $I$ by $[0,1)$. We can thus restrict ourselves to proving the case $I=[0,1)$.
Define
\[
E \coloneqq \left\{x\in I :\, \mathrm{M}_X (f\chi_I^{\mathsf{dec}})> 2^{14}\langle f \rangle_{X,I,+}  \right\}
\]	
and let $L\in \mathcal L$ be the collection of the maximal elements of $\mathcal D(I)$ such that $L \subset E $. Define also 
\[
\widetilde{E}\coloneqq \bigcup_{L\in \mathcal L} 5L.
\]
Using pairwise disjointness of $L\in \mathcal L$ and the weak-type $(X,L^{1})$-inequality of $\mathrm{M}_X$, which in dimension $1$ holds with constant $3$, let us first check that
\begin{equation}\label{e:measElam}
|\widetilde{E }| \leq 5\sum_{L\in \mathcal L} |L|\leq 5|E | \leq    \frac{15}{ 2^{14}\langle f \rangle_{X,I,+} } \|f\chi_I^{\mathsf{dec}}\|_{X} =  \frac{15 }{2^{14}} { |I|} \leq  2^{-10}  |I|.
\end{equation}
Say $P\in \mathbb Q_{\mathsf{g} }(I)  $ if $P \in\mathbb Q(I) $ and $I_P$  is not contained in $3L$ for any $L\in  \mathcal L$. Then 
\[
\mathbb Q (I) = \mathbb Q_{ \mathsf{g} }(I) \sqcup  \mathbb Q_{ \mathsf{b} }(I) , \qquad \mathbb Q_{ \mathsf{b} }(I)   \coloneqq \bigcup_{L\in \mathcal L} \bigcup_{j=0,\pm 1} \mathbb Q (L^{+j}).
\]
By construction,  and applying estimate \eqref{e:maintechref}
\begin{equation}\label{e:wtseq1}
\left\langle T_{\mathbb Q_{  \mathsf{g} } (I)} f\right\rangle_{2,I,-}\lesssim  C_X \sup_{P\in \mathbb Q_{ \mathsf{g} }(I)} \inf_{I_P} \mathrm{M}_X \left(f\chi_I^{\mathsf{dec}}\right) \lesssim   C_X  \langle f \rangle_{X,I,+}.
\end{equation} 
Furthermore, applying Lemma~\ref{l:L2sum} and estimate \eqref{e:maintechtiles},
\begin{equation}\label{e:wtseq2} 
\begin{split}
\left\|  \cic{1}_{\R\setminus \widetilde{ E }}T_{\mathbb Q_{ \mathsf{b} } (I)} f \right\|_{2} &= \left\| \sum_{L\in \mathcal L} \sum_{j=0,\pm1} \cic{1}_{\R\setminus \widetilde{ E }}  T_{\mathbb Q_{  \mathsf{b} } (L^{+j})} f\right\|_2\lesssim C_X \sqrt{|I|} \sup_{L\in \mathcal L} \inf_{L} \mathrm{M}_X f\\ &\lesssim C_X \sqrt{|I|}    \langle f \rangle_{X,I,+}.
\end{split}
\end{equation}
Let $C$ equal four times the largest of the implicit constants appearing in the inequalities \eqref{e:wtseq1} and \eqref{e:wtseq2}. Then  
\[
\begin{split}
 &\quad \left|\left\{ x\in I \setminus \widetilde{ E} : \, |T_{\mathbb Q  (I)} f(x)|>\Theta C_X\langle f \rangle_{X,I,+} \right\}  \right|  \leq \frac{1}{[\Theta C_X  \langle f \rangle_{X,I,+}]^2} \left\|   \cic{1}_
{\R\setminus \widetilde{ E }}T_{\mathbb Q  (I)}f \right\|_2^2   
\\ 
&\qquad \leq  \frac{[C C_X    \langle f \rangle_{X,I,+}]^2}{[\Theta C_X  \langle f \rangle_{X,I,+}]^2} |I| = \frac{C^2}{\Theta^2} |I| \leq 2^{-10} |I|
\end{split}
\]
provided $\Theta \geq 2^{5} C$. Combining the latter estimate with \eqref{e:measElam} returns the claim.
\end{proof}
%%%%%%%%%%%%%%%%%%%%%%%%%%%%%% PROOF PROOF PROOF

%%%%%%%%%%%%%%%%%%%%%%%%%%%%%% REMARK REMARK REMARK
\begin{remark}\label{rmrk:vvaluedwtype} The vector-valued version of the estimate in Proposition~\ref{p:weaktype} holds, namely  
\[
\left|\left\{ x\in I : \, \left\| {T}_{\mathbb Q^\omega  (I^{+j})} f(x)\right\|_{\ell^2(\omega \in \Omega_\Xi)}>  \Theta  \ZZ^\star(X,\Xi)\langle f \rangle_{X,I,+} \right\} \right|  \leq 2^{-9}  |I|.
\] 
This follows by an obvious modification of the proof of Proposition~\ref{p:weaktype} above, using the estimates in \eqref{e:maintechrefvv}.
\end{remark}
%%%%%%%%%%%%%%%%%%%%%%%%%%%%%% REMARK REMARK REMARK

%%%%%%%%%%%%%%%%%%%%%%%%%%%%%% SECTION SECTION SECTION
\subsection{A bilinear stopping estimate} This paragraph contains a strengthening of Proposition~\ref{p:maintech} into a bilinear form estimate.   Hereafter $\mathcal L$ stands for a fixed stopping collection, as in \S\ref{ss:stopL2},  $\mathcal G$ is defined as in \eqref{e:stopcol}, and the decomposition \eqref{e:gsplit} is referred.  All instances of \eqref{e:lambdf} and \eqref{e:dec1.1}, as well as the splitting \eqref{e:gsplit}, refer to  the stopping collection $\mathcal L$.

%%%%%%%%%%%%%%%%%%%%%%%%%%%%%% PROPOSITION PROPOSITION PROPOSITION
\begin{proposition}\label{p:maintechdual} $\displaystyle \left|\left \langle T_{\mathbb P^{O_\Xi}_{\mathcal G}(I)} f_1,f_2  \right\rangle \right| \lesssim |I| \prod_{j=1,2} 
	\ZZ^\star(X_j,\Xi) \lambda_{f_j, X_j}.$
\end{proposition}
%%%%%%%%%%%%%%%%%%%%%%%%%%%%%% PROPOSITION PROPOSITION PROPOSITION

%%%%%%%%%%%%%%%%%%%%%%%%%%%%%% REMARK REMARK REMARK
\begin{remark} Before the proof of the general case, note that 
if $X_2=L^2$ the proposition follows directly from Proposition~\ref{p:maintech}. Indeed\[
\begin{split}
\left|\left\langle T_{\mathbb P^{O_\Xi}_{\mathcal G}(I)} f_1,f_2  \right\rangle \right| & \leq|I|\left\langle T_{\mathbb P^{O_\Xi}_{\mathcal G}(I)} f_1\right\rangle_{2,I,-}  \left\langle  f_2  \right\rangle_{2,I,+} = \sqrt{|I|}
\left\| \widetilde{T}_{\mathbb P^{O_\Xi}_{\mathcal G}(I)} f_1\right\|_{2} \left\langle  f_2  \right\rangle_{2,I,+} \\ & \lesssim |I |	\ZZ^\star(X_1,\Xi)   \lambda_{f_1, X_1}\lambda_{f_2, L^2}
\end{split}
\] 
applying the localization trick {in Remark~\ref{rmrk:loc}} in the first equality, and Proposition~\ref{p:maintech} with $X=X_1,f=f_1$ to pass to the second line.
\end{remark}
%%%%%%%%%%%%%%%%%%%%%%%%%%%%%% REMARK REMARK REMARK

%%%%%%%%%%%%%%%%%%%%%%%%%%%%%% PROOF PROOF PROOF
\begin{proof}[Proof of Proposition~\ref{p:maintechdual}]   As before, the superscript $O_\Xi$ is dropped from all instances in this proof, and $C_{X_j}\coloneqq 	\ZZ^\star(X_j,\Xi).$ The basic decomposition is
\[
\begin{split}
\left\langle  T_{\mathbb P_{\mathcal G}(I)} f_1,f_2  \right  \rangle =  \left\langle T_{\mathbb P_{\mathcal G}(I)} f_1,\ind_{\R\setminus B} f_2   \right\rangle  +   \left\langle f_1 \ind_{\R\setminus B} ,  T_{\mathbb P_{\mathcal G}(I)}^\star \left( \ind_{  B} f_2  \right) \right\rangle +
 \left\langle  T_{\mathbb P_{\mathcal G}(I)}\left( f_1 \ind_{ B} \right),    \ind_{  B} f_2   \right\rangle.
\end{split}
\]
The first two terms are estimated via the  exact same strategy, relying on the further splitting
\begin{equation}\label{e:maintechdual0}
\left\langle T_{\mathbb P_{\mathcal G}(I)} f_1,\ind_{\R\setminus B} f_2   \right\rangle = \left\langle T_{\mathbb P_{\mathcal G}(I)} f_1,\ind_{3I\setminus B} f_2   \right\rangle + \left\langle
T_{\mathbb P_{\mathcal G}(I)} f_1,\ind_{\mathbb R\setminus 3I} f_2   \right\rangle.
\end{equation}
For the first term in the above display we have
\begin{equation}
	\label{e:maintechdual1}
	\left| \left\langle T_{\mathbb P_{\mathcal G}(I)} f_1,\ind_{3I\setminus B} f_2   \right\rangle \right| % \leq 
	{\lesssim} \sqrt{|I|}\left\| T_{\mathbb P_{\mathcal G}(I)} f_1 \right\|_2 \|f_2 \ind_{3I\setminus B}\|_\infty  \lesssim |I| C_{X_1} \lambda_{f_1,X_1}\lambda_{f_2,X_2},
\end{equation}
with a straightforward application of Proposition~\ref{p:maintech} and  using Definition~\eqref{e:lambdf}.

For the second term we partition $\R$ and $\R \setminus 3{I}$ into intervals of the form $ I^{+j}$, $ I^{+k}$ for $j,k \in \Z$ with $|k| \geq 2$, and apply \eqref{e:newtail2_fin} of Lemma~\ref{l:newtail_fin}. It is not difficult to see that the summation returns a control of the second term {in} \eqref{e:maintechdual0} by the correct right hand side. It remains to handle $\left\langle  T_{\mathbb P_{\mathcal G}(I)}\left( f_1 \ind_{ B} \right),    \ind_{  B} f_2  \right \rangle$. 
Arguing as in \eqref{e:dec1.4a}-\eqref{e:dec1.4b} 
\[
T_{\mathbb P_{\mathcal G}(I)}(f\ind_B)= \sum_{L\in\mathcal L} T_{\mathbb P_{\uparrow,L}} \left(f\cic{1}_{L}\right) +\sum_{L\in\mathcal L} \sum_{|k|\geq 2} T_{\mathbb P_{\mathcal  G_{k}(L)}}(f\ind_L).
\]	
Applying this to $f=f_1$ yields
\[
\left\langle  T_{\mathbb P_{\mathcal G}(I)}\left( f_1 \ind_{ B} \right),    \ind_{  B} f_2  \right \rangle = \sum_{L\in\mathcal L} \left\langle T_{\mathbb P_{\uparrow,L}} \left(f\cic{1}_{L}\right) , \ind_{  B} f_2  \right \rangle +\sum_{L\in\mathcal L} \sum_{|k|\geq 2} \left\langle T_{\mathbb P_{\mathcal  G_{k}(L)}}(f\ind_L) , \ind_{  B} f_2  \right \rangle .
\]
For the first term we argue as in Lemma~\ref{lem:ovfccc} with \eqref{e:sszygm_new} replaced by \eqref{e:newtail2}. For the second term we argue as in the proof of \eqref{eq:farfccc} of Lemma~\ref{lem:estimates} with \eqref{e:sszygm_fin} replaced by \eqref{e:newtail2_fin}, completing the proof.
\end{proof}
%%%%%%%%%%%%%%%%%%%%%%%%%%%%%% PROOF PROOF PROOF

 %%%%%%%%%%%%%%%%%%%%%%%%%%%%%% SECTION SECTION SECTION
 \section{Proof of Theorem~\ref{thm:roughsf}} \label{s:thmroughsf} By virtue of the considerations in \S\ref{ss:disc}, it suffices to prove the same estimate of Theorem~\ref{thm:roughsf} for the model square function $\left\|T_{\mathbb P^{\omega}} f  \right\|_{\ell^2(\omega \in \Omega_\Xi)}$, cf.\ \eqref{eq:modelr}. We do so by reducing to the localized estimate of the next proposition.

%%%%%%%%%%%%%%%%%%%%%%%%%%%%%% PROPOSITION PROPOSITION PROPOSITION
\begin{proposition} \label{p:modelr} Let ${I_0}\in \mathcal D$. Given any $f\in  L^\infty_0(\R)$, there exists a sparse collection $\mathcal S$ with the property that
\label{p:rough}
\[
\cic{1}_{I_0}\left\| T_{\mathbb P^{\omega}(3{I_0})}f \right\|_{\ell^2(\omega \in \Omega_\Xi)} \lesssim \ZZ^\star (X,\Xi){ \mathcal S}_{X,1} f
\]
pointwise almost everywhere. The implicit constant is absolute.
\end{proposition}
%%%%%%%%%%%%%%%%%%%%%%%%%%%%%% PROPOSITION PROPOSITION PROPOSITION  

This reduction is carried out in \S\ref{ss:redr} while the proof of Proposition~\ref{p:modelr} is given in \S\ref{ss:pfpropr}.

%%%%%%%%%%%%%%%%%%%%%%%%%%%%%% SECTION SECTION SECTION
\subsection{Reduction to Proposition~\ref{p:modelr}} \label{ss:redr}  Fix a function $f\in L^\infty _0 (\R)$. As $\mathcal D$ is a standard dyadic grid,  we may find $I_0\in\mathcal D$
 with the property that $\supp f\subset (1+3^{-1})I_0$. The reduction consists in handling the square sum of the tail terms
\[
T_{\mathbb P^\omega} ^{\mathrm{out}}f\coloneqq T_{\mathbb P^\omega}f-\ind_{5I_0}T_{\mathbb P^\omega(3{I_0})}f, \qquad \omega \in \Omega_\Xi,
\]
by means of the following lemma.
%%%%%%%%%%%%%%%%%%%%%%%%%%%%%% LEMMA LEMMA LEMMA
\begin{lemma}\label{l:tails} Let $I_0\subset \R$ be a finite interval and $f\in L^\infty _0  (\R)$ with $\supp f \subset {(1+3^{-1})I_0} $. Then
\[
\left\|{T}_{\mathbb P^\omega} ^\mathrm{out} f \right\|_{\ell^2(\omega\in\Omega_\Xi)}\lesssim \ZZ^{\star}(X,\Xi) \langle f \rangle_{X,3I_0}.
\]
\end{lemma}
%%%%%%%%%%%%%%%%%%%%%%%%%%%%%% LEMMA LEMMA LEMMA

%%%%%%%%%%%%%%%%%%%%%%%%%%%%%% PROOF PROOF PROOF
\begin{proof}  Firstly, the estimate 
\[
\left\|\ind_{\R\setminus 5I_0} {T}_{\mathbb P^\omega (3I_0)}f \right\|_{\ell^2(\omega\in \Omega_\Xi)}\lesssim   \ZZ^\star(X,\Xi) \inf_{I_0} \M_X(f )\lesssim \ZZ^\star(X,\Xi) \langle f \rangle_{X,3I_0}
\]
is a consequence of \eqref{eq:2asszyg} of Lemma~\ref{lem:estimates}. {It will thus suffice to control the $\ell^2(\omega\in \Omega_\Xi)$-norm of} %the vector whose entries are given with the help of 
\[
T_{\mathbb P^\omega}f- T_{\mathbb P^\omega(3{I_0})}f =T_ {\mathbb P_{\mathsf{far},I_0} ^\omega}f +T_{\mathbb P_{\mathsf{ov},I_0} ^\omega }f,\qquad \omega\in \Omega_\Xi,
\]
where the equality relies upon \eqref{e:splitsf}.
Now the estimates
\[
{\left\| T_ {\mathbb P_{\mathsf{far},I_0} ^\omega}f \right\|_{\ell^2(\omega\in \Omega_\Xi)}, \left\| T_ {\mathbb P_{\mathsf{ov},I_0} ^\omega}f \right\|_{\ell^2(\omega\in \Omega_\Xi)}} \lesssim \ZZ^\star(X,\Xi) \langle f \rangle_{X,3I_0}
\]
follow by \eqref{eq:farfccc} of Lemma~\ref{lem:estimates} and Lemma~\ref{lem:ovfccc}, respectively. 
\end{proof}
%%%%%%%%%%%%%%%%%%%%%%%%%%%%%% PROOF PROOF PROOF

With the estimates for the tails being taken care of by Lemma~\ref{l:tails}, the proof of Theorem~\ref{thm:roughsf} reduces to proving the sparse estimate for the operator $f\mapsto \|\ind_{5I_0} {T}_{\mathbb P^\omega (3I_0)}f\|_{\ell^2(\omega \in \Omega_\Xi)}$. To that end we write
\[
\ind_{5I_0} T_{\mathbb P^\omega(3I_0)}=\sum_{0\leq |j|\leq 2} \ind_{I_0 ^{+j}} \widetilde{T}_{\mathbb P^\omega ({3} I_0 ^{+j})} f+\sum_{|j|=1} \ind_{I_0 ^{+j}} \widetilde{T}_{\mathbb P^\omega (I_0 ^{-j})}f+\sum_{
|j|=2} \ind_{I_0 ^{+j}}  \widetilde{T}_{\mathbb P^\omega (I_0\cup I_0 ^{{-} \mathrm{sgn}(j)} )}f
\]
where $\widetilde T$ is an operator of the form \eqref{e:Ttil} given by possibly different wave packets than the ones defining $T$, allowing the choices $\upphi_P,\psi_P\equiv 0$ for $P\in\mathbb P$. The first summand in the display above is estimated in $\ell^2(\omega\in \Omega_\Xi)$ by a constant multiple of  $\ZZ^\star(X,\Xi) \mathcal S_{X,1}f$ by appealing to Proposition~\ref{p:modelr}. The last two summands are estimated   by a constant multiple of  $ \ZZ^\star(X,\Xi) \langle f\rangle_{X,3I_0}$ using an appropriate modification of \eqref{eq:2asszyg} of Lemma~\ref{lem:estimates} after noting that
\[
|j|=1\implies \dist(I_0 ^{+j},I_0 ^{-j})\geq |I_0|,\qquad |j|=2\implies \dist(I_0 ^{+j},I_0\cup I_0 ^{{-} \mathrm{sgn}(j)})\geq |I_0|,
\]
and the reduction is complete.

%%%%%%%%%%%%%%%%%%%%%%%%%%%%%% SECTION SECTION SECTION
\subsection{Proof of Proposition~\ref{p:modelr}} \label{ss:pfpropr}  As usual, we adopt the shorthand $C_{X}\coloneqq \ZZ^\star(X,\Xi)$. The main step of the proof is carried out in the lemma below.

%%%%%%%%%%%%%%%%%%%%%%%%%%%%%% LEMMA LEMMA LEMMA
\begin{lemma} \label{l:itersp1} Let $I \in \mathcal D$. Then there exists a pairwise disjoint collection $\mathcal L(I)\subset \mathcal D(I)$ with the properties that
\begin{align}
\label{e:iterative1} 
&\left|\cic{1}_{I}\left\|  T_{\mathbb P^{\omega}(3{I})}f \right\|_{\ell^2({\omega \in \Omega_\Xi})} -  \sum_{L \in \mathcal L(I)}  \cic{1}_{L}\left\|T_{\mathbb P^{\omega}(3{L})}f \right\|_{\ell^2({\omega \in \Omega_\Xi})}\right|\leq KC_X   \left\langle f \right \rangle_{X,I,+} \cic{1}_I, 
 \\
\label{e:iterative2} &
\sum_{L \in \mathcal L(I) } |L| \leq 2^{-4} |I|.
\end{align}
\end{lemma}
%%%%%%%%%%%%%%%%%%%%%%%%%%%%%% LEMMA LEMMA LEMMA

%%%%%%%%%%%%%%%%%%%%%%%%%%%%%% PROOF PROOF PROOF
\begin{proof} For brevity, write  $\lambda  \coloneqq \left\langle f \right \rangle_{X,I,+}$. Below, we plan to apply \eqref{e:splitsf} with $\mathbb Q\coloneqq \mathbb P^{O_\Xi}(3I)$  for some $L\in \mathcal D(3I)$. This reads
\begin{equation}
\label{e:splitsfbis}
\begin{split} \mathbb Q\coloneqq \mathbb P^{O_\Xi}(3I)=
{\mathbb P^{O_\Xi}(3L)}\sqcup  \mathbb Q_{\mathsf{far},L}  \sqcup \mathbb Q_{\mathsf{ov},L} .
\end{split}
\end{equation}
Note that \eqref{eq:2sszyg} in Lemma~\ref{lem:estimates} and Lemma~\ref{l:3sszyg} may be legitimately appealed to,  as those estimates hold for generic $\mathbb Q \subset   \mathbb P^{O_\Xi} $, {to} which \eqref{e:splitsf} is applied.  To begin the proper proof, start with defining the set
\begin{equation}\label{e:roughset1} 
E\coloneqq \left\{x\in I:\, \left\|  T_{\mathbb Q^\omega }f (x) \right\|_{\ell^2(\Omega_\Xi)}> 3\Theta C_X \lambda  \right\}  \cup\left\{x\in 3I:\, \mathrm{M}_X (f\chi_I^{\mathsf{dec}})(x)>  2^{15} \lambda\right\},  
\end{equation}
{for some sufficiently large constant $\Theta$} and notice that
\begin{equation}\label{e:roughset2}
|E| \leq \sum_{|j|\leq 1 } \left| \left\{x\in I:\, \left\|  T_{\mathbb P^{\omega}(I^{+j})}f(x)  \right\|_{\ell^2(\Omega_\Xi)} > 
\Theta C_X \lambda  \right\} \right|  +2^{-9} |I|\leq 3\cdot 2^{-9} |I| + 2^{-9} |I| = 2^{-7}|I|
\end{equation}
where the first bound is obtained by the maximal theorem and the second by applying Proposition~\ref{p:weaktype} in the form of Remark~\ref{rmrk:vvaluedwtype}. 
Let $\mathcal{L}$ be the collection of the maximal elements $L$ of $\mathcal D(3I)$ such that
\begin{equation}
\label{e:roughset3} |L\cap E|> 2^{-3} |L|.
\end{equation}
The elements of $\mathcal L$ are pairwise disjoint and contained in the set  $\widetilde E $ where the dyadic maximal function of $\cic{1}_E$ is larger than $2^{-3}$, whence the  packing condition \eqref{e:iterative2}.  
By maximality, $|L\cap E|\leq 2^{-2} |L|$ for each $L\in \mathcal L$. It follows that for each {$L\in \mathcal{L}$}, the set $F_L\coloneqq L\setminus E$ has the properties
\begin{align}
\label{e:iterative3a} & |F_L| \geq 3\cdot 2^{-2} |L|, 
\\ 
 \label{e:iterative3} &\sup_{F_L} \mathrm{M}_X   (f\chi_I^{\mathsf{dec}}) \leq  2^{15} \lambda,
\\ 
 \label{e:iterative4} &\sup_{F_L} \left\| T_{\mathbb Q^\omega}f \right\|_{\ell^2(\Omega_\Xi)}  \leq 3  \Theta C_X  \lambda.
\end{align}
To further refine $F_L$ define  for each $L\in \mathcal L$
\begin{equation}\label{e:roughset4} 
E_L\coloneqq \left\{x\in L:\, \left\|  T_{\mathbb P^{\omega}(3L)}f (x) \right\|_{\ell^2(\Omega_\Xi)}> 3\Theta C_X  \left\langle f \right \rangle_{X,L,+} \right\}.  
\end{equation}
Another application of   Proposition~\ref{p:weaktype} {in the form of Remark~\ref{rmrk:vvaluedwtype}} tells us that $|E_L|\leq {3 \cdot} 2^{-9}|L|$, therefore $G_L \coloneqq F_L \setminus E_L $ satisfies also 
\begin{align}
\label{e:iterative3a+}  
|G_L|& \geq   2^{-1} |L|, 
\\ 
\label{e:iterative4+} 
\sup_{G_L} \left\|  T_{\mathbb P^{\omega}(3L)}f \right\|_{\ell^2(\Omega_\Xi)} &\leq 3  \Theta C_X  \left\langle f \right \rangle_{X,L,+} \leq K C_X \lambda
\end{align}
where $K$ is an absolute implicit constant and the last bound is obtained via the control
\[
\left\langle f \right \rangle_{X,L,+}  \leq \inf_{L}   \mathrm{M}_X (f\chi_L^{\mathsf{dec}})  \leq \sup_{F_L}  \mathrm{M}_X (f\chi_I^{\mathsf{dec}}) \lesssim \lambda.
\]
For any   $y_L\in  G_L$,  comparing with  \eqref{e:splitsfbis}, 
\begin{equation}\label{e:iterative6}
\begin{split}
\left\|T_{\mathbb Q_{\mathsf{ov},L}^\omega}f(y_L) \right\|_{\ell^2(\Omega_\Xi)} &\leq \left\| T_{\mathbb Q^\omega}f(y_L) \right\|_{\ell^2(\Omega_\Xi)}  + \left\|  T_{\mathbb P^\omega(3L)}f(y_L) \right\|_{\ell^2(\Omega_\Xi)} + \left\|  T_{\mathbb Q_{\mathsf{far}, L}^\omega}f(y_L) \right\|_{\ell^2(\Omega_\Xi)}
\\ 
&\leq KC_X \lambda 
\end{split}
\end{equation}
for some absolute constant $K$, as the first term on the right hand side of \eqref{e:iterative6} is controlled by \eqref{e:iterative4}, the central term  is controlled by \eqref{e:iterative4+}, and the  rightmost term is bounded by \eqref{eq:2sszyg} of Lemma~\ref{lem:estimates}.

We are ready to prove the sparse bound \eqref{e:iterative1}. First, note that the intervals $\mathcal L$ are contained in and cover the set $\widetilde E \supset E$. Suppose $x\in I\setminus \widetilde E$. Then $x\not \in E$ as well and 
\begin{equation}
\begin{split}
\label{e:iterative1b} \left|
\cic{1}_{I}\left\|  T_{\mathbb P^{\omega}(3{I})}f \right\|_{\ell^2(\Omega_\Xi)} -  \sum_{L \in \mathcal L(I)}  \cic{1}_{L}\left\|T_{\mathbb P^{\omega}(3{L})}f \right\|_{\ell^2(\Omega_\Xi)} \right|= \cic{1}_{I}\left\|  T_{\mathbb P^{\omega}(3{I})}f \right\|_{\ell^2(\Omega_\Xi)} \leq 3\Theta C_X \lambda,
\end{split}
\end{equation}
by {the} definition of $E$, complying with \eqref{e:iterative1}. Next,  assume $x\in \widetilde E$ so that  $x\in L$ for some unique  $L \in \mathcal L$. 
In this case, 
\begin{equation}
\label{e:iterative5}  
\begin{split}    
&\quad 
\left|\left\|  T_{\mathbb P^{\omega}(3{I})}f \right\|_{\ell^2(\Omega_\Xi)} -   \cic{1}_{L}\left\|T_{\mathbb P^{\omega}(3{L})}f \right\|_{\ell^2(\Omega_\Xi)}   \right|
\leq     \left\|  Z_{\mathbb Q_{\mathsf{ov},L}^\omega}f(x)  \right\| _{\ell^2(\Omega_\Xi)}+\left\|  T_{\mathbb Q_{\mathsf{far},L}^\omega }f(x)  \right\| _{\ell^2(\Omega_\Xi)} 
\\ 
& \leq \left\|  Z_{\mathbb Q_{\mathsf{ov},L}^\omega}f(x)  \right\| _{\ell^2(\Omega_\Xi)}+ KC_X \lambda  
\\   
& \leq \left[ \left\|  Z_{\mathbb Q^\omega_{\mathsf{ov},L}}f(x)  -   Z_{\mathbb Q^\omega_{\mathsf{ov},L}}f(y_L) \right\| _{\ell^2(\Omega_\Xi)}+\left\|  Z_{\mathbb Q^\omega_{\mathsf{ov},L}}f(y_L)  \right\| _{\ell^2(\Omega_\Xi)}\right] + KC_X \lambda 
\\ 
& \leq  KC_X\lambda.
\end{split}
\end{equation}
In the chain above, we have  applied \eqref{e:splitsfbis}, used  equality \eqref{e:Ttil2} and {the} triangle inequality for the first bound, {and} 	then applied \eqref{eq:2sszyg} of Lemma~\ref{lem:estimates} to control {the term} $\vec T_{\mathbb Q_{\mathsf{far},L}}f(x)$ and pass to the second line. At that point we picked any $y_L\in  G_L$, used {the} triangle inequality, applied Lemma~\ref{l:3sszyg} to control the difference term and applied \eqref{e:iterative6} together with the   equality \eqref{e:Ttil2}  to control the inserted term. The above inequality completes the proof of \eqref{e:iterative1} and thus the proof of the iterative lemma.
\end{proof}
%%%%%%%%%%%%%%%%%%%%%%%%%%%%%% PROOF PROOF PROOF

With Lemma~\ref{l:itersp1} in hand, we are ready to complete the proof of Proposition~\ref{p:modelr}. The first step is the following inductive construction of a pairwise disjoint collection $\mathcal S_n\subset \mathcal D$ for each $n\geq 0$. Begin with $\mathcal S_0=\{I_0\}$. For $n\geq 0$,  suppose $\mathcal S_{n}$ has already been constructed. For each $I\in \mathcal S_n$, apply Lemma~\ref{l:itersp1} to $I$. This application returns the pairwise disjoint collection $\mathcal L(I)$, and consequently we define
\[
\mathcal S_{n+1}(I)\coloneqq \mathcal L(I), \quad I \in \mathcal S_n, \qquad \mathcal S_{n+1}\coloneqq \bigcup_{I\in \mathcal S_n} \mathcal S_{n+1}(I),
\]
completing the inductive step. Leveraging \eqref{e:iterative2}, it is not difficult to see that $\mathcal S\coloneqq \bigcup_{n\geq 0} \mathcal S_n$ is a sparse collection, with major  pairwise disjoint subsets
\[
E_I\coloneqq I \setminus \bigcup_{J\in \mathcal S_{n+1}(I)} J, \qquad I \in \mathcal S_n, \quad n\geq 0.
\]
Setting
\[
\Delta_I(x) \coloneqq \cic{1}_{I}(x)\left\| T_{\mathbb P^{\omega}(3{I})}f(x) \right\|_{\ell^2( \Omega_\Xi)} - \sum_{J\in \mathcal S_{n+1}(I)}   \cic{1}_{J}(x)\left\|T_{\mathbb P^{\omega}(3{J})}f(x) \right\|_{\ell^2(  \Omega_\Xi)}, \quad x\in \mathbb R,
\]
estimate \eqref{e:iterative1} may instead be rewritten as
\[
\begin{split}
&\left|\Delta_I \right|\leq KC_X   \left\langle f \right \rangle_{X,I,+} \cic{1}_I, \qquad I \in \mathcal S_n.
\end{split}
\]
Telescoping and applying the above estimate, we obtain for each $n\geq 0$ 
\begin{equation}
\label{e:limarg0}
\begin{split}&\quad \left|
\cic{1}_{I_0}\left\| T_{\mathbb P^{\omega}(3{I_0})}f \right\|_{\ell^2( \Omega_\Xi)} - \sum_{J\in \mathcal S_{n+1}}   \cic{1}_{J}\left\|T_{\mathbb P^{\omega}(3{J})}f \right\|_{\ell^2(  \Omega_\Xi)}\right| =  
\left|\sum_{k=0}^n {\sum_{I \in \mathcal S_k}}  \Delta_I \right| \leq \sum_{k=0}^n \sum_{I \in \mathcal S_k} \left| \Delta_I \right|
\\ 
& \leq  K C_X  \sum_{k=0}^n \sum_{I \in \mathcal S_k}    \left \langle f \right \rangle_{X,I,+} \cic{1}_I \eqqcolon K C_X \mathcal S_{X,1,+}  f.
\end{split}
\end{equation}
As $f\in L^\infty_0(\R)$, we have that $\mathcal S_{X,1,+}  f$ belongs to $L^{1,\infty}(\mathbb R)$ and is finite almost everywhere. Assuming for the moment that up to a subsequence
\begin{equation}\label{e:limarg}
\lim_{n\to \infty} F_n(x)=0 \quad  \textrm{a.e.} \, x\in \mathbb R,\qquad  F_n\coloneqq \sum_{J\in \mathcal S_{n}}    \cic{1}_{J}\left\|T_{\mathbb P^{\omega}(3{J})} f \right\|_{\ell^2(  \Omega_\Xi)}
\end{equation}
we get, by passing to the limit in \eqref{e:limarg0}, the intermediate sparse estimate
\begin{equation}\label{e:intspars1}
\cic{1}_{I_0}\left\| T_{\mathbb P^{\omega}(3{I_0})}f \right\|_{\ell^2(\omega \in \Omega_\Xi)} \lesssim {\ZZ^\star(X,\Xi)}{ \mathcal S}_{X,1,+} f.
\end{equation}
Subsequently, \eqref{e:intspars1} upgrades to the claim of Proposition~\ref{p:modelr}, with a possibly different sparse collection, by appealing to \cite{CR}*{Theorem~A {and Corollary A.1}}. It remains to prove \eqref{e:limarg}. The easiest way to do so is to prove that $\|F_n\|_2\to 0$ and pass to a subsequence. To see this, note that the intervals $J\in \mathcal S_n$ are pairwise disjoint by construction, as revealed by an easy induction argument. The localization trick then yields
\[
\|F_n\|_2^2 \leq \sum_{J \in \mathcal S_n} \left\|\left\|T_{\mathbb P^{\omega}(3{J})} f \right\|_{\ell^2(  \Omega_\Xi)}\right\|_{2}^2 \lesssim \sum_{J \in \mathcal S_n}  |J| \langle f\rangle_{2,J,+}^2 \lesssim \|f\|_\infty^2  \sum_{J \in \mathcal S_n}  |J|  \lesssim 2^{-4n}|I_0|\|f\|_\infty^2
\]
where the last inequality is also easily proved by induction on $n$. Therefore \eqref{e:limarg} holds, up to a subsequence, and the proof of Proposition~\ref{p:modelr} is complete.

%%%%%%%%%%%%%%%%%%%%%%%%%%%%%% SECTION SECTION SECTION 
\section{Proof of Theorem~\ref{thm:smoothsf}} \label{s:thmsmoothsf}  The proof of Theorem~\ref{thm:smoothsf} consists of the combination of an abstract Carleson embedding theorem generalizing \cite{DPF+RDF}*{Prop.\ 2.4} with an application of Proposition~\ref{p:maintech}. The first tool is developed in \S\ref{ss:gce} while the {proof of Theorem~\ref{thm:smoothsf} is concluded in the} final paragraph  \S\ref{ss:ptb}.

%%%%%%%%%%%%%%%%%%%%%%%%%%%%%% SECTION SECTION SECTION 
\subsection{Generalized Carleson embedding} \label{ss:gce}{For a}  fixed dyadic grid $\mathcal D$ {and a collection} $\mathcal I\subset \mathcal D$, the  $\mathcal I$-\emph{balayage}  of  the sequence of complex numbers $\mathsf{a}=\{a_Q:Q\in \mathcal D\}$ is
\begin{equation}
\label{e:balaspars}
A_{\mathcal I}[\mathsf{a}]\coloneqq \sum_{Q\in \mathcal I} a_Q \cic{1}_Q.   
\end{equation} 
The definition of {a} subordinated Carleson sequence \cite{DPF+RDF}*{eq.~(2.3)} can be modified as follows. {Firstly}, {let us call} $\mathcal G(I)\subset \mathcal D(I)$  \emph{major} if the maximal dyadic elements $\mathcal M (I)$ of $\mathcal D(I)\setminus \mathcal G(I)$ satisfy
\[
\sum_{L\in \mathcal M (I)} |L| \leq \frac{1}{4}|I|. 
\]
Observe that, necessarily, any $\mathcal{G}(I)$ {which is} major contains $I$. {Fixing an Orlicz space $X$, say that $\mathsf{a}$ is a \emph{generalized Carleson sequence subordinated to $f$} if}
\begin{equation}\label{e:packing}\frac{1 }{|I|} \inf_{\substack{\mathcal G(I)\subset \mathcal D(I) \\ \mathcal G(I)\textrm{ major}}}
\sum_{J\in \mathcal G(I) } |J||a_J| \leq C {\langle f\rangle_{X,I,+}^2}
\end{equation}
{with $C>0$ a numerical constant which is independent of $I\in \mathcal D$ and $f$}. The least constant $C$ such that \eqref{e:packing} holds is denoted by $\|\mathsf{a}\|_{\mathcal D}$ and termed the \emph{generalized Carleson norm} of $\mathsf{a}$. 
A minor {modification of} the proof of \cite{DPF+RDF}*{Prop.~2.4} yields the following result.

%%%%%%%%%%%%%%%%%%%%%%%%%%%%%% PROPOSITION PROPOSITION PROPOSITION
\begin{proposition}  \label{p:balayage}  For $f\in L^\infty _0 (\R)$  and  a generalized  Carleson sequence  $\mathsf a$ subordinated to $f$,  there exists a sparse collection $\mathcal J$ of intervals with the property  
\begin{equation}\label{e:spros2}
A_{\mathcal D}[\mathsf{a}] \lesssim \|\mathsf{a}\|_{\mathcal D}  \sum_{J \in \mathcal J} \cic{1}_J \langle f\rangle_{X,J}^2  
\end{equation}
pointwise almost everywhere.
\end{proposition}
%%%%%%%%%%%%%%%%%%%%%%%%%%%%%% PROPOSITION PROPOSITION PROPOSITION

%%%%%%%%%%%%%%%%%%%%%%%%%%%%%% PROOF PROOF PROOF
\begin{proof} We follow closely the proof of \cite{DPF+RDF}*{Prop.~2.4}. We fix a function $f$ with compact support and we select an interval $Q\in \mathcal{D}$ such that $\supp f \subset (1+5^{-1})Q$. Let $\mathsf a$ be a fixed {generalized} Carleson sequence subordinated to $f$ and let $\mathcal{G}^\ast(I)$ be a fixed major collection of intervals    that realizes the infimum in \eqref{e:packing} within a factor of two. 
%Observe that the collection $\mathcal{G}^\ast(I)$ can always {be} taken to be convex. 
Denote by $\mathcal{M}^\ast (I) $ the set of maximal dyadic intervals in $\DD (I)\setminus \mathcal{G}^\ast (I)$. Following \cite{DPF+RDF} we {have the basic estimate}
\begin{equation} \label{e:splitting_as_in_RDF}
 A_\DD {[\mathsf{a}]} \leq \sum_{|j|\leq 1} A_{\DD (Q^{+j})} {[\mathsf{a}]} +\sum_{\substack{k\geq 1 \\ |j|\leq 1}} a_{Q^{\uparrow k, +j}} \ind_{Q^{\uparrow k, +j}} + \sum_{\substack{ k\geq 0 \\ 2 \leq |j|\leq 3}} A_{\DD (Q^{\uparrow k, +j})} {[\mathsf{a}]},
\end{equation}
arising from the decomposition
\begin{equation}
\DD\subseteq \bigg( \bigcup_{\vert j\vert \leq 1} \DD(Q^{+j}) \bigg) \cup \bigg( \bigcup_{\substack{k\geq 1 \\ \vert j \vert\leq 1}} 
  \big\{ Q^{\uparrow k, +j} \big\} \bigg)\cup \bigg( \bigcup_{\substack{k\geq 0 \\ 2\leq \vert j \vert \leq 3}} \DD(Q^{\uparrow k, +j})\bigg).
\end{equation}

The second and third terms in \eqref{e:splitting_as_in_RDF} correspond to the tails and can be treated as in \cite{DPF+RDF} by accounting for the obvious modifications, obtaining
\[
\sum_{\substack{k\geq 1 \\ |j|\leq 1}} a_{Q^{\uparrow k, +j}} \ind_{Q^{\uparrow k, +j}} + \sum_{\substack{ k\geq 0 \\ 2 \leq |j|\leq 3}} A_{\DD (Q^{\uparrow k, +j})} {[\mathsf{a}]} \lesssim \|\mathsf{a}\|_{\mathcal D}\sum_{I\in\mathcal{R}(Q)}\langle f \rangle^2_{X,I}\ind_{I},
 \]
where $\mathcal{R}(Q)$ is a suitabl{y} constructed sparse collection.

Hence we are left to deal with the main term, namely with the first term in \eqref{e:splitting_as_in_RDF}. Following \cite{DPF+RDF}, we want to control it by the intermediate estimate
\begin{equation}\label{e:intermediate_estimate_as_in_RDF}
A_{\DD(J)} {[\mathsf{a}]} \lesssim \| \mathsf{a}\|_{\mathcal D} \sum_{I \in \mathcal{R}(J)} \langle f \rangle^2_{X,I,+} \ind_I
\end{equation}
for each $J\in \{ Q^{+j}: \, |j|\leq 1 \}$, where $\mathcal{R}(J)$ is another suitably constructed sparse collection. From \eqref{e:intermediate_estimate_as_in_RDF}  we deduce a sparse bound 
\[
A_{\DD(J)} {[\mathsf{a}]} \lesssim \| \mathsf{a}\|_{\mathcal D} \sum_{I\in \mathcal{R}'(J)} \langle f \rangle^2_{X,I}\ind_I
\]
for some possibly different sparse collection $\mathcal{R}'(J)$ by the very same arguments in \cite{DPF+RDF} noticing that \cite{CR}*{Theorem~A {and Corollary~A.1}} remain true if we replace the usual local averages with local averages in the Orlicz space $X$. Hence we are {left with proving} \eqref{e:intermediate_estimate_as_in_RDF}. For each $R\in \DD(J)$ let $\mathcal{G}^\ast (R)$ 
be as above. Recall that $\mathcal{M}^\ast (R)$ consists of the maximal elements of $\DD (R)\setminus \mathcal{G}^\ast (R)$ and it is such that
\[
\sum_{L\in\mathcal{M}^\ast (R)} \vert L \vert \leq \frac{\vert R\vert}{4}.
\]
For each $R\in \DD(J)$ we define
\[
\mathcal{S} (R):= \text{ maximal elements of } \left\lbrace Z\in \mathcal{G}^\ast(R): \, \sum_{\substack{W\in \mathcal{G}^\ast (R)
\\ Z\subset W}} a_W > 24 \| \mathsf{a}\|_{\mathcal D} \langle f \rangle^2_{X,R,+} \right\rbrace. 
\]
Using the definition of $\mathcal{S}(R)$ and the Carleson condition we have
\[
\sum_{Z\in\mathcal{S}(R)}\vert Z\vert \leq \frac{1}{ 12 \|\mathsf{a}\|_\DD \langle f \rangle^2_{X,R,+}} \sum_{Z\in\mathcal{S}(R)} \int_Z A_{\mathcal{G}^\ast (R)} \leq \frac{|R|}{12}.
\]
Armed with this fact, we perform the following step of recursive nature, which will lead us to the desired sparse domination in \eqref{e:intermediate_estimate_as_in_RDF}. We consider the decomposition
\[ 
A_{\mathcal{D}(R)} {[\mathsf{a}]} = {A}_{\mathcal{G}^\ast (R) \setminus \bigcup_{Z
\in \mathcal{S}(R)} \mathcal{D}(Z)} {[\mathsf{a}]} + \sum_{Z \in \mathcal{S}(R)}
A_{\mathcal{D}(Z)} {[\mathsf{a}]} + \sum_{L \in \mathcal{M}^\ast(R)}
A_{\mathcal{D}(L)} {[\mathsf{a}]}. 
\]
It follows from the definition of $\mathcal{S}(R)$ that
\[ 
{A}_{\mathcal{G}^\ast (R) \setminus \bigcup_{Z\in \mathcal{S}(R)} \mathcal{D}(Z)} {[\mathsf{a}]} \leq  24  \|\mathsf{a}\|_\DD \langle f \rangle_{X,R,+}^2 \ind_{R}. 
\]
Moreover, we have that
\[ 
\sum_{Z \in \mathcal{S}(R)} |Z| + \sum_{L \in \mathcal{M}^\ast (R) }
|L| \leq \frac{|R|}{3}. 
\]
We can therefore construct a sparse collection $\mathcal{R}(J)$ by the very same John-Nirenberg-type recursive argument as in \cite{DPF+RDF}, applying the above decomposition to each $\mathcal{D}(Z)$ and each $\mathcal{D}(L)$. This concludes the proof of \eqref{e:intermediate_estimate_as_in_RDF}.
\end{proof}
%%%%%%%%%%%%%%%%%%%%%%%%%%%%%% PROOF PROOF PROOF

%%%%%%%%%%%%%%%%%%%%%%%%%%%%%% SECTION SECTION SECTION
\subsection{Deducing Theorem~\ref{thm:smoothsf}} \label{ss:ptb} As usual, we employ the shorthand $C_X\coloneqq \ZZ^\star(X,\Xi)$. As explained in \S\ref{ss:disc}, it suffices to prove the same estimate for the model $G_\Xi$. For a suitable choice of wave packets $\{\upphi_P:P\in \mathbb P\}$, 
\[
{[G_\Xi f]^2}\leq 2 \sum_{P\in \mathbb{P}^{O_\Xi}}    |\langle f,\upphi_P \rangle|^2 1_{I_P}  =  2\sum_{I\in\mathcal{D}} a_I\ind_{I}, \qquad a_I \coloneqq \sum_{ P\in\mathbb{P}^{O_\Xi}_= (I) } |\langle f,\upphi_P \rangle|^2.
\]
By virtue of Proposition~\ref{p:balayage}, it is then enough to show that $\{a_I:\, I\in \mathcal D\}$ is a generalized Carleson sequence subordinated to $f$. We fix $I\in\mathcal{D}$ and turn to the task of constructing the major collection $\mathcal{G}(I)$. Let 
\begin{equation} \label{e:defXorl}
	E_I\coloneqq   \left\{ x\in \R:\, \M_X (f\chi_I^{\rm{dec}})(x)>\Theta \langle f \rangle_{X,I,+}\right\} {.}
\end{equation}
We claim the estimate
\begin{equation}
	\label{e:wot}
	|E_I| \leq 2^{-9}|I|,
\end{equation} and postpone the proof at the end of the section. Therefore, if  $\mathcal{L}(I)$ stands for  the collection of maximal elements of $\mathcal D$ contained in the set $E_I\cap 3I$,
\begin{align}
\label{e:pack3} &  \sum_{L \in \mathcal L(I)} |L| \leq |E_I| \leq 2^{-9} |I|,
\\ \label{e:pack4} & \sup_{L\in \mathcal L(I)} \inf_L \mathrm{M}_X\left(f\chi_I^{\mathsf{dec}}\right) \lesssim \langle f \rangle_{X,I,+} {.}
\end{align}
In particular \eqref{e:pack3} shows that $\mathcal G(I)${, constructed from $\mathcal{L}(I)$ as $\mathcal G$ from $\mathcal L$ in \eqref{e:stopcol} of \S\ref{ss:stopL2},} is major. Constructing    $T_{\mathbb P_{\mathcal G}^{O_\Xi}(I) }$ as in   \eqref{e:dec1.1} {we have}
\begin{align}\label{e:pack5} 
  \sum_{G\in \mathcal G(I)}|G|a_G&= \sum_{G\in \mathcal G(I)}|G|\sum_{P \in \mathbb P^{O_\Xi}_=(G)} |\langle f,\upphi_P \rangle|^2  
\lesssim \left\| T_{\mathbb P_{\mathcal G}^{O_\Xi}(I) } f\right\|_{2}^2 
\\ 
&\lesssim C_X^2|I|\bigg( \left\| f\chi_I^{\mathsf{dec}}\right\|_{L^\infty(3I\setminus B)} + \sup_{L\in\mathcal{L}}\inf_L  \left[\mathrm{M}_X\left(f\chi_I^{\mathsf{dec}}\right) \right] \bigg)^2 \lesssim C_X^2 |I|\langle f \rangle_{X,I,+}^2,
\end{align}
having applied Proposition~\ref{p:maintech}  to pass to the second line. This {verifies} the assumption of Proposition~\ref{p:balayage} and therefore completes the proof of Theorem~\ref{thm:smoothsf}.

%%%%%%%%%%%%%%%%%%%%%%%%%%%%%% PROOF PROOF PROOF
\begin{proof}[Proof of \eqref{e:wot}] Let $\mathcal D'$ be any dyadic grid on $\R $. {Because of our assumptions on the Orlicz space $X$ and its Young function ${\mathrm{Y}_X}$ in Definition~\ref{d:locorl}}, the maximal function 
\[
\mathrm{M}_{X,\mathcal D'} g \coloneqq \sup_{I\in \mathcal D'}\cic{1}_I\langle g \rangle_{X,I}.
\]
enjoys the modular inequality 
\[
\left|
\left\{x\in \mathbb R:\, \mathrm{M}_{X,\mathcal D'} g(x) >\lambda \right\}\right| \leq \int_{\R} \mathrm{Y}_X\left( \frac{g(x)}{\lambda} \right)\, \d x,
\] 
{see \cite{Perez95}*{Lemma 4.1} and \cite{CMP}*{Theorem 5.5}}. Using the three lattice theorem as in \cite{LN}, it follows that
\begin{equation}
	\label{e:wot1}
\left| E_{g,\lambda}\coloneqq 
\left\{x\in \mathbb R:\, \mathrm{M}_{X} g(x) >\lambda \right\}\right| \leq 3 \int_{\R} \mathrm{Y}_X\left( \frac{3g(x)}{\lambda} \right)\, \d x.
\end{equation}
Now, notice that $E_I$ defined in \eqref{e:defXorl} coincides with the set $E_{f\chi_I^{\mathsf{dec}}, \Theta \langle f \rangle_{X,I,+}}$. Therefore, provided that $\Theta>3$, there holds 
\[
\begin{split}
|E_I| & \leq 3 	\int_{\R} \mathrm{Y}_X\left( \frac{ 3}{\Theta} \frac{f(x)\chi_I^{\mathsf{dec}}(x)}{\langle f \rangle_{X,I,+}} \right)\, \d x % {\lesssim \mathrm{Y}_X\left(\frac{ 3}{\Theta}\right)} 
\leq {{\frac{9}{\Theta}}} \int_{\R} \mathrm{Y}_X\left(   \frac{f(x)\chi_I^{\mathsf{dec}}(x)}{\langle f \rangle_{X,I,+}} \right)\, \d x
 \\ 
 & \leq % {{\mathrm{Y}_X}\left(\frac{3}{\Theta} \right)}
 {{\frac{9}{\Theta}}} |I| \int_{\R} \mathrm{Y}_X\left( \frac{f(\ell_I y+c_I)\chi^{\mathsf{dec}}(y)}{\langle f \rangle_{X,I,+}} \right)\, \d y  \leq %{{\mathrm{Y}_X}\left(\frac{3}{\Theta} \right)} 
 |I|
\end{split}
\]
where the second inequality {uses the convexity of ${\mathrm{Y}_X}$}, the step to the second line  is a change of variable and the subsequent bound is the Orlicz norm definition of $\langle f \rangle_{X,I,+}$. The bound \eqref{e:wot} is then obtained by choosing  $\Theta >3$ sufficiently large.
\end{proof}
%%%%%%%%%%%%%%%%%%%%%%%%%%%%%% PROOF PROOF PROOF

%%%%%%%%%%%%%%%%%%%%%%%%%%%%%% SECTION SECTION SECTION
\section{Proofs of Theorems~\ref{thm:HMsparse}, \ref{thm:Marc}, \ref{thm:F}} \label{s:nextprob} In this section we complete the proofs of Theorems~\ref{thm:HMsparse}, \ref{thm:Marc} and \ref{thm:F}.  {Using} the discretization arguments in \S\ref{ss:disc} and \S\ref{sec:Rpq}, the proofs of Theorems~\ref{thm:HMsparse}~and~\ref{thm:Marc} reduce to a general sparse estimate for model operators $f\mapsto T_{\mathbb Q}f$ with $\mathbb Q\subset \mathbb P^{O_\Xi}$ a finite collection of tiles.

%%%%%%%%%%%%%%%%%%%%%%%%%%%%%% PROPOSITION PROPOSITION PROPOSITION
\begin{proposition}\label{prop:sparsemodelform}Let $\Xi\subset\R$ be a null closed set and $\mathbb P^{O_\Xi}$ be defined as in \eqref{e:tilesO}. For $f_1,f_2\in L^\infty _0 (\R)$, and {$I_0 \in \mathcal D$ with  $\supp f_j\subset 3I_0$ for $j\in\{1,2\}$},  there holds
\[
\sup_{\substack{\mathbb Q \subset \mathbb P^{O_\Xi} \\ \# \mathbb Q <\infty } } 
\left|\langle \ind_{I_0}T_{\mathbb Q(3I_0)} f_1,f_2\rangle\right| \lesssim \left(\prod_{j=1} ^2 \ZZ^\star(X_j,\Xi)\right) \sup_{\mathcal S \,\mathrm{ sparse}} {\Lambda^{\mathcal S}_{X_1,X_2}}(f_1,f_2)
\]
with absolute implicit constant.
\end{proposition}
%%%%%%%%%%%%%%%%%%%%%%%%%%%%%% PROPOSITION PROPOSITION PROPOSITION

Given the proposition above and the discussion that preceded it, the proof of Theorem~\ref{thm:HMsparse} follows via \eqref{e:cvxhullms} in \S\ref{ss:disc}. Similarly, the proof of Theorem~\ref{thm:Marc} follows by the proposition above in combination with Lemma~\ref{lem:marctomodel}. We prove Proposition \ref{prop:sparsemodelform} at the end of this section.

%%%%%%%%%%%%%%%%%%%%%%%%%%%%%% PROOF PROOF PROOF
\begin{proof}[Proof of Theorem \ref{thm:F}] The case $p=1$ is a special case of Theorem \ref{thm:Marc} with $X_1=X_2=X$. The case $p=2$ is dealt with next. By Lemma \ref{lem:Jatom}, it suffices to obtain a uniform estimate for atoms  $m\in R_{2,1,J}^\Xi$, $J\in  2^{\mathbb N}$.  To this aim, if $\Xi_m$ is the set defined within 		Lemma~\ref{lem:R21tomodel}, 
an application of Proposition \ref{prop:sparsemodelform} with $X_1=X,X_2=L^2$ yields
\[
\begin{split}
\sup_{\substack{\mathbb Q \subset \mathbb P^{O_{\Xi_m}} \\ \# \mathbb Q <\infty } } 
\left|\langle \ind_{I_0}T_{\mathbb Q(3I_0)} f_1,f_2\rangle\right|& \lesssim   \ZZ^{\star}(X, \Xi_m)  \ZZ^{\star}(L^2, \Xi_m) \sup_{\mathcal S \,\mathrm{ sparse}} {\Lambda^{\mathcal S}_{X,L^2}}(f_1,f_2)
\\ 
& \lesssim \sqrt{J} \ZZ^{\star\star}(X, \Xi) \sup_{\mathcal S \,\mathrm{ sparse}} {\Lambda^{\mathcal S}_{X,L^2}}(f_1,f_2) 	  
\end{split}
\]
and the $p=2$ case of the theorem follows by taking advantage once more of Lemma~\ref{lem:R21tomodel}.

The case $1<p<2$  is obtained by means of an interpolation argument of independent interest. While the exposition is tailored to this specific setting, the argument may easily be   formulated as an interpolation result for generic sparse bounds. Fixing a nonzero function $f_1\in L^\infty_0(\R)$, let $\mathcal E$ be the union of the three canonical shifted dyadic grids on $\R$, see e.g.\ \cite{LN}. For $\theta \in [0,1]$, referring to {the notation in \eqref{e:Xintp},} and $g\in L^\infty_0(\R)$ say, define
\begin{equation}\label{e:sparsenorm_1} 
\begin{split} G_\theta[g](x,I) &\coloneqq \cic{1}_{I}(x) \langle f_1 \rangle_X \langle g \rangle_{X_\theta}, \qquad (x,I) \in \R \times \mathcal E, 
\\  
\|g\|_{Y_\theta} &\coloneqq \left\| \left\|G_\theta[g](x,I)\right\|_{\ell^\infty(I \in \mathcal E)} \right\|_{L^1(x\in \mathbb R)}.
\end{split}
\end{equation}
It is immediate to check that $Y_\theta$ is a Banach space norm for each $\theta \in [0,1]$. Furthermore, the characterization of sparse bounds from e.g.\ \cites{CDOBP,LMT} tells us that
\begin{equation}
\label{e:sparsenorm_2} 
\sup_{\mathcal S \,\mathrm{ sparse}} \mathcal S_{X,X_\theta}(f_1,g)  \sim \|g\|_{Y_\theta} , \qquad \theta \in [0,1].
\end{equation}
Theorems \ref{thm:Marc} and the $p=2$ case  just obtained  can therefore be restated as   estimates for the complex valued map {$(m,g)\mapsto  \Lambda(m,g)\coloneqq \langle \mathrm{T}_m f_1,g\rangle$}
\[
\left|\Lambda(m,g) \right| \lesssim \ZZ^{\star\star}(X,\Xi)^{2-\theta} \|m\|_{R_{\frac{2}{2-\theta},1}^\Xi}\|g\|_{Y_\theta}
\]
for $\theta=0,1$ respectively. Applying  the bilinear complex interpolation theorem  \cite{GMs}*{Theorem 1.1}, 
and recalling the well-known characterization of complex interpolates of Lorentz spaces \cite{HS} which in this particular case reads $[R_{1,1}^\Xi, R_{2,1}^\Xi]_\theta = R_{\frac{2}{2-\theta},1}^\Xi $, we obtain the estimate
\[
\left|\Lambda(m,g) \right| \lesssim \ZZ^{\star\star}(X,\Xi)^{2-\theta} \|m\|_{R_{\frac{2}{2-\theta},1}^\Xi}\|g\|_{Y_\theta}, \qquad 0<\theta<1{,}
\]
which is exactly  the claim of the theorem for $p=\frac{2}{2-\theta}$. 	\end{proof}
 
%%%%%%%%%%%%%%%%%%%%%%%%%%%%%% PROOF PROOF PROOF
\begin{proof}[Proof of Proposition~\ref{prop:sparsemodelform}] We begin by fixing a finite collection of tiles $\mathbb Q\subset \mathbb P^{O_\Xi}$ and a dyadic {interval} $I_0$ such that $\mathbb Q=\mathbb Q(3I_0)$ and $\supp f_i\subset 3I_0$ for $i\in\{1,2\}$. We let $\mathcal D_{n_0}$ denote the dyadic grid of the same generation as $I_0$, namely $|I_0|=2^{-n_0}$.	

In order to facilitate the reader we recall some notation from the previous sections. Given a collection of intervals $\mathcal S\subset\mathcal D$ and an interval $I\in\mathcal S$ we will consider \emph{stopping collections} $\mathcal L_I$, indexed by the elements of $\mathcal S$, with the following properties
\begin{align}
&\mathcal L_I\subset \mathcal D(3I),\quad B_I\eqqcolon \bigcup_{L\in\mathcal L_I}L\subset 3I,\quad\mathcal G_I\coloneqq\big\{G\in\mathcal D(I):\, G\nsubseteq 3L\text{ for every }L\in\mathcal L_I\big\},
\\
&\mathbb Q(I)\coloneqq \mathbb Q^{O_\Xi}_{\mathcal G}(I)=\left\{P\in\mathbb Q^{O_\Xi} (I):\, I_P\in\mathcal G_I\right\},\qquad\lambda_{i,I}\coloneqq \lambda_{f_i,X_i}= \left\|f_i\ind_{3I\setminus B_I}\right\|_\infty+\sup_{L\in\mathcal L_I}\inf_{L}\M_{X_i}f_i.
\end{align}
We will construct a \emph{sparse collection} $\mathcal S\subset\cup_{n\geq n_0}\mathcal D_n$  such that
\[
\mathbb Q =\bigsqcup_{I\in\mathcal S} \mathbb Q(I),\qquad \lambda_{i,I}\lesssim \langle f_i \rangle_{X_i,I},\qquad I\in\mathcal S.
\]
Once this is done we can use Proposition~\ref{p:maintechdual} to complete the proof. Indeed
\[
\begin{split}
\left|\langle T_{\mathbb Q}f_1,f_2\rangle\right|&\leq \sum_{I\in\mathcal S}\left|\langle T_{\mathbb Q(I)}f_1,f_2\rangle\right|\lesssim \sum_{I\in\mathcal S}|I|\prod_{j=1} ^2 \ZZ^{\star}(X_j,\Xi){\lambda_{j,I}}
\\
&\lesssim \prod_{j=1} ^2 \ZZ^{\star}(X_j,\Xi)\sum_{I\in\mathcal S} \langle f_1 \rangle_{X_1,I}\langle f_2 \rangle_{X_2,I}|I|
\end{split}
\]
as desired. It thus remains to construct the sparse family $\mathcal S$ with the desired properties above. 

This is done exactly as in \cite{DPFR}*{\S6}. The construction is inductive, where we begin by setting $\mathcal S_0=\mathcal D_{n_0}$. Then for $m\geq 1$ and for each $I\in\mathcal S_{m-1}$ we choose $\mathcal L_I$ to denote the collection of maximal elements of the collection
\[
J\in\mathcal D,\qquad J\subseteq I,\qquad J\in\bigcup_{j=1,2}\left\{J:\, \langle f_j\rangle_{J,X_j}>\Theta \langle f_j\rangle_{X_j,3I}\right\}.
\]
We then set 
\[
\mathcal S_m\coloneqq \bigcup_{I\in\mathcal S_{m-1}}\mathcal L_I,\qquad \mathcal S\coloneqq \bigcup_{m\geq 0} \mathcal S_m
\]
and the desired properties can be verified as in \cite{DPFR}*{\S6}, inserting the bounds for the maximal functions $\M_{X_j}$ for $j\in\{1,2\}$ in order to verify the sparseness condition.
\end{proof}
%%%%%%%%%%%%%%%%%%%%%%%%%%%%%% PROOF PROOF PROOF

%%%%%%%%%%%%%%%%%%%%%%%%%%%%%% SECTION SECTION SECTION
\section{Weighted norm inequalities} \label{s:weights} This section summarizes the weighted norm inequalities that can be derived from Theorems \ref{thm:roughsf}, \ref{thm:smoothsf}, \ref{thm:Marc} and \ref{thm:F}, assuming  either
\begin{align}\label{e:growthassumpt_1}
\sup_{0<\eps\leq 1} \eps^{\frac{\tau}{2}} \ZZ^{\star}(L^{1+\eps},\Xi) &\eqqcolon K^\star(\tau,\Xi) <\infty, 
\\ 	
 \label{e:growthassumpt_2}
\sup_{0<\eps\leq 1} \eps^{\frac{\tau}{2}} \ZZ^{\star\star}(L^{1+\eps},\Xi) &\eqqcolon K^{\star\star}(\tau,\Xi)  <\infty,
\end{align}
where $\tau>0$ is a fixed growth order. These conditions are verified when $\Xi$ is a lacunary set of order $\tau \in \mathbb N$, see \eqref{e:zygplac}. For $1<p<\infty$, the $p$-th dual  of a weight $w$ is the weight $\sigma\coloneqq w^{-\frac{1}{p-1}}$. The standard characteristics
\begin{equation} \label{e:char_weight}
\begin{aligned}
& [w]_{A_p} \coloneq \sup_{I  } \langle w \rangle_{1,I} \langle w^{-1} \rangle_{\frac{1}{p-1},I}^{-1}  , \qquad  1 < p < \infty, 
\\
& [w]_{A_1} \coloneq \sup_{I } \langle w \rangle_{1,I} \| w^{- 1} \ind_I \|_{\infty}^{-1} ,
\\ 
& [w]_{A_\infty} \coloneqq \sup_{I}\frac{ \langle \mathrm{M}(w\cic{1}_I) \rangle_{1,I}}{\langle w   \rangle_{1,I}},
\\ 
& [w]_{\mathrm{RH}_s} \coloneq \sup_{I} \frac{\langle w \rangle_{s,I} }{\langle w \rangle_{1,I}}, \qquad 1<s<\infty,
\end{aligned}
\end{equation}
 with {the supremum in the definitions above} being taken over all bounded intervals $I\subset \mathbb R$, are referred throughout this section. See e.g.\ \cite{CMP} for a thorough review of their  basic properties.

 %%%%%%%%%%%%%%%%%%%%%%%%%%%%%% SECTION SECTION SECTION
 \subsection{Main estimates} The first corollary contains  weighted bounds for the $\Xi$-Marcinkie\-wicz square function $\mathrm{H}_\Xi$. {By virtue of the considerations in \S\ref{ss:disc}, analogous bounds can be proved for $\mathrm{H}_{\Xi,m}$, with $m\in \mathrm{HM}(\Xi)$.}

%%%%%%%%%%%%%%%%%%%%%%%%%%%%%% COUNTER SET

\setcounter{counter}{2} \setcounter{corollary}{1}

%%%%%%%%%%%%%%%%%%%%%%%%%%%%%% COROLLARY COROLLARY COROLLARY
\begin{corollary} \label{cor:rsf1}
Let $1<p<\infty$, $s\in\{p,\infty\}.$ Referring to the square function $\mathrm{H}_\Xi$ defined in \eqref{e:Hxi}, the  operator norm bounds
\begin{equation}
\label{e:rsf1}
\|\mathrm{H}_\Xi \|_{L^{p}(w) \to L^{p,s}(w)} \lesssim_{p} K^\star(\tau,\Xi)  [\sigma]_{A_{p'}}^{\frac\tau 2} [w]_{A_p}^{\frac1p} 
\begin{cases} [\sigma]_{A_\infty}^{\frac{1}{p}}  +[w]_{A_\infty}^{\frac{1}{p'}}  & s=p,\\  [w]_{A_\infty}^{\frac{1}{p'}} &  s=\infty ,
\end{cases}
\end{equation}
hold with implicit constant depending  on $1<p<\infty$ only.
\end{corollary}
%%%%%%%%%%%%%%%%%%%%%%%%%%%%%% COROLLARY COROLLARY COROLLARY

%%%%%%%%%%%%%%%%%%%%%%%%%%%%%% REMARK REMARK REMARK
\begin{remark} The case $s=p$ of \eqref{e:rsf1} can be slightly improved to
	\begin{equation}
\label{e:rsf1imp}
\|\mathrm{H}_\Xi \|_{L^{p}(w) } \lesssim_{p} K^\star(\tau,\Xi)  [\sigma]_{A_\infty}^{\frac\tau 2} [w]_{A_p}^{\frac1p} 
\left( [\sigma]_{A_\infty}^{\frac{1}{p}}  +[w]_{A_\infty}^{\frac{1}{p'}}  \right), \qquad 1<p<\infty.
\end{equation}
We chose to state \eqref{e:rsf1} as  the two cases may be given a unified proof. An argument for \eqref{e:rsf1imp} is sketched at the end of \S\ref{ss:pfcorw1}. 
\end{remark}
%%%%%%%%%%%%%%%%%%%%%%%%%%%%%% REMARK REMARK REMARK

%%%%%%%%%%%%%%%%%%%%%%%%%%%%%% REMARK REMARK REMARK
\begin{remark}  As a further corollary, we obtain the less precise estimates
\begin{align} \label{e:strongapbd0}
\left\|\mathrm{H}_\Xi\right\|_{L^{p}(w)} & \lesssim  K^\star(\tau,\Xi)  [w]_{A_p}^{\alpha(p,\tau)} , \qquad \alpha(p,\tau)\coloneqq \frac{\tau}{2(p-1)} + \max\left\{\frac{1}{p-1},1 \right\},
\\ \label{e:weakapbd0}
\qquad \left\|\mathrm{H}_\Xi \right\|_{L^{p}(w)\to L^{p,\infty}(w) }& \lesssim 
K^\star(\tau,\Xi) [w]_{A_p}^{\beta(p,\tau)} , \qquad \beta(p,\tau) \coloneqq \frac{\tau}{2(p-1)} + 1.
\end{align}
When $\Xi$ is   lacunary set of order $\tau$,  estimate \eqref{e:strongapbd0} was proved by Lerner \cite{LerMRL2019}*{Theorem 1.1} for $\tau=1$ and in \cite{Bak2021}*{Eq.\ (4.3)} for $\tau \geq 2$.  In these, it is   also established that the exponent $\alpha(p,\tau)$ is sharp for $1<p\leq 2$.
Sharpness follows from the  knowledge of the blowup rate  $\left\|\mathrm{H}_\Xi\right\|_{L^{p}(\R)}\sim   (p-1)^{-(\frac\tau2+1)}$ as $p\to 1^+$,  see  Bourgain's paper \cite{BouConst} for $\tau=1$ and  \cite{Bak2021}   for $\tau \geq 2$, together with an application of \cite{LPR}*{Theorem 1.2}.   The weak-type estimate \eqref{e:weakapbd0} is new; of course, it is only relevant in the  range $1<p< 2$, being superseded by \eqref{e:strongapbd0} otherwise; we cannot comment on its sharpness outside of the well known case $\tau=0$.
\end{remark}
%%%%%%%%%%%%%%%%%%%%%%%%%%%%%% REMARK REMARK REMARK

The following corollary summarizes the weighted bounds for the smooth higher order square function
{$\mathrm{G}_{\Xi}$, with $\Xi$ as above.

%%%%%%%%%%%%%%%%%%%%%%%%%%%%%% COUNTER SET

\setcounter{counter}{3} \setcounter{corollary}{1}

%%%%%%%%%%%%%%%%%%%%%%%%%%%%%% COROLLARY COROLLARY COROLLARY
\begin{corollary}\label{cor:ssfw} Referring to the square function $\mathrm{G}_\Xi$ defined in \eqref{e:Gxi}, the  operator norm bounds
\[
\left \| \mathrm{G}_{\Xi}\right \|_{L^p(w)} \lesssim  K^\star(\tau,\Xi)
\begin{cases}
[w]_{A_p} ^{\frac{1+\frac{\tau}{2}}{p-1}},\quad &1<p\leq 3,\vspace{1.5em}
\\
[\sigma]_{A_\infty} ^{\frac{\tau}{2}}[w]_{A_p} ^{\frac 1p}\left([w]_{A_\infty} ^{\frac12-\frac1p}+[\sigma]_{A_\infty} ^{\frac1p}\right),\quad &3\leq p<\infty,
\end{cases}
\]
hold with implicit constant depending  on $1<p<\infty$ only.  
\end{corollary}
%%%%%%%%%%%%%%%%%%%%%%%%%%%%%% COROLLARY COROLLARY COROLLARY

%%%%%%%%%%%%%%%%%%%%%%%%%%%%%% REMARK REMARK REMARK
\begin{remark} {As in \eqref{e:strongapbd0} we can conclude the following less precise estimate
\[
\|\mathrm{G}_\Xi\|_{L^p(w)}\lesssim K^\star(\tau,\Xi) [w]_{A_p} ^{\alpha_1(p,\tau)},\qquad \alpha_1(p,\tau)\coloneqq \frac{\tau}{2(p-1)}+\max\left(\frac{1}{p-1},\frac{1}{2}\right).
\]} 
When $\Xi$ is a lacunary set of order $\tau\geq 0$, we have the blowup rate	$\left \| \mathrm{G}_{\Xi}\right \|_{L^p(\R)} \sim (p-1)^{-(1+\frac{\tau}{2})}$ {as $p\to 1^+$}, see for example \cite{BCPV}*{\S6.1.2}, and an application of \cite{LPR}*{Theorem 1.2} reveals that the exponent in estimate of Corollary \ref{cor:ssfw} for $1<p\leq 3$ is sharp.
\end{remark}
%%%%%%%%%%%%%%%%%%%%%%%%%%%%%% REMARK REMARK REMARK

The next estimate is a corollary of Theorem \ref{thm:Marc} on Marcinkiewicz multipliers. An analogous  estimate  for   H\"ormander-Mihlin multipliers \eqref{e:hmintro}, with $K^\star(\tau,\Xi)$ replacing $K^{\star\star}(\tau,\Xi)$, may instead be deduced from Theorem \ref{thm:HMsparse}. 

%%%%%%%%%%%%%%%%%%%%%%%%%%%%%% COUNTER SET

\setcounter{counter}{5} \setcounter{corollary}{1}

%%%%%%%%%%%%%%%%%%%%%%%%%%%%%% COROLLARY COROLLARY COROLLARY
\begin{corollary}\label{cor:Mw} Let $m$ be a Marcinkiewicz, or equivalently, $R_{1,1}^\Xi$ multiplier, with singular set $\Xi$ and $\|m\|_{R_{1,1}^\Xi}\leq 1$. Then \begin{equation}
\label{e:mweigh}
\left\| \mathrm{T}_m \right\|_{L^p(w)}\lesssim  \left[K^{\star \star} (\tau,\Xi)  \right]^{2}[\sigma]_{A_\infty} ^{\frac {\tau}{2}} [w]_{A_\infty} ^{\frac{\tau}{2}}[w]_{A_p}^{\frac1p} \left( [\sigma]_{A_\infty}^{\frac{1}{p}}  +[w]_{A_\infty}^{\frac{1}{p'}}  \right), \qquad 1<p<\infty.
\end{equation}
\end{corollary}
%%%%%%%%%%%%%%%%%%%%%%%%%%%%%% COROLLARY COROLLARY COROLLARY

%%%%%%%%%%%%%%%%%%%%%%%%%%%%%% REMARK REMARK REMARK
\begin{remark} Further %estimation 
	{estimate} of the right hand side of \eqref{e:mweigh} leads to the less precise bound
\begin{equation}\label{e:mweighless}
\left\| \mathrm{T}_m \right\|_{L^p(w)} \lesssim [w]_{A_p}^{\gamma(p,\tau)} , \qquad \gamma(p,\tau)\coloneqq \frac{\tau p'}{2} +\max\left\{1,\frac{1}{p-1}\right\}, 
\end{equation}
for $1<p<\infty$: details are provided in \S\ref{ss:corw3}. When $\Xi $ is a first order lacunary set, for which $\tau=1$, \eqref{e:mweighless} coincides with the upper bound obtained by Lerner in \cite{LerMRL2019}*{Theorem 1.2}. Interestingly, the approach of  \cite{LerMRL2019} is to forgo direct estimation of  weighted norms of $\mathrm{T}_m$,  appealing instead to the weighted Chang-Wilson-Wolff inequality  to reduce to a rough square function bound analogous to that of Corollary \ref{cor:rsf1}. Our proof avoids the Chang-Wilson-Wolff inequality which may not be available for a generic set $\Xi$. 
\end{remark}
%%%%%%%%%%%%%%%%%%%%%%%%%%%%%% REMARK REMARK REMARK

Finally, we turn to weighted estimates for the rougher classes $R_{\mu,1}^\Xi$, $1<\mu\leq 2$.

%%%%%%%%%%%%%%%%%%%%%%%%%%%%%% COUNTER SET

\setcounter{counter}{7} \setcounter{corollary}{0}

%%%%%%%%%%%%%%%%%%%%%%%%%%%%%%% COROLLARY COROLLARY COROLLARY
\begin{corollary} \label{cor:mult} Let $1\leq \mu \leq 2$. Let $m$ be a $R_{\mu,1}^\Xi$ multiplier with singular set $\Xi$ and $\|m\|_{R_{\mu,1}^\Xi}\leq 1$. Then 
\[
\left\| T_m \right\|_{L^p(w)}\lesssim %\|m\|_{R_{\mu,1}^\Xi}
\left[K^{\star \star} (\tau,\Xi)  \right]^{\frac 2\mu} \begin{cases} \left([w]_{A_p}[w]_{\mathrm{RH}_{\frac{\mu}{\mu-p(\mu-1)}}}\right)^{\beta} & 1<p<\mu',
\\
 [w]_{A_{\frac{p}{\mu}} }^\beta  & \mu<p<\infty,
\end{cases}
\]
where the exponent $\beta=\beta(p,\mu,\tau)$ depends only on the indicated parameters  and may be explicitly computed.
\end{corollary}
%%%%%%%%%%%%%%%%%%%%%%%%%%%%%%% COROLLARY COROLLARY COROLLARY

%%%%%%%%%%%%%%%%%%%%%%%%%%%%%% SECTION SECTION SECTION
\subsection{Proof of Corollary \ref{cor:rsf1}} \label{ss:pfcorw1} The proof hinges on the inequality
\begin{equation} \label{e:constha}
	\|\mathrm{H}_\Xi\|_{L^p(w) \to L^{p,s}(w)} \lesssim K^\star(\tau,\Xi)\sup_{\mathcal S \;\mathrm{sparse}}\inf_{0<\eps\leq 1}  \frac{\|\mathcal{S}_{1+\eps,1}\|_{L^p(w) \to L^{p,s}(w)}}{\eps^{\frac\tau 2}}
\end{equation}
 for $ 1<p<\infty, \, s\in\{p,\infty\}$, which  is a consequence of Theorem \ref{thm:roughsf}, with reference to the   notation \eqref{e:thesparsedefs}. The main steps are summarized in the upcoming two lemmas.

%%%%%%%%%%%%%%%%%%%%%%%%%%%%%% LEMMA LEMMA LEMMA
\begin{lemma} \label{l:911} Let $1<p<\infty$, $\sigma= w^{-\frac{1}{p-1}}$ be the $p$-th dual weight of $w$ and
\[\eps(w,p) \coloneqq \frac{1}{2^8 p  [\sigma]_{A_{p'}} }.
\]
 There holds
\[
\|\mathcal{S}_{1+\eps(w,p),1+\eps(w,p)}\|_{L^{p}(w) \to L^{p,s}(w)} \lesssim_{p}  
[w]_{A_p}^{\frac1p} \begin{cases}
 [\sigma]_{A_\infty}^{\frac{1}{p}}  +[w]_{A_\infty}^{\frac{1}{p'}}  & s=p,
\\ [w]_{A_\infty}^{\frac{1}{p'}} &  s=\infty ,
\end{cases}	
\]	
uniformly  over the sparse collection $\mathcal S$. 	
\end{lemma}
%%%%%%%%%%%%%%%%%%%%%%%%%%%%%% LEMMA LEMMA LEMMA

%%%%%%%%%%%%%%%%%%%%%%%%%%%%%% PROOF PROOF PROOF
\begin{proof} For brevity write $\eps=\eps(w,p)$ and
set $q=\frac{p}{1+\eps}$. Then 
\[
\|\mathcal{S}_{1+\eps,1+\eps}\|_{L^{p}(w) }
= \|\mathcal{S}_{1,1}\|_{L^{q}(w)}^{\frac{1}{1+\eps}}, \qquad\|\mathcal{S}_{1+\eps,1+\eps}\|_{L^{p}(w) \to L^{p,\infty}(w)}
= \|\mathcal{S}_{1,1}\|_{L^{q}(w) \to L^{q,\infty}(w)}^{\frac{1}{1+\eps}}.
\] 
The strong norm of $\mathcal{S}_{1,1}$ above is estimated via an appeal to \cite{LKC}*{Theorem 1.2} and a subsequent use of the sharp reverse H\"older inequality of \cite{HPRe}, where the latter motivates the definition of $\eps(w,p)$. See \cite{DHL}*{Section 7} for details of a similar computation. The weak norm of $\mathcal{S}_{1,1}$ is instead estimated by an appeal to \cite{LISU}*{Theorem 1.2}  and a similar usage of the reverse H\"older inequality. 
\end{proof}
%%%%%%%%%%%%%%%%%%%%%%%%%%%%%% PROOF PROOF PROOF

%%%%%%%%%%%%%%%%%%%%%%%%%%%%%% LEMMA LEMMA LEMMA
\begin{lemma} \label{l:912}
For $1<p<\infty, s\in\{p,\infty\} $ we have
\[
\|\mathcal{S}_{1+\eps,1}\|_{L^p(w) \to L^{p,s}(w)} \lesssim_p  [w]_{A_\infty}^{\frac{\eps}{1+\eps}}  \|\mathcal{S}_{1+\eps,1+\eps}\|_{L^p(w) \to L^{p,s}(w)} 	  
\]
with constant independent of $\eps \in (0,1]$.
\end{lemma}
%%%%%%%%%%%%%%%%%%%%%%%%%%%%%% LEMMA LEMMA LEMMA

%%%%%%%%%%%%%%%%%%%%%%%%%%%%%% PROOF PROOF PROOF
\begin{proof}{The norm control follows}  easily from the good-$\lambda$ inequality
\begin{equation}  \label{e:goodlambdahere}
w\left( \left\{  \mathcal{S}_{1+\eps,1} f> 2\lambda, \mathcal{S}_{1+\eps,1+\eps} f\leq \gamma \lambda \right\} \right) \leq C \exp\left( -\delta\frac{\gamma^{{-\frac{1+\eps}{\eps}}}}{[w]_{A_\infty}} \right)w\left( \left\{  \mathcal{S}_{1+\eps,1} f> \lambda\right\} \right) 
\end{equation}
with absolute constants $C,\delta>0$, whose proof is a minor variation of \cite{DPF+RDF}*{Theorem E}, to which we send for details.
\end{proof}
%%%%%%%%%%%%%%%%%%%%%%%%%%%%%% PROOF PROOF PROOF

Combining \eqref{e:constha} with Lemma \ref{l:912} entails the control
\begin{equation} \label{e:H}
\|\mathrm{H}_{\Xi}\|_{L^p(w) \to L^{p,s}(w)} \lesssim {K^\star(\tau,\Xi) \eps(w,p)^{-\frac{\tau}{2}}} [w]_{A_\infty}^{\frac{\eps(w,p)}{1+\eps(w,p)}}
\|\mathcal{S}_{1+\eps(w,p),1+\eps(w,p)}\|_{L^p(w) \to L^{p,s}(w)}  .
\end{equation}
However 
\[
[w]_{A_\infty}^{\frac{\eps(w,p)}{1+\eps(w,p)}} \leq [\sigma]_{A_{p'}}^{(p-1)\eps(w,p)}=\exp\left(\frac{\log [\sigma]_{A_{p'}}}{2^8p [\sigma]_{A_{p'}}}\right) \lesssim_p 1
\]
and Corollary \ref{cor:rsf1} follows from \eqref{e:H} together with Lemma \ref{l:911}.

%%%%%%%%%%%%%%%%%%%%%%%%%%%%%% PROOF PROOF PROOF
\begin{proof}[Proof of \eqref{e:rsf1imp}] An application of \cite{LKC}*{Theorem 1.2} together with the sharp reverse H\"older inequality yield
\begin{equation} \label{e:thateps}
\|\mathcal{S}_{1+\widetilde{\eps}(w,p),1+\widetilde{\eps}(w,p)}\|_{L^{p}(w) } \lesssim_p [w]_{A_p}^{\frac1p} \left( [\sigma]_{A_\infty}^{\frac{1}{p}}  +[w]_{A_\infty}^{\frac{1}{p'}}  \right), \qquad  \widetilde{\eps}(w,p) 
\coloneqq \frac{1}{ 2^8 p  [\sigma]_{A_{\infty}} },
\end{equation}
and \eqref{e:rsf1imp} is then a direct consequence of \eqref{e:H}.
\end{proof}
%%%%%%%%%%%%%%%%%%%%%%%%%%%%%% PROOF PROOF PROOF

%%%%%%%%%%%%%%%%%%%%%%%%%%%%%% SECTION SECTION SECTION
\subsection{Proof of Corollary \ref{cor:ssfw}} The relevant consequence of Theorem \ref{thm:smoothsf} is 
\begin{equation} \label{e:constG}
\|\mathrm{G}_\Xi\|_{L^p(w)  } \lesssim K^\star(\tau,\Xi)\sup_{\mathcal S \;\mathrm{sparse}}\inf_{0<\eps\leq 1}  \frac{\|\mathcal{S}_{1+\eps,2}\|_{L^p(w)}}{\eps^{\frac\tau 2}}, \qquad 1<p<\infty,
\end{equation}
{with reference to the   notation \eqref{e:thesparsedefs}}. Fix $p\geq  3$, let $\sigma$ stand for the $p$-th dual of $w$ and $\eps=\widetilde{\eps}(w,p)$ {as defined} in \eqref{e:thateps}. Then, uniformly over sparse collections $\mathcal S$, 
\begin{equation} \label{e:pfcorssfw1}
\begin{split}
\left\|\mathcal{S}_{1+\eps,2}\right\|_{L^p(w)} &=\left\|\mathcal{S}_{1,\frac{2}{1+\eps}}\right\|_{L^{\frac{p}{1+\eps}}(w)}^{\frac{1}{1+\eps}}  \lesssim_p [w]_{ {A_{\frac{p}{1+\eps}}} }^{\frac1p} \left( [w]_{A_\infty}^{\frac12-\frac1p} + 	\left[\sigma^{1+\frac{
{\eps p}	}{p-(1+\eps)}}\right]_{A_\infty}^{\frac1p}\right)
\\ 
& \lesssim_p  [w]_{A_p}^{\frac1p} \left( [w]_{A_\infty}^{\frac12-\frac1p} + \left[\sigma \right]_{A_\infty}^{\frac1p}\right)
\end{split}
\end{equation}
where we have used \cite{LaLi}*{Theorem 2.3} for the {approximate inequality in} the first line, and the sharp reverse H\"older {property of $\sigma$} to pass to the second line. Note that \eqref{e:pfcorssfw1} yields the case $p\geq 3$ of the corollary in combination with \eqref{e:constG}.
Specializing to $p=3$ {we get}
\[ 
\|\mathrm{G}_\Xi\|_{L^p(w)  } \lesssim {K^\star(\tau,\Xi)} [w]_{A_p}^{\frac{1+\frac\tau 2}{{2}}}
\]
and the case $1<p<3$ is then obtained by sharp extrapolation theorems; {see e.g.\ \cite{Duo}.}

%%%%%%%%%%%%%%%%%%%%%%%%%%%%%% SECTION SECTION SECTION
\subsection{Proof of Corollary \ref{cor:Mw}} \label{ss:corw3}
This time the starting point is the inequality
\begin{equation}
	\label{e:consthc}
	\|\mathrm{T}_m\|_{L^p(w)} \lesssim \left[K^{\star \star} (\tau,\Xi)  \right]^{2} \sup_{\mathcal S \;\mathrm{sparse}}\inf_{0<\eps_1,\eps_2\leq 1}  \frac{\left\|\Lambda^{\mathcal{S}}_{1+\eps_1,1+\eps_2}\right\|_{L^p(w) \times L^{p'}(\sigma)   }}{(\eps_1\eps_2)^{\frac\tau 2}}
\end{equation}
 for $ 1<p<\infty$, a consequence of Theorem \ref{thm:Marc}, {with reference to the   notation \eqref{e:thesparsedefs_2}}.  Fix $1<p<\infty$ and apply again \cite{LKC}*{Theorem 1.2}together with two instances of the sharp reverse H\"older {properties of $\sigma$ and $w$}, yielding
\begin{equation} 
{\left\|\Lambda^{\mathcal{S}}_{1+\widetilde{\eps_1} ,1+\widetilde{\eps_2}}\right\|_{L^{p}(w) \times L^{p'}(\sigma) }} \lesssim_p [w]_{A_p}^{\frac1p} \left( [\sigma]_{A_\infty}^{\frac{1}{p}}  +[w]_{A_\infty}^{\frac{1}{p'}}  \right), \qquad  \widetilde{\eps_1} \coloneqq \textstyle  \frac{1}{2^8 p  [\sigma]_{A_{\infty}} },\;  \widetilde{\eps_2} \coloneqq \textstyle  \frac{1}{2^8 p'  [w]_{A_{\infty}} }. 
\end{equation}
With these choices, {estimate}  \eqref{e:mweigh} follows by combining the estimate in the previous display with  \eqref{e:consthc}.

%%%%%%%%%%%%%%%%%%%%%%%%%%%%%% PROOF PROOF PROOF
\begin{proof}[Proof of \eqref{e:mweighless}] First, let us argue for the case $1<p\leq 2$. Starting from  \eqref{e:mweigh} in this particular range, estimate
\begin{equation}
	\label{e:tobedualz}
	\frac{\left\| \mathrm{T}_m \right\|_{L^p(w)}}{\left[K^{\star \star} (\tau,\Xi)  \right]^{2}}\lesssim [\sigma]_{A_\infty} ^{\frac {\tau}{2}} [w]_{A_\infty} ^{\frac{\tau}{2}}[w]_{A_p}^{\frac1p} \left( [\sigma]_{A_\infty}^{\frac{1}{p}}  +[w]_{A_\infty}^{\frac{1}{p'}}  \right) \lesssim [w]_{A_p}^{\frac{\frac\tau 2 + 1}{p-1}+\frac\tau 2 }    .
\end{equation}
Now if $2<p<\infty$, use duality to obtain
\begin{equation}
	\label{e:tobedualz2}
	\left\| \mathrm{T}_m \right\|_{L^p(w)}= 	\left\| \mathrm{T}_m \right\|_{L^{p'}(\sigma)}   \lesssim \left[K^{\star \star} (\tau,\Xi)  \right]^{2} [\sigma]_{A_{p'}}^{\frac{\frac\tau 2 + 1}{p'-1}+\frac\tau 2 } = \left[K^{\star \star} (\tau,\Xi)  \right]^{2} [w]_{A_{p}}^{ \frac\tau 2 + 1+\frac{\tau}{ 2(p-1)} }   .
\end{equation}
Collecting \eqref{e:tobedualz} and \eqref{e:tobedualz2} with some algebra results exactly into \eqref{e:mweighless}.
\end{proof}
%%%%%%%%%%%%%%%%%%%%%%%%%%%%%% PROOF PROOF PROOF

%%%%%%%%%%%%%%%%%%%%%%%%%%%%%% SECTION SECTION SECTION
\subsection{Proof of Corollary \ref{cor:mult}} Let $m\in R_{\mu,1}^\Xi$, $1< \mu < 2$ of unit norm. It is best to argue separately for $1<p<\mu'$ and $p>\mu$. In this proof, we apply the conventions that $\eta=\eta(\eps)$ is a generic sublinear function of $\eps\in [0,1]$ and $\beta$ is  positive real exponent which are allowed to vary between occurrences. The exact forms of $\eta$ and $\beta$ are allowed to depend on $p,\mu$ and may be explicitly computed.

We begin with the first range. Applying Theorem \ref{thm:F} with $X=L^{1+\eps}(0,1)$ and computing the interpolation space 
$X_\mu {\coloneqq} [X,L^2(0,1)]_\theta$ in \eqref{e:Xintp} for the value $\theta= \frac{2(\mu-1)}{\mu}$   leads to the inequality
\begin{equation}\label{e:consthf} 	
\frac{\|\mathrm{T}_m\|_{L^p(w)} }{\left[K^{\star\star}(\tau,\Xi)\right]^{\frac2\mu}}\lesssim \sup_{\mathcal S \;\mathrm{sparse}}\inf_{0<\eps \leq 1}  \frac{ {\left\|\Lambda^{\mathcal{S}}_{1+\eps,\mu+\eps \delta(\eps,\mu)}\right\|_{L^p(w) \times L^{p'}(\sigma)  }}
	}{\eps^{\frac\tau \mu}}, \textstyle \qquad \delta(\eps,\mu)=
{\frac{\mu (2-\mu)}{1+\eps( \mu  - 1)}},
\end{equation}
{with reference to the   notation \eqref{e:thesparsedefs_2}}. Applying \cite{LKC}*{Theorem 1.2} leads to  
\[
\left\|\Lambda^{\mathcal{S}}_{1+\eps,\mu+\eps \delta(\eps,\mu)}\right\|_{L^p(w) \times L^{p'}(\sigma)}  \lesssim \left[ w^{   {\frac{\mu}{\mu-p(\mu-1)}+\eta(\eps)}} \right]_{A_{ {\frac{\mu'}{\mu'-p}(p-1) + 1 -\eta(\eps) }}}^\beta,   \qquad 0<\eps \leq 1.
\]
Choosing now $\eps=c[w^{ \frac{\mu}{\mu-p(\mu-1)}}]_{A_\infty}^{-1}$ and using \eqref{e:consthf} leads via the reverse H\"older inequality to 

\[
\frac{\|\mathrm{T}_m\|_{L^p(w)}}{\left[K^{\star\star}(\tau,\Xi)\right]^{\frac2\mu}}
\lesssim  \left[w^{  \upsilon}\right]_{A_\infty}^{\frac{\tau}{\mu}}\left[ w^{\upsilon}\right]_{A_{ \nu}}^\beta \lesssim \left[ w^{\upsilon}\right]_{A_{ \nu}}^\beta \lesssim \left( [w]_{A_p}[w]_{\mathrm{RH}_{{\upsilon}}} \right)^\beta
\]
where
\[
{\upsilon(\mu,p) = \frac{\mu}{\mu-p(\mu-1)},} \qquad {\nu(\mu,p) = \frac{\mu'}{\mu'-p}(p-1) + 1,}
\]
and this completes the proof of the case $1<p<\mu'$. 

We turn to the range $p>\mu$. This time  Theorem \ref{thm:F} is applied in adjoint form,  leading to the inequality
\begin{equation}\label{e:consthf2} 
\|\mathrm{T}_m\|_{L^p(w)} \lesssim \left[K^{\star \star} (\tau,\Xi)  \right]^{\frac 2\mu} \sup_{\mathcal S \;\mathrm{sparse}}\inf_{0<\eps \leq 1}  \frac{{\left\|\Lambda^{\mathcal{S}}_{\mu+\eps \delta(\eps,\mu),1+\eps}\right\|_{L^p(w) \times L^{p'}(\sigma)  }}}{\eps^{\frac\tau \mu}}, \qquad \delta (\eps,\mu)=\textstyle  {\frac{\mu (2-\mu)}{1+\eps( \mu  - 1)}}.
\end{equation}
The claimed estimate is again a simple consequence of \cite{LKC}*{Theorem 1.2} and of the reverse H\"older inequality, and we omit the details.

%%%%%%%%%%%%%%%%%%%%%%%%%%%%%% SECTION SECTION SECTION 
\section{Upper bounds for $\ZZ^\star(X,\Xi)$}\label{proof_equiv_Zyg_prop}  This section contains the proofs of the leftmost upper bounds in Theorem \ref{t:mainchar},  Corollary \ref{c:sample} eq.\ \eqref{e:maincharsf}, and Corollary \ref{c:sample2} eq.\ \eqref{e:maincharssf}. These share the same broad goal, which, up to easy algebra, corresponds to showing
\begin{equation}\label{e:reusal}
	\ZZ^\star(X,\Xi) \leq C \frac{1}{\mathrm{Y}_X^{-1}\left(\frac{1}{CK}\right)}, 
\end{equation}
where $K$ is the corresponding maximal modular estimate in the assumption and $C\geq 1$ is an absolute constant which will be explicitly computed in each case.
All arguments rely on the characterization
\begin{equation}
\ZZ^\star(X,\Xi) = \sup\left\{ \sqrt{\sum_{k\in\lfloor \mu \Xi \rfloor}  |\widehat{f} (k)|^2 } :\, \mu \in 2^{\mathbb Z}, \,f=\sum_{|k|\leq N}\widehat {f}(k)\exp(2\pi i k \cdot):\, \|f\|_X=1,  \, N\geq 1 \right\},
\end{equation}
which a simple consequence of duality, as well as   upon the upcoming construction~\eqref{e:thecount}. Pick a nonnegative even Schwartz function $\phi$ with $\mathrm{supp} \,\phi \subset [-1,1]$, $\widehat{\phi} (0)= \int_\mathbb{R} \phi (x) \, \d x= 1$. Let $\eps >0$ be a fixed parameter to be chosen during each  argument, and for each trigonometric polynomial $f$  define 
\begin{equation} \label{e:thecount}
Pf (x) \coloneq \sum_{|k| \leq N} \widehat{f} (k) \int_{\mathbb{R}} \phi \left( \frac {\xi-k}{\eps } \right) \e^{ 2\pi i  \xi x}  \, \frac{\d \xi}{\eps}, \qquad Qf(x) \coloneq \widehat{\phi} (\eps x) P{f} (x), \qquad x\in \mathbb R.
\end{equation}
Notice that $Pf,Qf$ are Schwartz functions. For further use, we claim the existence of a constant $C$ depending on $\eps$ only, such that 
\begin{equation}\label{e:thecure}
 	\int_{\R}\mathrm{Y}_X \left( \frac{|Qf|}{\lambda}  \right) \leq C \int_{[0,1]}\mathrm{Y}_X \left( \frac{|f|}{\lambda}  \right) 
\end{equation}
uniformly over $\lambda>0$. To verify this, notice the equality
\[
P{f} (x) =  \int_{\mathbb R} f(x) \e^{ 2\pi i t x}  \phi \left(  \frac{t}{ \eps }  \right)\, \frac{\d t} {\eps} , \qquad x\in \mathbb R 
\]
which in particular entails $|Pf|\leq |f|$, and exploit the rapid decay of $ \widehat{\phi} (\eps \cdot)$ at scale $\eps^{-1}$.
Furthermore, the equality
\begin{equation}
	\label{e:Qf}
	\widehat{Q f}  = \sum_{|k|\leq N} \widehat f(k) \widetilde{\phi}_{k,\eps}, \qquad \widetilde{\phi}_{k,\eps} (\xi)\coloneqq \int_{\R}  \phi\left(\frac{\xi-\zeta}{\eps} \right)  \phi\left(\frac{\zeta-k}{\eps} \right)  \,\frac{\d\zeta}{\eps^2}
\end{equation}
holds, and we highlight the properties that $\widetilde{\phi}_{k,\eps}$ is {nonnegative, supported in $(k-4\eps,k+4\eps)$, and} has integral 1. This has in particular the consequence that $\widehat{\widetilde{\phi}_{k,\eps}}(x)\geq \frac12$ for $\eps|x| \leq c$ for some absolute constant $c$. 

%%%%%%%%%%%%%%%%%%%%%%%%%%%%%% SECTION SECTION SECTION
\subsection{Proof of the leftmost bounds in \eqref{e:maincharsf}, \eqref{e:maincharssf}} By virtue of the inequality
\begin{equation}
	\label{e:4graph}
 \sup_{\Xi'\subset \Xi} \sup_{\|m\|_{\mathrm{HM}(\Xi')} \leq 1}\left[\mathrm{H}_{\Xi',m}\right]_{X} \gtrsim
  \sup_{\Xi'\subset \Xi}  \left[\mathrm{G}_{\Xi'}\right]_{X},
\end{equation}
 it suffices to present the argument for  \eqref{e:maincharssf}. For convenience, let $K$ stand for the right hand side of \eqref{e:4graph}.  Fix a trigonometric polynomial $f$ with $\|f\|_{X}=1$ and $\widehat{f}(k)=0$ for $|k|>N$. Fix also $\mu\in 2^{\mathbb Z}$ and define
  \[
  A_j\coloneqq \left\{k \in  \lfloor \mu \Xi \rfloor:\,  k\equiv {j  \Mod 4},\,  |k|\leq N\right\}, \qquad j=0,\ldots, 3.
  \]
 Fixing $j\in \{0,\ldots,3\}$ for the moment, let $a_1<\ldots<a_M$ be an ordering of $A_j$. The set $\mu \Xi \cap [a_m,a_m+1)$ is clearly nonempty for each $m=1,\ldots, M$, therefore  $\xi_m\coloneqq \max  \left[\mu \Xi \cap [a_m,a_m+1)\right]$ is well defined. Setting also $\zeta_m\coloneqq \xi_m  +\frac12$, $\zeta_0\coloneqq -\infty$  it holds that
 \[
 a_m\in (\zeta_{m-1},\zeta_{m}), \quad \textstyle \frac12 \dist\left(a_m, \zeta_{m}\right) <\frac32, \quad  \dist\left(\zeta_{m-1}, \zeta_{m}\right) \geq 3 \qquad  m=1,\ldots, M. 
 \]
 If $Z=\{\zeta_1,\ldots, \zeta_m\}$, the above situation entails for each $m=1,\ldots, M$ the existence of an interval $\omega_m\in \mathbf{w}_{\mathcal D}(O_Z)$ with the property that the interval $(a_{m}-2^{-9}, a_{m}+2^{-9}) \subset \mathcal \omega_m$, so that we may pick $\varphi_m\in c\Phi_{\omega_m}$ with the property that $\widehat{\varphi_m} =1$ on $(a_{m}-2^{-10}, a_{m}+2^{-10}) $.  
 The latter choice, together with \eqref{e:Qf}, entails the chain 
 \[
 \left|
 \mathrm{T}_{\widehat {\varphi_m}} [Qf](x)\right| = |\widehat f(a_m)|\left| \int  {\widetilde\phi}_{a_m,\eps}(\xi) \e^{2\pi i x \xi} \, \d \xi \right|
  = |\widehat f(a_m)|\left| \widehat{   {\widetilde\phi}_{a_m,\eps}}( {-} x)  \right| \geq \frac12 |\widehat f(a_m)| , \quad x\in [0,1]
 \] 
 provided that $\eps$ is chosen small enough. We have achieved the key inequality
 \begin{equation}\label{e:keyeq} 
 \mathrm{G}_{Z} [Qf](x) \geq c \left(\sum_{m=1}^M  \left| \mathrm{T}_{\widehat {\varphi_m}} [Qf](x)\right| ^2\right)^{\frac12} \geq c \left(\sum_{k\in A_j}  |\widehat f(k)|^2\right)^{\frac12}, \qquad   x\in [0,1]
 \end{equation}
while on the other hand by modulation invariance
\begin{equation}\label{e:keyineq}     
\left[\mathrm{G}_{Z}\right]_{X} \leq    \sup_{\Xi'\subset \Xi}  \left[\mathrm{G}_{\Xi'}\right]_{X}= K.
\end{equation}
Call $\lambda$ the right hand side of \eqref{e:keyeq}. By virtue of \eqref{e:keyeq} itself,
\[
\begin{split}
1 &\leq   \left|\{ \mathrm{G}_{Z} [Qf]>\lambda \}\right| \leq   K  \int_{\R}\mathrm{Y}_X \left( \frac{|Qf|}{\lambda}  \right) \leq  C  K	\int_{[0,1]}\mathrm{Y}_X \left( \frac{|f|}{\lambda}  \right)
\\  
& \leq  C  K	\mathrm{Y}_X \left( \frac{1}{\lambda}  \right) \int_{[0,1]}\mathrm{Y}_X \left( |f| \right) \leq  C   K	\mathrm{Y}_X \left( \frac{1}{\lambda}  \right) 
\end{split}
\]
where the passage to the second line uses submultiplicativity and the last inequality exploits $\|f\|_X\leq 1$. Rearranging,
\[
\lambda \leq \mathcal Q(K), \qquad \mathcal Q(t) \coloneqq \frac{1}{\mathrm{Y}_X^{-1}\left(\frac{1}{C t}\right)}
\]
which, up to repeating the proof for all values of $j=0,\ldots, 3$, yields, with reference to \eqref{e:4graph},  
\begin{equation}
\ZZ^\star(X,\Xi) \leq 4 \mathcal Q\left(K\right)
\end{equation}
which is a form of \eqref{e:reusal}, and thus completes the proof of the leftmost estimate in \eqref{e:maincharssf}.
\subsection{Proof of the leftmost estimate of Theorem 
\ref{t:mainchar}} The proof uses the same tools and notation seen in the previous argument. In this proof,
\[
K\coloneqq \sup_{\|m\|_{\mathrm{HM}(\Xi)} \leq 1} [\mathrm{T}_m ]_{X}.
\]
 Fix again  a trigonometric polynomial $f$ with $\|f\|_{X}=1$, $\widehat{f}(k)=0$ for $|k|>N$, and  $\mu\in 2^{\mathbb Z}$. 
For our purposes, it suffices to deal with those $f$ with $\widehat f(k)\neq 0$ for at least one $k\in  \lfloor \mu \Xi \rfloor  $.
Let  $\eps>0$ to be chosen momentarily and introduce the corresponding  multiplier  
\begin{equation}
\label{e:multip4}	
m (\xi) \coloneq \sum_{k\in\lfloor \mu \Xi \rfloor  } \; \eta_k  \psi \left( \frac{ \xi -k} {\eps}\right)  
,\qquad  \eta_k \coloneq   \frac{ \overline{\widehat{f} (k)}}{\left( \sum_{j\in\lfloor \mu \Xi \rfloor } |\widehat{f}(j)|^2 \right)^{\frac 12}} , \quad  k\in\lfloor \mu \Xi \rfloor, \end{equation}
and
  $\psi \in {\mathcal S} (\mathbb{R})$ with the properties $\mathbf{1}_{[-4,4]} \leq \psi \leq \mathbf{1}_{[-6,6]}$. 
The condition on the support of $\psi$ ensures that $\| m  \|_{{\rm HM}(\Xi)} %\leq 1
{\leq M}$ {for a fixed $M < \infty$} and
$$ \mathrm{T}_m [Qf](x) = \sum_{k \in \lfloor \mu \Xi \rfloor} \eta_k \widehat{f} (k) \int_{\mathbb{R}} \widetilde{\phi}_{\epsilon, k} (\xi) e^{ 2  \pi i \xi x} \, \d \xi .$$
An easy computation shows that for $\eps |x|\leq c$ we have
$$ | T_m [Qf](x)| \geq   \left( \sum_{k \in\lfloor \mu \Xi \rfloor } |\widehat{f}(k)|^2 \right)^{\frac 12} { \inf_{k \in\lfloor \mu \Xi \rfloor }} \left|\widehat{ \widetilde{\phi}_{\epsilon, k}} (x) \right|\geq \frac 12  \left( \sum_{k \in\lfloor \mu \Xi \rfloor } |\widehat{f}(k)|^2 \right)^{\frac 12}. $$
Therefore, choosing $\eps\leq c$ and  setting $\lambda $ equal to the right hand side of the last display, we may argue in the previous section and achieve the inequality
$$ 1 \leq   \left|\{ \left| \mathrm{T}_m  [Qf] \right| >\lambda \}\right| \leq C K {M} {\mathrm{Y}_X} \left(  \frac{1}{\lambda} \right). $$
The leftmost estimate in Theorem~\ref{t:mainchar} then follows along the same exact lines.

%%%%%%%%%%%%%%%%%%%%%%%%%%%%%% SECTION SECTION SECTION
\section{ A characterization of the Littlewood-Paley property} \label{s:LPp} For the proof of the characterization of the $\mathrm{LP}(p)$-property in terms of the maximal multiscale Zygmund property, {as stated in Theorem~\ref{thm:LP}}, it will be useful to work with perturbations of a given sets of frequencies. Such definitions are standard in the literature, see \cite{HaKl}*{Definition 2.6}. We give below the appropriate version for the real line.

%%%%%%%%%%%%%%%%%%%%%%%%%%%%%% DEFINITION DEFINITION DEFINITION
\begin{definition}[Perturbations of $\mathrm{LP}$-sets] Given a closed null set $\Xi\subset \R$ as above we will say that $\Xi'$ is a \emph{perturbation of $\Xi$} if for every $\omega\in\Omega_{\Xi'}$ there exist at most two intervals $\omega_-,\omega_+\in\Omega_\Xi$ such that $\omega\subset \omega_-\cup\omega_+$ up to a set of measure zero.
\end{definition}
%%%%%%%%%%%%%%%%%%%%%%%%%%%%%% DEFINITION DEFINITION DEFINITION

%%%%%%%%%%%%%%%%%%%%%%%%%%%%%% REMARK REMARK REMARK
\begin{remark} \label{rmrk:pert} It is easy to check that a perturbation of an $\mathrm{LP}(p)$-set is also a $\mathrm{LP}(p)$-set. Indeed, for each $\omega\in \Omega_{\Xi'}$ there holds
\[
\mathrm{H}_{\omega}=\mathrm{H}_{\omega\cap \omega_-}+\mathrm{H}_{\omega\cap \omega_+}=\mathrm{H}_{\omega}\circ \mathrm{H}_{\omega_-}+\mathrm{H}_{\omega}\circ {\mathrm{H}}_{\omega_+}
\]
and our claim follows by $L^p(\ell^2)$ bound for the Hilbert transform.
\end{remark}
%%%%%%%%%%%%%%%%%%%%%%%%%%%%%% REMARK REMARK REMARK

%%%%%%%%%%%%%%%%%%%%%%%%%%%%%% PROOF PROOF PROOF
\begin{proof}[Proof of Theorem~\ref{thm:LP}] Clearly $\mathrm{1.}\Rightarrow \mathrm{2.}$ as $\ZZ^{\star \star}$ is a stronger version of $\ZZ^\star$. The fact that $\mathrm{2.}\Rightarrow \mathrm{3.}$ is a consequence of Theorem~\ref{thm:roughsf} applied with $X=L^p$ for any $p>q$. The sparse control of our theorem implies in particular that ${\mathrm H}_\Xi$ is bounded on $L^p$ for any $p>q$. It remains to show that $\mathrm{3.}\Rightarrow\mathrm{1.}$ For this note that any set $\{p_\omega:\, \omega\in\Omega\}$ with $p_\omega\in\overline{\omega}$ with one point per interval is a perturbation of $\Xi$. Hence, Remark~\ref{rmrk:pert} shows that $\widetilde \Omega\coloneqq\{p_\omega:\, \omega\in\Omega\}$ has the $\mathrm{LP}(p)$ property for any $p>q$, and by scale invariance of the square function estimate \eqref{eq:LP} so does any rescaling of the form $\lambda\Xi$ with $\lambda>0$. Applying Corollary~\ref{c:sample} with $X=L^p$ for any $p>q$ shows that $\ZZ^{\star}(L^p,\left\{p_\omega:\, \omega\in\Omega\right\})<+\infty$, uniformly over choices of points $p_\omega\in\overline{\omega}$ and this completes the proof of the theorem.
\end{proof}
%%%%%%%%%%%%%%%%%%%%%%%%%%%%%% PROOF PROOF PROOF

%%%%%%%%%%%%%%%%%%%%%%%%%%%%%% COUNTER SET
\appendix\setcounter{section}{25} 

%%%%%%%%%%%%%%%%%%%%%%%%%%%%%% SECTION SECTION SECTION
\section{Sparse domination implies modular inequalities} \label{s:app} This appendix is devoted to the statement and proof of Proposition \ref{p:weighmod} below. We have appealed to this proposition to deduce Corollaries \ref{c:sample} and \ref{cor:sample3} from Theorems \ref{thm:roughsf} and \ref{thm:HMsparse} respectively. Throughout the appendix, $X$ is a local Orlicz space with Young function $\mathrm{Y}_X$ enjoying the $B_p$ property as spelled out in Definition~\ref{d:locorl}. {We refer to \eqref{e:char_weight} for the definitions of the characteristics of weights.} 

%%%%%%%%%%%%%%%%%%%%%%%%%%%%%% PROPOSITION PROPOSITION PROPOSITION
\begin{proposition} \label{p:weighmod} Let  $\textstyle 1\leq  q<\infty,  1<p<q'$ and let $X$ be a local Orlicz space with $B_p(X)\lesssim 1$. Referring to Definition \ref{d:locspa}, assume $\|T\|_{X,L^q}\lesssim 1$. Then, there exists a positive increasing function $\mathcal Q$ such that
\[
w\left(\left\{x\in \mathbb R: |Tf(x)|>\lambda\right\} \right) \leq \mathcal Q\left([w]_{A_1}, [w]_{\mathrm{RH}_{\frac{q}{q-p(q-1)}}}\right) \int_{\mathbb R} \mathrm{Y}_{X} \left(\frac{|f(x)|}{\lambda}\right)  w(x) \, \d x
\]
uniformly over all weights $w$.
\end{proposition}
%%%%%%%%%%%%%%%%%%%%%%%%%%%%%% PROPOSITION PROPOSITION PROPOSITION

%%%%%%%%%%%%%%%%%%%%%%%%%%%%%% REMARK REMARK REMARK
\begin{remark} If $\|T\|_{X,L^q}\lesssim_q 1$  for all $1<q<\infty$, and $B_p(X)\lesssim_p 1$ for some $1<p<\infty$, the reverse H\"older property of $A_1$ weights{, see \cite{SteinBook}*{Chpt. 5, Prop. 3}}, may be also used to prove that 
  \[
w\left(\left\{x\in \mathbb R: |Tf(x)|>\lambda\right\} \right) \leq  \widetilde{\mathcal Q} \left([w]_{A_1} \right) \int_{\mathbb R} \mathrm{Y}_X \left(\frac{|f(x)|}{\lambda}\right)  w(x) \, \d x ,
\]
{where $\widetilde{\mathcal Q}(x) \coloneq \mathcal Q (x,c(x))$ for a specific increasing function $c$, hence $\widetilde{\mathcal Q}$ is {positive} increasing as well.}
\end{remark}
%%%%%%%%%%%%%%%%%%%%%%%%%%%%%% REMARK REMARK REMARK

%%%%%%%%%%%%%%%%%%%%%%%%%%%%%% SECTION SECTION SECTION
\subsection{Proof of Proposition \ref{p:weighmod}} We will need two lemmas. The first one is proved exactly in the same way as the Fefferman-Stein theorem. The second one follows from the first via a simple layer cake argument which we omit. In both, $f, \lambda$ are  fixed and $\eta>0$ is the sparsity constant.
\begin{equation}\label{e:weighmod1} H_\lambda \coloneqq \left\{x\in \mathbb R: \, \mathrm{M}_X f(x) >\lambda \right\}, \qquad \widetilde{H_\lambda} \coloneqq\left\{x\in \mathbb R: \, \mathrm{M} \ind_{H_\lambda} (x)>2^{-4}\eta \right\}, \qquad G_\lambda \coloneqq \R \setminus H_\lambda{.} 
\end{equation}

%%%%%%%%%%%%%%%%%%%%%%%%%%%%%% LEMMA LEMMA LEMMA
\begin{lemma}
\label{l:weighmod1}{$\displaystyle w(\widetilde {H_\lambda}) \lesssim [w]_{A_1} w(H_\lambda)\lesssim [w]_{A_1}^2\int_{\mathbb R} \mathrm{Y}_X \left(\frac{|f(x)|}{\lambda}\right)   w(x) \, \d x. $}
\end{lemma}
%%%%%%%%%%%%%%%%%%%%%%%%%%%%%% LEMMA LEMMA LEMMA

%%%%%%%%%%%%%%%%%%%%%%%%%%%%%% LEMMA LEMMA LEMMA
\begin{lemma}\label{l:weighmod2} $\displaystyle \left\|\cic{1}_{G_\lambda}\mathrm{M}_X f \right\|_{L^p(w)} \lesssim [w]_{A_1}^{{\frac{1}{p}}} \lambda \left[p  B_p(X)^{{p}} \right]^{\frac1p} \left( \int_{\mathbb R} \mathrm{Y}_X \left(\frac{|f(x)|}{\lambda}\right)  w(x) \, \d x\right)^{\frac1p} $.
\end{lemma}
%%%%%%%%%%%%%%%%%%%%%%%%%%%%%% LEMMA LEMMA LEMMA

We now come to the actual proof of the proposition. By virtue of Lemma \ref{l:weighmod1}, it suffices to prove the estimate
\begin{equation} \label{e:weighmod3}
w( {F_\lambda}) \lesssim\mathcal Q\left([w]_{A_1}, [w]_{\mathrm{RH}_{\frac{q}{q-p(q-1)}}}\right)\int_{\mathbb R} \mathrm{Y}_X \left(\frac{|f|}{\lambda}\right) w , \quad F_\lambda\coloneqq \left\{ x\in   \widetilde {H_\lambda}^c: \, |Tf(x)|>\lambda\right \}.
\end{equation} 
The observation we will use in the sequel is the following. If $I$ is any interval with $I \cap {F_\lambda}\neq \varnothing $, then $|I \cap H_\lambda|< 2^{-4}\eta|I|$. Therefore, if $E_I$ is an $\eta$-major subset of such $I$, the set $  E_I \cap G_\lambda$ is a $\frac{\eta}{2}$-major subset of $I$ as well. At this point, invoking the sparse domination with collection $\mathcal S$ for the pair $(f,\theta w  \ind_{F_\lambda})$ for a suitably chosen unimodular function $\theta$ yields

\begin{equation}
\begin{split}
\lambda w(F_\lambda) & \leq \langle Tf,\theta w\ind_{F_\lambda}  \rangle  \lesssim  \sum_{I\in \mathcal S} |I| \langle  f \rangle_{X,I}\langle w \ind_{F_\lambda} \rangle_{q,I} 
\\ 
& = \sum_{\substack{I\in \mathcal S\\ I\cap F_\lambda \neq \varnothing}} |I| \langle  f \rangle_{X,I}\langle w \ind_{F_\lambda} \rangle_{q,I} \lesssim 
\sum_{I\in \mathcal S} |E_I \cap G_\lambda| \langle  f \rangle_{X,I}\langle w \ind_{F_\lambda} \rangle_{q,I}
\\ 
&\leq \int_{G_\lambda} \mathrm{M}_X f \mathrm{M}_{q}(w \ind_{F_\lambda}) = \int_{G_\lambda} \left[w^{\frac1p}\mathrm{M}_X f \right]\left[\mathrm{M}_{q}(w \ind_{F_\lambda}) w^{-\frac1p} \right]
\\ 
& \leq \left\|\cic{1}_{G_\lambda}\mathrm{M}_X f \right\|_{L^p(w)} \left\|\mathrm{M}_{q}(w \ind_{F_\lambda}) \right\|_{L^{p'}(w^{1-p'})} 
\\
& \lesssim \mathcal Q\left([w]_{A_1}, [w]_{\mathrm{RH}_{\frac{q}{q-p(q-1)}}}\right) \lambda    \left( \int_{\mathbb R} \mathrm{Y}_X \left(\frac{|f(x)|}{\lambda}\right)   w(x) \, \d x\right)^{\frac1p}  \left( w({F_\lambda})\right)^{\frac{1}{p'}}
\end{split}
\end{equation}
and this completes the proof of \eqref{e:weighmod3}. In the passage to the last line we have used Lemma \ref{l:weighmod2} and the well known weighted estimate 
\[
\left\|\mathrm{M}_{q} \right\|_{L^{p'}(w^{1-p'})} \lesssim {\widetilde{\mathcal Q}} \left([w^{1-p'}]_{A_{  \frac{p'}{q}} } \right) \sim \mathcal Q\left([w]_{A_p}, [w]_{\mathrm{RH}_{\frac{q}{q-p(q-1)}}}\right),	
\]
where $\widetilde{\mathcal Q}$ is a positive increasing function}; see \cite{IMS} for the last equivalence.
%%%%%%%%%%%%%%%%%%%%%%%%%%%%%% BIBLIO BIBLIO BIBLIO

\bibliography{SparseRDF}{}
\bibliographystyle{amsplain}

%%%%%%%%%%%%%%%%%%%%%%%%%%%%%% BIBLIO BIBLIO BIBLIO

\end{document}